\documentclass[a4paper]{amsart}

\usepackage{amssymb,pstricks,amscd}
\usepackage[all]{xy}
\usepackage[totalwidth=17cm,totalheight=24cm]{geometry}
\usepackage{graphicx}
\input psfig
\begin{document}

\newtheorem{cor}{Corollary}[section]
\newtheorem{theorem}[cor]{Theorem}
\newtheorem{prop}[cor]{Proposition}
\newtheorem{lemma}[cor]{Lemma}
\theoremstyle{definition}
\newtheorem{defi}[cor]{Definition}
\theoremstyle{remark}
\newtheorem{remark}[cor]{Remark}
\newtheorem{example}[cor]{Example}

\newcommand{\cD}{{\mathcal D}}
\newcommand{\FF}{{\mathcal F}}
\newcommand{\cH}{{\mathcal H}}
\newcommand{\cL}{{\mathcal L}}
\newcommand{\cM}{{\mathcal M}}
\newcommand{\cT}{{\mathcal T}}
\newcommand{\cML}{{\mathcal M\mathcal L}}
\newcommand{\cGH}{{\mathcal G\mathcal H}}
\newcommand{\C}{{\mathbb C}}
\newcommand{\N}{{\mathbb N}}
\newcommand{\R}{{\mathbb R}}
\newcommand{\Z}{{\mathbb Z}}
\newcommand{\Kt}{\tilde{K}}
\newcommand{\Mt}{\tilde{M}}
\newcommand{\dr}{{\partial}}
\newcommand{\tr}{\mbox{tr}}
\newcommand{\isom}{\mbox{Isom}}
\newcommand{\isomz}{\mbox{Isom}_{0,+}}
\newcommand{\vect}{\mbox{Vect}}
\newcommand{\kappab}{\overline{\kappa}}
\newcommand{\pib}{\overline{\pi}}
\newcommand{\Sigmab}{\overline{\Sigma}}
\newcommand{\gd}{\dot{g}}
\newcommand{\diff}{\mbox{Diff}}
\newcommand{\dev}{\mbox{dev}}
\newcommand{\devb}{\overline{\mbox{dev}}}
\newcommand{\devt}{\tilde{\mbox{dev}}}
\newcommand{\vol}{\mbox{Vol}}
\newcommand{\hess}{\mbox{Hess}}
\newcommand{\db}{\overline{\partial}}
\newcommand{\gammab}{\overline{\gamma}}
\newcommand{\Sigmat}{\tilde{\Sigma}}
\newcommand{\mut}{\tilde{\mu}}
\newcommand{\phit}{\tilde{\phi}}

\newcommand{\cunc}{{\mathcal C}^\infty_c}
\newcommand{\cun}{{\mathcal C}^\infty}
\newcommand{\dd}{d_D}
\newcommand{\dmin}{d_{\mathrm{min}}}
\newcommand{\dmax}{d_{\mathrm{max}}}
\newcommand{\Dom}{\mathrm{Dom}}
\newcommand{\dn}{d_\nabla}
\newcommand{\ded}{\delta_D}
\newcommand{\delmin}{\delta_{\mathrm{min}}}
\newcommand{\delmax}{\delta_{\mathrm{max}}}
\newcommand{\hmin}{H_{\mathrm{min}}}
\newcommand{\maxi}{\mathrm{max}}
\newcommand{\oL}{\overline{L}}
\newcommand{\oP}{{\overline{P}}}
\newcommand{\Ran}{\mathrm{Ran}}
\newcommand{\tgamma}{\tilde{\gamma}}
\newcommand{\cotan}{\mbox{cotan}}
\newcommand{\lambdat}{\tilde\lambda}
\newcommand{\St}{\tilde S}

\newcommand{\II}{I\hspace{-0.1cm}I}
\newcommand{\III}{I\hspace{-0.1cm}I\hspace{-0.1cm}I}
\newcommand{\HSt}{\tilde{\operatorname{HS}}}
\newcommand{\note}[1]{\marginpar{\tiny #1}}

\newcommand{\op}{\operatorname}

\newcommand{\AdS}{\operatorname{AdS}}
\newcommand{\uAdS}{\widetilde{\operatorname{AdS}}}
\newcommand{\dS}{\operatorname{dS}}
\newcommand{\HH}{\mathbb H}
\newcommand{\PP}{\mathbb P}
\newcommand{\RR}{\mathbb R}
\newcommand{\uRP}{\widetilde{\R\PP}^1}
\newcommand{\HS}{\operatorname{HS}}
\newcommand{\SO}{\operatorname{SO}}
\newcommand{\cF}{\operatorname{\mathcal{F}}}
\newcommand{\kD}{\mathfrak{D}}

\title[Collisions of particles]{Collisions of particles in locally AdS spacetimes}

\author[]{Thierry Barbot}
\address{Laboratoire d'analyse non lin\'eaire et g\'eom\'etrie\\
Universit\'e d'Avignon et des pays de Vaucluse\\
33, rue Louis Pasteur\\
F-84 018 AVIGNON
}
\email{thierry.barbot@univ-avignon.fr}
\author[]{Francesco Bonsante}
\address{Dipartimento di Matematica dell'Universit\`a degli Studi di Pavia,
via Ferrata 1, 27100 Pavia (ITALY)}
\email{francesco.bonsante@unipv.it}
\author[]{Jean-Marc Schlenker}
\address{Institut de Math\'ematiques de Toulouse, 
UMR CNRS 5219 \\
Universit\'e Paul Sabatier\\
31062 Toulouse Cedex 9}
\email{schlenker@math.univ-toulouse.fr}
\thanks{T. B. and F. B. were partially supported by CNRS, ANR GEODYCOS. J.-M. S. was
partially supported by the A.N.R. programs RepSurf,
ANR-06-BLAN-0311, GeomEinstein, 06-BLAN-0154, and ETTT, 2009-2013}

\date{\today}

\begin{abstract}
We investigate 3-dimensional globally hyperbolic AdS manifolds containing ``particles'',
i.e., cone singularities along a graph $\Gamma$. We impose physically relevant conditions on
the cone singularities, e.g. positivity of mass (angle less than $2\pi$ on time-like
singular segments). We construct examples of such manifolds, describe the cone singularities
that can arise and the way they can interact (the local geometry near the vertices of $\Gamma$).
The local geometry
near an ``interaction point'' (a vertex of the singular locus) has a simple
geometric description in
terms of polyhedra in the extension of hyperbolic 3-space by the de Sitter space.

We then concentrate on spaces containing only (interacting) massive particles.
To each such space we associate a graph and a finite family of pairs of hyperbolic
surfaces with cone singularities. We show that this data is sufficient to recover
the space locally (i.e., in the neighborhood of a fixed metric). This is a partial
extension of a result of Mess for non-singular globally hyperbolic AdS manifolds.
\end{abstract}

\maketitle

\tableofcontents

\section{Introduction}

\subsection{AdS geometry}

The 3-dimensional anti-de Sitter (AdS) space can be defined as a quadric in the
4-dimensional flat space of signature $(2,2)$:
$$ AdS_3 = \{ x\in \R^{2,2}~|~ \langle x,x\rangle =-1\}~. $$
It is a complete Lorentz manifold of constant curvature $-1$. It is however
not simply connected, its fundamental group is $\Z$. One is sometimes 
lead to consider its universal cover, $\widetilde{AdS_3}$, or its quotient
by the antipodal map, $AdS_{3,+}$.

\subsection{GHMC AdS manifolds}

A Lorentz $3$-manifold $M$ is AdS if it is locally modeled on $AdS_3$. 
Such a manifold is {\it globally hyperbolic maximal compact} (GHMC) if 
it contains a closed, space-like surface $S$, any inextendible time-like
curve in $M$ intersects $S$ exactly once, and $M$ is maximal under this
condition (any isometric embedding of $M$ in an AdS 3-manifold satisfying
the same conditions is actually an isometry). Those manifolds can in some
respects be considered as Lorentz analogs of quasifuchsian hyperbolic 
3-manifolds.

Let $S$ be a closed surface of genus at least $2$. A well-known theorem 
of Bers \cite{bers} asserts that the space of quasifuchsian hyperbolic metrics on 
$S\times \R$ (considered up to isotopy) is in one-to-one correspondence 
with $\cT_S\times \cT_S$, where $\cT_S$ is the Teichm\"uller space of $S$.

G. Mess \cite{mess,mess-notes} discovered a remarkable analog of this 
theorem for globally hyperbolic maximal anti-de Sitter manifolds:
the space of GHM AdS metrics on $S\times \R$ is also parameterized by 
$\cT_S\times \cT_S$. Both the Bers and the Mess results can be described 
as describing ``stereographic pictures'': the full structure of a 3-dimensional
constant curvature manifold is encoded in two hyperbolic metrics.

\subsection{AdS manifolds and particles}

3-dimensional AdS manifolds were first studied as a lower-dimensional toy model
of gravity: they are solutions of Einstein's equation, with negative cosmological constant
but without matter. A standard way to add physical realism to this model is to
consider in those AdS manifolds some point particles, modeled by cone singularities
along time-like lines (see e.g. \cite{thooft1,thooft2}). 

Here we will call ``massive particle'' such a cone singularity, of
angle less than $2\pi$, along a time-like line. Globally hyperbolic AdS spaces
with such particles were considered in \cite{cone}, when the cone angles are less than
$\pi$. It was shown that an extension of the Mess theorem exists in this setting, 
with elements of the Teichm\"uller space of $S$ replaced by hyperbolic metrics
with cone singularities, with cone angles equal to the angles at the 
``massive particles''. 

\subsection{Different particles}

The main goal here is to extend this study to more general ``particles'', where 
a ``particle'' is a cone singularity along a line which is not necessarily
time-like. The situations considered here are by definition much richer than those
in \cite{cone}, where the cone angles were restricted to be less than $\pi$. 
One key feature, which was absent from \cite{cone}, is that
those particles can ``interact'': the singular set of the manifolds under consideration
is a graph, with vertices where more than 2 particles ``interact''. 

The first point of course is to understand the possible particles. This study is done
in section 2. Massive particles (cone singularities along time-like lines) 
are considered, under the physically natural and
mathematically relevant condition that the cone angle is less than $2\pi$ (this
is physically interpreted as a positive mass condition). Among the other
particles are tachyons (cone singularities along space-like lines), gravitons
(which are along light-like lines) and BTZ black holes (space-like lines with a past but no
future) as well as extremal BTZ black holes (same but along a light-like line). 

\subsection{Interacting singularities}

Still in section 2 we consider some local examples of cone singularities
that can arise in this manner, and in particular as the vertices of
the singular locus, that is, the ``interaction points'' of the particles. 
It is pointed out that the link of an interaction point is what we
call an ``HS-surface'' as defined in section 2, an object already appearing elsewhere 
\cite{shu,cpt}
in relation to convex polyhedra in Lorentzian space-forms. 

In section 3 a systematic study of the HS-surfaces that can occur as
the link of a particle interaction is started.

\subsection{Particle interactions -- the Riemannian case}

There is an intimate but not completely obvious relationship between 
particle interactions -- vertices of the singular set of a manifold
with particles -- and convex polyhedra. 

It is easier to understand this first in the Riemannian case, that is, 
for a hyperbolic 3-manifold $M$ with cone singularities along a graph, 
when the angles at the singular lines are less than $2\pi$. Let
$v$ be a vertex of the singular set of $M$, and let $e_1, \cdots,
e_n$ be the edges of the singular set adjacent to $v$, that is,
the singular segments having $v$ as an endpoint. The {\it link}
of $M$ at $v$ is the space of geodesic segments in $M$ starting from
$v$, with the natural angle distance. With this distance it is
a spherical cone-manifold, with cone points corresponding to the
edges $e_i$, and with cone angle at each cone point equal to the
angle of $M$ at the corresponding singular segment, so that all
cone angles are less than $2\pi$.

According to a classical theorem of Alexandrov \cite{alex}, spherical
metrics with cone singularities on $S^2$, with cone angles less
than $2\pi$, are exactly the induced metrics on convex polyhedra
in $S^3$, and each metric is obtained on a unique polyhedron. 
Therefore, the possible links of interaction points in a hyperbolic
(or for that matter spherical or Euclidean) cone-manifold are exactly
the induced metrics on convex polyhedra in $S^3$. This provides a 
convenient way to understand spherical cone-metrics on $S^2$.

\subsection{Particles and polyhedra -- the Lorentzian case}

In the Lorentzian context considered here, a similar idea can be
followed, with some twists. The link of a point in $AdS^3$ (actually
in any Lorentz 3-manifold) is more interesting than in a Riemannian
3-manifolds: one part, corresponding to the time-like geodesic 
segments, is locally modeled on the hyperbolic 3-space, while
the part corresponding to the space-like geodesic segments is
locally modeled on the de Sitter space. Together the link of
a regular point in an AdS 3-manifold is modeled on the ``extended
hyperbolic space'', denoted here by $\HSt^3$, which has two
parts each isometric to the hyperbolic plane and one part isometric
to the de Sitter plane. The basic properties of the link can be 
found in section 2. A classification of the surfaces which can
occur as links, under natural conditions, can be found in section 3. 

Let now $M$ be an AdS manifold with particles, and let $v$ be
an interaction point, that is, a vertex of the singular locus. 
Its link -- defined as above -- is locally modeled on the
extended hyperbolic space, with cone singularities corresponding
to the singular segments adjacent to $v$. The description of 
the kind of metrics that can be obtained is richer than in the
Riemannian case, since for instance the cone singularities
can be on ``hyperbolic'' or ``de Sitter'' points, or on the
boundary ``at infinity'' between the two. The ``curvature''
conditions at the vertices (analog to the condition that the
cone angle is less than $2\pi$ in the Riemannian case) are
also more interesting, and depend on the conditions imposed
on the singular segments in $M$. Here for instance we
suppose that for time-like singular segments the cone angle
is less than $2\pi$.

In view of the Riemannian example and of its relation to the
Alexandrov theorem on polyhedra, it would be desirable to
have a similar result describing the induced metrics on 
polyhedra in the extended hyperbolic space $HS^3$. Again, the
statement itself is necessarily more complicated than in the
spherical case considered by Alexandrov since the vertices, 
edges and faces of the polyhedra can be of different type.
Nonetheless, an analog of Alexandrov's theorem seems to be
within reach, and large parts of its can be found in \cite{shu,cpt}.
This is described in section 4, and leads to the construction of
many local examples of particle interactions.

\subsection{The mass of an interaction point}

An interesting consequence of the relationship between interaction
points and polyhedra in $HS^3$ is that a new and apparently natural
condition occurs on the types of interactions that should be allowed:
each interaction should have ``positive mass'' in a sense which
extends the notion of positive mass for massive particles. 
This corresponds to a length condition found in \cite{cpt} as
a necessary length condition for the metrics induced on some types
of polyhedra in $HS^3$. Although this condition remained rather
mysterious in the context of \cite{cpt}, it has a straightforward
``physical'' interpretation when considered in light of its 
relation with particle interactions.

\subsection{Some global examples}

In section 5 we recall some important properties of Lorentz manifolds,
extended to the singular context considered here. Section 6 contains some global
examples of AdS spaces with interacting particles, obtained by gluing
constructions from polyhedra and by surgeries from simpler spaces.

It is proved later, in section 7, that those examples can be constructed
so as to be ``good space-times'' (according to Definition \ref{df:good}).
As a consequence, Theorem \ref{tm:homeo} can be applied, showing that the
(quite special) examples constructed in section 6 can be deformed to yield
many more generic cases.

\subsection{Stereographic picture of manifolds with colliding particles}

In section 7 we concentrate on manifolds with only massive particles
(cone singularities along time-like curves),
still under the condition that the angles are less than $2\pi$. 
We show how the ideas of Mess \cite{mess} extend in a richer way
in this setting. Recall that it was shown in \cite{mess} that a
globally hyperbolic AdS manifold $M$ gives rises to two hyperbolic
metrics on a surface (associated to its left and right representations)
and that $M$ is uniquely determined by those two metrics. 
This was extended in \cite{cone} (using some ideas from \cite{minsurf})
to globally hyperbolic AdS manifolds with massive particles of angle
less than $\pi$, a condition that prevents interactions. 

With angles less than $2\pi$, interactions can and actually do occur.
So to each space is associated a sequence (or more precisely 
a graph) of ``spacial slices'', each corresponding to a domain where
no interaction occurs. To each slice we associate a ``stereographic
picture'': a ``left'' and a ``right'' hyperbolic metric, both with
cone singularities of the same angles, which together are sufficient
to reconstruct the spacial slice. This is described in section 7.

Here too a new but apparently natural notion occurs, that
of a ``good'' spacial slice: one containing space-like surfaces which
are not too ``bent'', or more generally one on which there is a 
time-like, unit vector field which does not vary too much locally. 
There are examples of space-times containing a ``good'' spacial slice 
which stops being ``good'' after a particle interaction. A GHMC
AdS manifolds with particles is ``good'' if it is made of good
space-like slices, see Definition \ref{df:good}.

Adjacent spacial slices are ``related'' by a particle interaction. 
Still in section 7 we show that the left and right hyperbolic
metrics before and after the interaction in a good space-time are related by a 
simple surgery involving, for both the left and right metrics, 
the link of the interaction point. 

\subsection{The stereographic picture is a complete description (locally)}

So, to an AdS space with interacting massive particles, we associate
a topological data (a graph describing the spacial slices and the
way they are related by interactions) as well as a geometric data
(two hyperbolic metrics with cone singularities on a surface for each
spacial slice, along with simple surgeries for the interactions). 

In section 8 we show that this locally provides a complete description of
possible AdS spaces with interacting massive particles, i.e., 
given an AdS metric $g$ with interacting particles, a small neighborhood 
$g$ in the space of AdS metric with interacting particles is parameterized
by the admissible deformations of the ``stereographic pictures'' associated
to the spacial slices. This is Theorem \ref{tm:homeo}.

In other terms, locally at least, the topological data along with the
stereographic picture contain all the information necessary to reconstruct
an AdS manifold with interacting massive particles.

\section{Local examples}

\subsection{Singular lines}
\label{sub.singular_lines}

A \textit{singular $\AdS$-spacetime} (without interaction) is a manifold locally
modeled on $\AdS$ with cone singularities, i.e. with an atlas made of chards
taking value in one of the model spacetimes $\mathcal{P}_\theta$, $\mathcal{T}_m$,
$\mathcal{B}_m$, $\mathcal{G}_\pm$ or $\mathcal{E}$ described below, and such that the coordinate
changes are isometries on the regular part. Recall that we always assume the ambient manifold to be oriented
and time oriented.

\subsubsection{Massive particles}

Let $D$ be a domain in
$\uAdS$ bounded by two timelike totally geodesic half-planes $P_1$, $P_2$
sharing as common boundary a timelike geodesic $c$. The angle $\theta$ of $D$ is
the angle between the two geodesic rays $H \cap P_1$, $H \cap P_2$
issued from $c \cap H$, where $H$ is a totally geodesic hyperbolic plane
orthogonal to $c$. Glue $P_1$ to $P_2$ by the elliptic isometry of
$\uAdS$ fixing $c$ pointwise. The
resulting space, up to isometry,
only depends on $\theta$, and not on the choices of $c$
and of $D$ with angle $\theta$. We denote it $\mathcal{P}_\theta$.
The complement of $c$ in $\mathcal{P}_\theta$ is locally modeled on $\AdS$, while
$c$ corresponds to a cone singularity with angle $\theta$.

We can also consider a domain $D$ still bounded by two timelike planes,
but not embedded in $\uAdS$, wrapping around $c$, maybe several time, by an angle $\theta > 2\pi$.
Glueing as above, we obtain a singular spacetime $\mathcal{P}_\theta$ with angle $\theta > 2\pi$.

In these examples, the singular line is called \textit{massive particle.}
We define the mass as $m := 1 - \theta/2\pi$. Hence $\theta > 2\pi$
corresponds to particles with negative mass.

%Finally, $\uAdS$ is isometric to $\HH^2 \times \RR$ endowed
%with the metric $- \cosh(\rho)^2dt^2 + ds_{hyp}^2$ where
%$(\HH^2, ds_{hyp}^2)$ is the
%hyperbolic plane and $\rho: \HH^2 \to \RR_+$ is the hyperbolic distance
%to a base point in $\HH^2$. Consider polar coordinates $(r, \Theta)$, $r <1$
%for the the Poincar\'e disk model. The $\uAdS$-metric expresses as:

%\[ -\frac{dt^2}{1-r^2} + \frac{r^2d\Theta^2 + dr^2}{(1-r^2)^2} \]

%The same formula expresses the singular metric around a spinless
%particle $\{ r=0 \}$ but with $\Theta$ varying in $[0, \theta[$
%instead of $[0, 2\pi[$. Hence a linear change on the coordinate
%$\Theta$ gives the expression of this metric in polar coordinates:

%\[ -\frac{dt^2}{1-r^2} + \frac{(m-1)^2r^2d\Theta^2 + dr^2}{(1-r^2)^2} \]

\subsubsection{Tachyons}
\label{sub.tachyon}
Consider now a spacelike geodesic $c$ in $\uAdS$, and two timelike
totally geodesic planes $Q_1$, $Q_2$ containing $c$. Let $x$ be
a point in $c$ and consider the 2-plane $P$ in the tangent space $T_x\uAdS$
orthogonal to $c$. Let $l_1$, $l_2$ be the two isotropic lines in
$P$, and let $d_1$, $d_2$ be the intersection between $P$ and the planes
tangent to respectively $Q_1$, $Q_2$. We choose the indexation so that
the cross-ratio $[l_1 : d_1 : d_2 : l_2]$ is $\geq 1$. We define
the angle between $Q_1$ and $Q_2$
along $c$ as the logarithm of this cross-ratio. It is a positive real number,
not depending on the choice of $x$. We denote it by $m$.

$Q_1$ and $Q_2$ intersect each other
along infinitely many spacelike geodesics, always under the same angle.
In each of these planes,
there is an open domain $P_i$ bounded by $c$ and another component $c_-$ of
$Q_1 \cap Q_2$ in the past of $c$ and which does not intersect another component
of $Q_1 \cap Q_2$.
The union $c \cup c_- \cup P_1 \cup P_2$ disconnects
$\uAdS$. One of these components, denoted $W$, is contained in the future of $c_-$
and the past of $c$. Let $D$ be the other component, containing the future of
$c$ and the past of $c_{-}$. Consider the closure of $D$, and
glue $P_1$ to $P_2$ by an hyperbolic isometry of $\uAdS$ fixing every point
in $c$ and $c_-$. The resulting spacetime, denoted $\mathcal{T}_m$, contains two spacelike singular
lines, and is locally modeled on $\AdS$ on the complement of these lines.
These singular lines are called \textit{tachyon of mass $m$}.

We can also define tachyons with negative mass: cut $\uAdS$ along a timelike totally geodesic
annulus with boundary
$c \cup c_{-}$, and insert a domain isometric to $W$, with angle $m$. The resulting singular space
contains then two spacelike singular lines, that we call tachyons of mass $-m$.

There is an alternative description of tachyons: still start from a spacelike geodesic $c$ in $\uAdS$,
but now consider two spacelike half-planes $S_{1}$, $S_{2}$ with common boundary $c$, such that
$S_{2}$ lies above $S_{1}$, i.e. in the future of $S_{1}$. Then remove the intersection $V$ between
the past of $S_{2}$ and the future
of $S_{1}$, and glue $S_{1}$ to $S_{2}$ by a hyperbolic isometry fixing every point in $c$.
The resulting singular
spacetime, denoted $\mathcal{T}^0_m$, contains a singular spacelike line.
As we will see in \S~\ref{sec.link} this singular line is a tachyon of negative mass.
If instead of removing a wedge $V$ we insert it in the spacetime obtained by cutting
$\uAdS$ along a spacelike half-plane $S$,
we obtain a spacetime with a tachyon of positive mass.

\subsubsection{Future singularity of a black hole}
Consider the same data $(c, c_{-}, P_{1}, P_{2})$ used for the description of tachyons, but now remove $D$,
and glue the boundaries $P_{1}$, $P_{2}$ of $W$ by a hyperbolic element $\gamma_{0}$ fixing every point in $c$.
The resulting space is a manifold $\mathcal{B}_m$ containing two singular lines, that we abusively
still denote $c$ and $c_-$, and is locally $\AdS$ outside $c$, $c_-$.

Let $E$ be the open domain in $\uAdS$, intersection between the past of $c$ and the future of $c_{-}$. Observe that
$W$ is a fundamental domain for the action on $E$ of the group $\langle\gamma_{0}\rangle$ generated
by $\gamma_0$. In other words, the regular part of $\mathcal{B}_m$ is isometric
to the quotient $E/\langle\gamma_{0}\rangle$. This quotient is precisely a \textit{static BTZ black-hole}
as described in \cite{BTZ, barbtz1, barbtz2}. It is homeomorphic to the product of the annulus by the real
line. The singular spacetime $\mathcal{B}_m$ is obtained by adjoining to this BTZ black-hole
two singular lines: it follows that $\mathcal{B}_m$ is homeomorphic to the product of a 2-sphere with the real line.
Every point in the BTZ black-hole lies in the past of the singular line corresponding
to $c$ and in the future of the singular line corresponding to $c_-$. Therefore,
these singular lines are called respectively future singularity and past singularity.
More details will be given in \S~\ref{sub.bhfrancesco}.

\begin{remark}
In \S~\ref{sec.link}, we will see that, contrary to particles or tachyons, there is no dichotomy
positive mass/negative mass for black-holes.
\end{remark}

\begin{remark}
Truely speaking, the black-hole is a open domain in $\mathcal{B}_{m}$ (the region invisible from the "conformal boundary at infinity).
But it is useful for notational convenience to call the future singular line $c$ a black-hole singularity, and the past singular line
a white-hole singularity.
\end{remark}

%%If we glue $P_1$ to $P_2$ by a hyperbolic isometry inducing
%%a non-trivial translation along $c$, we obtain a singular spacetime homeomorphic NON!
%%to $\mathbb{S}^2 \times \R$, containing two periodic spacelike singularities,
%%and such that the regular part (the complement of the singular circles) is isometric
%%to a \textit{non-static BTZ black-hole} as described in \cite{BTZ}, with the ``inner regions''
%%removed (see also \cite[Remark 10.3]{barbtz2}).

\subsubsection{Gravitons and extreme black holes}
\label{sub.graviton}
Graviton is a limit case of tachyon: the definition is similar to that of tachyons,
but starts with the selection of a lightlike geodesic $c$ in $\uAdS$. Given such
a geodesic, we consider another lightlike geodesic $c_-$ in the past of $c$, and two timelike
totally geodesic annuli $P_1$, $P_2$, disjoint one from the other and with boundary $c \cup c_-$.
More precisely, consider pairs of spacelike geodesics $(c^n, c^n_-)$ as the one appearing in the description of tachyons, contained in timelike planes $Q^n_1$, $Q^n_2$, so that $c^n$ converge to the lightlike geodesic $c$. Then, $c^n_-$ converge to a lightlike geodesic $c_-$, whose future extremity
in the boundary of $\uAdS$ coincide with the past extremity of $c$. The timelike planes
$Q^n_1$, $Q^n_2$ converge to timelike planes $Q_1$, $Q_2$ containing $c$ and $c_1$. Then
$P_i$ is the annulus bounded in $Q_i$ by $c$ and $c_-$.
Glue the boundaries $P_1$ to $P_2$ of the component $D$ of $\uAdS \setminus (P_1 \cup P_2)$
containing the future of $c$ by an isometry of $\uAdS$ fixing every point in $c$ (and in $c_-$): the
resulting space is denoted $\mathcal{G}_+$. As we will see later,
it does not depend up to isometry on the choice of $P_1$, $P_2$: it follows from
the fact that, given two unipotent elements $u$, $u'$ of $\operatorname{PSL}(2, \R)$,
$u'$ is conjugate to $u$ or to $u^{-1}$.
Nevertheless, there is a reverse procedure, consisting in inserting a wedge $W$ instead of
removing, leading to a singular space $\mathcal{G}_-$.

Gravitons are singular lines in $\mathcal{G}_\pm$. As we will see in \S~\ref{sec.link}, there are two
non-isometric types of graviton: one corresponding to $c$ in $\mathcal{G}_{+}$ and to $c_{-}$ in
$\mathcal{G}_{-}$, called \textit{positive}, the other, called \textit{negative},
corresponding to $c$ in $\mathcal{G}_{-}$ and to $c_{-}$ in
$\mathcal{G}_{+}$.
One can also be consider gravitons as limit cases of massive particles.

There is an alternative way to construct gravitons: let $P$ be one of the two half-planes  bounded by $c$ inside the totally geodesic lightlike plane containing $c$. Cut $\uAdS$ along this half-plane, and glue back by an unipotent element fixing $c$. We will see in \ref{sec.link} that the singular line in the resulting spacetime is a graviton.

Finally, extreme black-holes $\mathcal{E}$ are similar to (static) BTZ black-holes - there are a limit case: they consist in
glueing the other component $W$ of $\uAdS \setminus (P_1 \cup P_2)$ along $P_1$, $P_2$ by an
unipotent element. Further comments and details are left to the reader (see also
\cite[\S~3.2, \S~10.3]{barbtz2}).

\subsubsection{Singular lines in singular $\AdS$ spacetimes}
\label{sub.lines}
The notions of particles, tachyons, future or past singularity of black-holes (extreme or not), gravitons,
extend to the general context of
singular $\AdS$ spacetimes without interaction: there are one-dimensional
submanifolds made of points admitting neighborhoods isometric to neighborhoods of singular points
in $\mathcal{P}_\theta$, $\mathcal{T}_m$, $\mathcal{B}_m$, $\mathcal{E}$ or $\mathcal{G}_\pm$.

\subsection{HS geometry}
In this {\S} we present another way to define singular spacetimes, maybe less visualizable,
but presenting several advantages: it clarifies the equivalence between the two
definitions of tachyons in \S~\ref{sub.tachyon}, and allows to define
interactions.

Given a point $p$ in $\uAdS$, let $L(p)$ be the link of $p$, i.e. the set of (non-parametrized)
oriented geodesic rays based at $p$. Since these rays are determined by their initial tangent vector at $p$
up to rescaling, $L(p)$ is naturally identified with the set of rays in $T_p\uAdS$.
Geometrically, $T_p\uAdS$ is a copy of Minkowski space $\R^{1,2}$. Denote by $\tilde{\operatorname{HS}}^2$
the set of geodesic rays in $\R^{1,2}$. It admits a natural decomposition in five subsets:

-- the domains $\HH^2_+$ and $\HH^2_-$ comprising respectively future oriented and past oriented
timelike rays,

-- the domain $\dS^2$ comprising spacelike rays,

-- the two circles $\partial\HH^2_+$ and $\partial\HH^2$, boundaries of $\HH^2_\pm$ in $\HS^2$.

The domains $\HH^2_\pm$ are notoriously Klein models of the hyperbolic plane,
and $\dS$ is the Klein model of de Sitter space of dimension $2$. The group $\SO_0(1,2)$,
i.e. the group of of time-orientation preserving and orientation preserving isometries of
$\R^{1,2}$, acts naturally (and projectively) on $\HS^2$, preserving this decomposition.

\begin{defi}
A HS-surface is a topological surface endowed with a $(\SO_0(1,2), \HS^2)$-structure,
i.e. an atlas with chards taking value in $\HS^2$ and coordinate changes made of restrictions of
elements of $\SO_0(1,2)$.
\end{defi}

The $\SO_0(1,2)$-invariant orientation on $\HS^2$ induces an orientation on
every HS-surface. Similarly, the $\dS^2$ regions admits a canonical time orientation.
Hence any HS-surface is oriented, and its de Sitter regions are time oriented.

Given a HS-surface $\Sigma$, and once fixed a point $p$ in $\uAdS$,
we can compose a locally $\AdS$ manifold $e(\Sigma)$, called the suspension
of $\Sigma$, defined as follows:

-- for any $v$ in $\HS^2 \approx L(p)$, let $r(v)$ be the geodesic ray issued from $p$
tangent to $v$. If $v$ lies in the closure of $\dS^2$, defines $e(v) := r(v)$; if
$v$ lies in $\HH^2_\pm$, let $e(v)$ be the portion of $r(v)$ containing $p$ and of
proper time $\pi$.

-- for any open subset $U$ in $\HS^2$, let $e(U)$ be the union of all $e(v)$ for $v$ in $U$.

Observe that $e(U) \setminus \{ p \}$ is an open domain in $\uAdS$,
and that $e(\HS^2)$ is the intersection $I^-(p^+) \cap I^+(p^-)$,
where $p^+$ is the first conjugate point in $\uAdS$ to $p$ in the future of $p$,
and where $p^-$ is the first conjugate point in $\uAdS$ to $p$ in the past of $p$.

The HS-surface $\Sigma$ can be understood as the disjoint union of open domains
$U_i$ in $\HS^2$, glued one to the other by coordinate change maps $g_{ij}$ given
by restrictions of elements of $\SO_0(1,2)$:

\[ g_{ij}: U_{ij} \subset U_j \to U_{ji} \subset U_i \]

But $\SO_0(1,2)$ can be considered
as the group of isometries of $\AdS$ fixing $p$. Hence every $g_{ij}$ induces
an identification between $e(U_{ij})$ and $e(U_{ji})$. Define $e(\Sigma)$
as the disjoint union of the $e(U_i)$, quotiented by the relation
identifying $x$ in $e(U_{ij})$ with $g_{ij}(x)$ in $e(U_{ji})$.
This quotient space contains a special point $\bar{p}$, represented in every $e(U_i)$
by $p$, and called the \textit{vertex}. The fact that $\Sigma$ is a surface implies that $e(\Sigma) \setminus \bar{p}$ is
a three-dimensional manifold, homeomorphic to $\Sigma \times \R$. The topological space
$e(\Sigma)$ itself is homeomorphic to the cone over $\Sigma$.
It is a manifold if and only if $\Sigma$ is homeomorphic to the 2-sphere.
But it is easy to see that every HS-structure on the 2-sphere is isomorphic
to $\HS^2$ itself. In order to obtain (singular) $\AdS$-manifolds, we need to
consider singular HS-surfaces. We could define this notion with the same method
we used to define singular lines in $\AdS$. But it is more convenient
to pursue further the present approach by considering links of points in the link $L(p)$.
For that purpose, we need some preliminary considerations on $\R\PP^1$-structures
on the circle.

\subsubsection{Real projective structures on the circle}
\label{sub.RP1circle}
Let $\R\PP^1$ be the real projective line, and let $\widetilde{\R\PP}^1$ be its universal
covering. Observe that $\widetilde{\R\PP}^1$ is homeomorphic to the real line.
Let $G$ be the group $\operatorname{PSL}(2, \R)$ of projective transformations
of $\R\PP^1$, and let $\tilde{G}$ be its universal covering: it is the group of
projective transformations of $\widetilde{\R\PP^1}$. We have an exact sequence:

\[ 0 \rightarrow \Z \rightarrow \tilde{G} \rightarrow G \rightarrow 0\]

Let $\delta$ be a generator of the center $\Z$. It acts on $\widetilde{\R\PP}^1$
as a translation on the real line. We orient $\widetilde{\R\PP}^1$, i.e. we fix a
total archimedian order on it, so that for every $x$ in $\widetilde{\R\PP}^1$ the
inequality $\delta x > x$ holds. The quotient of $\widetilde{\R\PP}^1$ by $\Z$ is projectively
isomorphic to $\R\PP^1$.

The elliptic-parabolic-hyperbolic classification of elements of $G$
is well-known. It induces a similar classification for elements in $\tilde{G}$, according
to the nature of their projection in $G$. Observe that non-trivial elliptic elements acts on
$\widetilde{\R\PP}^1$ as translations, i.e. freely and properly discontinuously.
Hence the quotient space of their action is naturally a real projective structure on the circle.
We call these quotient spaces \textit{elliptic circles.} Observe that it includes
the usual real projective structure on $\R\PP^1$.

Parabolic and hyperbolic elements can all be decomposed as a product $\tilde{g}=\delta^kg$
where $g$ has the same nature (parabolic or hyperbolic) than $\tilde{g}$, but admits fixed
points in $\widetilde{\R\PP}^1$. The integer $k \in \Z$ is uniquely defined. Observe
that if $k \neq 0$, the action of $\tilde{g}$ on $\widetilde{\R\PP}^1$ is free and properly
discontinuous. Hence the associated quotient space, which is naturally equipped with a real
projective structure, is homeomorphic to the circle. We call it \textit{parabolic} or
\textit{hyperbolic circle,} according to the nature of $g$, \textit{of degree $k$.} Reversing
$\tilde{g}$ if necessary, we can always assume $k \geq 1$.

Finally, let $g$ be a parabolic or hyperbolic element of $\tilde{G}$ fixing a point $x_0$
in $\widetilde{\R\PP}^1$. Let $x_1$ be another fixed point of $g$, with $x_1 > x_0$ (it
exists since $\delta x_0 > x_0$). More precisely, we take the unique such a fixed point
so that $g$ admits no fixed point between $x_0$ and $x_1$ (if $g$ is parabolic, $x_1 = \delta x_0$;
if $g$ is hyperbolic, $x_1$ is the unique $g$-fixed point in $] x_0, \delta x_0 [$).
Then the action of $g$ on $]x_0, x_1[$ is free and properly discontinuous, the quotient space
is a \textit{parabolic} or \textit{hyperbolic circle of degree $0$.}

These examples exhaust the list of real projective structures on the circle up to
real projective isomorphism. We briefly recall the proof: the developping map
$d: \R \to \widetilde{\R\PP}^1$ of a real projective structure on $\R/\Z$ is a local
homeomorphism from the real line into the real line, hence a homeomorphism onto its image $I$.
Let $\rho: \Z \to \tilde{G}$ be the holonomy morphism: being a homeomorphism, $d$ induces
a real projective isomorphism between the initial projective circle and $I/\rho(\Z)$.
In particular, $\rho(1)$ is non-trivial, preserves $I$, and acts freely and properly discontinuously on
$I$. An easy case-by-case study leads to a proof of our claim.

It follows that every cyclic subgroup of $\tilde{G}$ is the holonomy group of a real projective
circle, and that two such real projective circles are projectively isomorphic if and only if
their holonomy groups are conjugate one to the other. But some subtlety appears when one consider orientations:
usually, by real projective structure we mean a $(\operatorname{PGL}(2,\R), \R\PP^1)$-structure,
ie coordinate changes might reverse the orientation. In particular, two such structures are isomorphic
if there is a real projective transformation conjugating the holonomy groups, even if this transformation
reverses the orientation. But here, by $\R\PP^1$-circle we mean
a $(G, \R\PP^1)$-structure on the circle. In particular, it admits a canonical orientation: the one whose lifting
to $\R$ is such that the developping map is orientation preserving. To be a $\R\PP^1$-isomorphism, a real projective conjugacy needs to preserve the orientation.

Let $L$ be a $\R\PP^1$-circle. The canonical orientation above allow us to distinguish a generator $\gamma_0$ of its fundamental group:
the one for which $\rho(\gamma) = \tilde{g} =  \delta^kg$ satisfies $\tilde{g}x > x$ for every element $x$ in the image of the developping map. It follows that the degree $k$ cannot be negative. Moreover:

\textbf{The elliptic case:} Elliptic $\R\PP^1$-circles are uniquely parametrized by a positive real number (the angle).

\textbf{The case $k \geq 1$:} Non-elliptic $\R\PP^1$-circles of degree $k \geq 1$ are uniquely parametrized by
the pair $(k, [g])$, where $[g]$ is a conjugacy class in $G$. Hyperbolic conjugacy classes are uniquely parametrized by
a positive real number: the absolute value of their trace. There are exactly two parabolic conjugacy classes: the \textit{positive parabolic class,} comprising the parabolic elements $g$ such that $gx \leq x$ for every $x$ in $\widetilde{\R\PP}^1$, and the
\textit{negative parabolic class,} comprising the parabolic elements $g$ such that $gx \geq x$ for every $x$ in $\widetilde{\R\PP}^1$ ({\S}~\ref{sec.link} will
justify this terminology).

\textbf{The case $k = 0$:} In this case, $L$ is isomorphic to the quotient by $g$ of a segment $]x_0, x_1[$
admitting as extremities two successive fixed points of $g$. Since we must have $gx > x$ for every $x$ in this segment,
$g$ cannot belong to the positive parabolic class: \textit{Every parabolic $\R\PP^1$-circle of degree $0$ is
negative.} Concerning the hyperbolic $\R\PP^1$-circles, the conclusion is the same than in the case $k \geq 1$:
they are uniquely parametrized by a positive real number. Indeed, given an hyperbolic element $g$ in
$\tilde{G}$, any $\R\PP^1$-circle of degree $0$ with holonomy $g$ is a quotient of a segment $]x_0, x_1[$ where
the left extremity $x_0$ is a repelling fixed point of $g$, and the right extremity an attractive fixed point.

\subsubsection{Singular HS-surfaces}
For every $p$ in $\HS^2$, let $\Gamma_p$ be the stabilizer in $\SO_0(1,2)\approx \op{PSL}(2, \R)$ of $p$, and let $L_p$
be its link, i.e. the space of oriented half-projective lines starting from $p$. Since $\HS^2$ is oriented and
admits an underlying projective structure, $L_p$ admits a natural
$\R\PP^1$-structure, and thus any $(\Gamma_p, L_p)$-structure on the circle admits a natural underlying $\R\PP^1$-structure.
We define HS-singularities from $(\Gamma_p, L_p)$-circles:
given $p$ and such a $(\Gamma_p, L_p)$-circle $L$, we can construct a singular HS-surface $\mathfrak{e}(L)$: for every
element $v$ in the link of $p$, define $\mathfrak{e}(v)$ as the closed segment $[-p, p]$ contained
in the projective ray defined by $v$, where $-p$ is the antipodal point of $p$ in $\HS^2$,
and then operate as we did for defining the AdS space $e(\Sigma)$ associated to
a HS-surface. The resulting space $\mathfrak{e}(L)$ is topologically a sphere, locally modeled on
$\HS^2$ in the complement of two singular points corresponding to $p$ and $-p$.

There are several types of singularity, mutually non isomorphic:

\begin{itemize}

\item \textit{Elliptic singularities:} they correspond to the case where $p$ lies
in $\HH^\pm$. Then, $\Gamma_p$ is a 1-parameter elliptic subgroup of $G$, and
$L$ is an elliptic $\R\PP^1$-circle. We insist on the fact that we restrict to
orientation and time orientation preserving elements of $\op{O}(1,2)$, hence
we must distinguish past elliptic singularities ($p \in \HH^-$) from future
elliptic singularities ($p \in \HH^+$).

\item \textit{Parabolic singularities:} it is the case where $p$ lies in $\partial\HH^\pm$.
The stabilizer $\Gamma_p$ is parabolic, and the link $L$ is a parabolic $\R\PP^1$-circle. We still have to
distinguish between past and future parabolic singularities.

\item \textit{Hyperbolic singularities:} when $p$ lies in $\dS^2$, $\Gamma_p$
is made of hyperbolic elements, and $L$ is a hyperbolic $\R\PP^1$-circle.
\end{itemize}

\begin{defi}
A singular HS-surface $\Sigma$ is an oriented surface containing a discrete subset $\mathcal{S}$
such that $\Sigma \setminus \mathcal{S}$ is a regular HS-surface, and such that
every $p$ in $\mathcal{S}$ admits a neighborhood HS-isomorphic to the neighborhood
of a singularity $\mathfrak{e}(L)$ constructed above.
\end{defi}

Ther underlying $\R\PP^1$-structure almost define the $(\Gamma_p, L_p)$-structure
of links, but not totally. It is due to the fact that
half-real projective lines through $p$ may have several different type: near $p$,
if $p$ lies on $\partial\HH^2_\pm$, they can be
timelike or lightlike rays, and if $p$ lies in $\dS^2$, they can be also spacelike rays.

\begin{defi}
Let $L$ be a $(\Gamma_p, L_p)$-circle. We denote by $i^+(L)$ (resp. $i^-(L)$) the open subset of
$L$ comprising future oriented (resp. past oriented) timelike rays.
\end{defi}

In the case
where $p$ lies on $\partial\HH^2_+$ (resp. $\partial\HH^2_-$) future oriented timelike rays are contained, near $p$,
in $\HH^2_+$ (resp. $\dS^2$) whereas past oriented rays are contained near $p$ in $\dS^2$ (resp. $\HH^2_-$).

We invite the reader to convince himself that
the $\R\PP^1$-structure and the additional data of future oriented and past oriented arcs
determine the $(\Gamma_p, L_p)$-structure on the link,
hence the HS-singular point up to HS-isomorphism.

For hyperbolic singularities of degree $0$ the $(\Gamma_p, \R\PP^1)$-structure has a holonomy
group generated by a hyperbolic element $g$, and is a quotient
by $\langle g \rangle$ of one component $I$ of the complement in $\R\PP^1$ of
the set of $g$-fixed points.
Either $I$ comprises only future-oriented timelike rays,
or only past-oriented timelike rays, or only spacelike rays. In the two former cases,
we say that the singularity is a \textit{timelike hyperbolic singularity}; we furthermore
distinguish between \textit{future hyperbolic singularity} (i.e. the case $L = i^+(L)$) and  \textit{past hyperbolic singularity}
($L = i^-(L)$).
In the latter case $i^+(L) = i^-(L) = \emptyset$, the singularity is a \textit{spacelike hyperbolic singularity.}

Similarly, a parabolic singularity of degree $0$ is the quotient of an interval
$I=]x_{0}, \delta x_{0}[$.
Either every timelike element in $I$ is future oriented, or every timelike element of $I$
are past oriented. If the singularity itself corresponds to a point $p$ in $\partial\HH^2_+$,
and if all elements of $I$ are future oriented, then $\mathfrak{e}(L)$ contains a neighborhood
of the singularity made of future timelike elements; the complement of the singular point in this neighborhood
is an annulus locally modelled on the quotient of $\HH^2_+$ by a parabolic isometry, i.e., a cusp.
In that case, we say that the singularity is a \textit{future cuspidal parabolic singularity.} Still in
the case where $p$ lies in $\partial\HH^2_+$, but where elements of $I$ are past oriented,
the complement of the singular point in sufficiently small neighborhoods is locally modelled on $\dS^2$:
such singular points are \textit{future extreme parabolic singularities.} A similar situation holds
in the case where $p$ lies in $\partial\HH^2_-$: we get \textit{past cuspidal} and \textit{past extreme parabolic
singularities.}

The situation is slightly more delicate for $(\Gamma_p,L_p)$-hyperbolic circles of degree $k \geq 1$, i.e.
$(\Gamma_p, L_p)$-circles for which the
underlying $\R\PP^1$ structure is the quotient of
$\widetilde{\R\PP}^1$ by $\langle \delta^kg \rangle$ where $g$ is hyperbolic. Let $x_{0}$ be fixed
point of $g$ which is a left extremity of a future timelike component: this component is of the form $]x_{0}, x_{1}[$
where $x_{1}$ is another $g$-fixed point. All the other $g$-fixed points are the $x_{2i}=\delta^{i}x_{0}$
and $x_{2i+1}=\delta^{i}x_{1}$. Then future timelike components are the intervals $\delta^{2i}]x_{0}, x_{1}[$
and the past timelike components are $\delta^{2i+1}]x_{0}, x_{1}[$. It follows that the degree $k$ is an even integer.
In the previous {\S} we observed that there is only one $\R\PP^1$ hyperbolic circle of holonomy $\langle g \rangle$
up to $\R\PP^1$-isomorphism, but this remark does not extend to hyperbolic $(\Gamma_p, L_p)$-circles since
a real projective conjugacy between $g$ and $g^{-1}$, if preserving the orientation, must permute timelike and spacelike
components. Hence we must distinguish between positive hyperbolic $(\Gamma_p, L_p)$-circles and negative
ones: the former are characterized by the property that every $x_{2i}$ is an attracting $g$-fixed point,
whereas for the latter
the attracting $g$-fixed points are the $x_{2i+1}$.

Similarly, parabolic $(\Gamma_p, L_p)$-circles have even degree, and dichotomy past/future among parabolic
$(\Gamma_p, L_p)$-circles of degree 2 splits into two subcases: the positive case for which the parabolic
element $g$ satisfies $gx \leq x$
on $\uRP$, and the negative case satisfying the reverse inequality (this positive/negative dichotomy is
inherent of the structure of $\uRP$-circle data, cf. the end of {\S}~\ref{sub.RP1circle}).

%\begin{figure}[ht]
%\begin{center}
%\input singdgre.pstex_t
%\end{center}
%\caption{Singularities of degree $0$.}
%\label{fig.singdegre}
%\end{figure}

\subsection{Link of singular lines}
\label{sec.link}

It should be clear to the reader that the construction of AdS-manifolds $e(\Sigma)$
extends to singular HS-surfaces. Moreover, let $L$ be the link
of a HS-singularity, and let $\Sigma = \mathfrak{e}(L)$:

\begin{itemize}

\item  if $L$ is elliptic, the singular line in $e(\Sigma)$ is
a massive particle,

\item  if $L$ is a parabolic or hyperbolic
circle of degree $2$, the singular line in $e(\Sigma)$ is a graviton or a tachyon.

\item if $L$ is a future (resp. past) hyperbolic singularity of degree $0$, the singular line is the past
(resp. future) singularity of a white-hole (resp. black-hole),

\item if $L$ is a future extreme parabolic singularity, or a past cupidal parabolic singularity, the singular
line is the future singularity of an extreme black-hole.

\item if $L$ is a future cuspidal, or past extreme parabolic singularity, the singular line is the past
singularity of an extreme white-hole.

\end{itemize}

One way to convince oneself the validity of this claim is to observe that singular lines of singular
spacetimes defined in \S~\ref{sub.singular_lines} locally coincide with $e(\mathfrak{e}(L))$,
where $L$ is the link of the singularity, ie. the space of totally geodesic half-planes bounded by the singular line.
In every case, the holonomy is easy to identify (its projection in $G$ corresponds to the map glueing $P_{1}$ to
$P_{2}$ or $S_{1}$ to $S_{2}$). The claim follows from our previous study of characterization
of $(\Gamma_p, L_p)$-circles of given degree and holonomy.

In the list above extreme black-holes or white-holes appears as suspensions of either
cuspidal parabolic singularities, or extreme parabolic singularities. It is due to the fact
that for a parabolic $(\Gamma_p, L_p)$-circle of degree $0$, the singular point corresponding to $-p$
has not the same type than $p$: it is cuspidal if $p$ is extreme, and extreme if $p$ is cuspidal.

The hyperbolic case (tachyons) deserves more discussion.
We need to prove that the different constructions of tachyons in \S~\ref{sub.tachyon}
leads indeed to the isometric singular spacetimes, and moreover that, if
$L$ is a hyperbolic $(\Gamma_p, \R\PP^1)$-circle of degree 2, the mass
of the tachyon in $e(\mathfrak{e}(L))$ is positive (resp. negative) if $L$ is
of positive (resp. negative) type.

A way to determine the sign of the mass of tachyons is to consider AdS half-spaces with boundary
consisting of a timelike totally geodesic plane containing the tachyon. These half-spaces
correspond to intervals in the link $L$ projectively isomorphic to the affine real line
and with extremities lying in the timelike regions of $L$. Clearly, if the tachyon has positive mass,
one can cover the spacetime by two such half-spaces, since we removed a wedge in the past of
the spacelike geodesic $c$.
Similarly, such a covering is impossible if we inserted a wedge, i.e. if the mass is negative.
Complete affine real lines in $L$ are quotients of segments in $\uRP$ of the form $]x, \delta x[$.
If the associated tachyon has positive mass, there are two such segments $]x, \delta x[$,
$]y, \delta y[$, where $x$ and $y$ are timelike, $y$ lying in $]x, \delta x[$, and with projections
in $L$ covering the entire $L$. It means:

\[ x < y < \delta x <  g{\delta}^2 x < \delta y < \delta^2 x \]

Hence, $gx < x$: it follows as we claimed that $L$ is positive. The proof in the negative case
is similar. The case of gravitons as well: positive gravitons correspond to positive parabolic singularities
of degree 2.

\begin{remark}
From now, we qualify HS-singularities according to the nature of the associated
AdS-singular lines: an elliptic HS-singularity is a (massive) particle, a
parabolic of degree two is a graviton, positive or negative, etc...
\end{remark}

\begin{remark}
\label{rk.cuttachyon}
The same argument implies that there is a third way to build tachyons, similar
to the second construction of gravitons in \S~\ref{sub.graviton}:
given a spacelike geodesic $c$, let $P$ be one of the four lightlike totally
geodesic half-planes with boundary $c$. Cut along it, and glue back by a hyperbolic
isometry fixing every point in $c$.

More precisely, let $P_1$, $P_2$ be the two totally geodesic lightlike half-planes
appearing in the boundary of $D = \uAdS \setminus P$.
Up to time reversal, we can assume that $P$ is contained in the causal
future of $c$. Select the indexation so that
$P_1$ is contained in the boundary of the causal future of the geodesic $P_1 \cap P_2$.
Then $P_2$ is disjoint from this causal boundary.

Let $\gamma$ be the hyperbolic isometry preserving $P$ used to glue
$P_1$ onto $P_2$. As a transformation of $P$, $\gamma$ preserves every
lightlike geodesic ray issued from $c$. We leave to the reader the proof of the following fact:
\textit{the tachyon has positive mass if and only if for every $x$ in $P$ the lightlike
segment $[x, \gamma x]$ is future-oriented,\/} i.e. $\gamma$ sends every point in $P$ in
its own causal future.

The same criterion distinguishes between positive and negative gravitons:
\textit{a graviton is positive if and only if the map glueing the side
of $D$ contained in the causal future of the singular line to the other side
of $D$ sends every point in its causal future.}
\end{remark}

\begin{remark}
\label{rk.cuttachyon2}
As a corollary we get the following description HS-singularities corresponding to tachyons:
consider a small disk $U$ in $\dS^2$ and a point $x$ in $U$. Let $r$ be one lightlike geodesic ray
contained in $U$ issued from $x$, cut along it and glue back by a hyperbolic $\dS^2$-isometry $\gamma$
- observe that one cannot match one side on the other, but the resulting space
is still homeomorphic to the disk. The resulting HS-singularity is a tachyon. If $r$ is future
oriented, this tachyon has positive mass if and only if for every $y$ in $r$ the image $\gamma y$ is
in the future of $x$, see figure~\ref{fig.tachyon}. If $r$ is past oriented, the mass is positive
if and only if $\gamma y$ lies in the past of $y$ for evey $y$ in $r$.

\begin{figure}[ht]
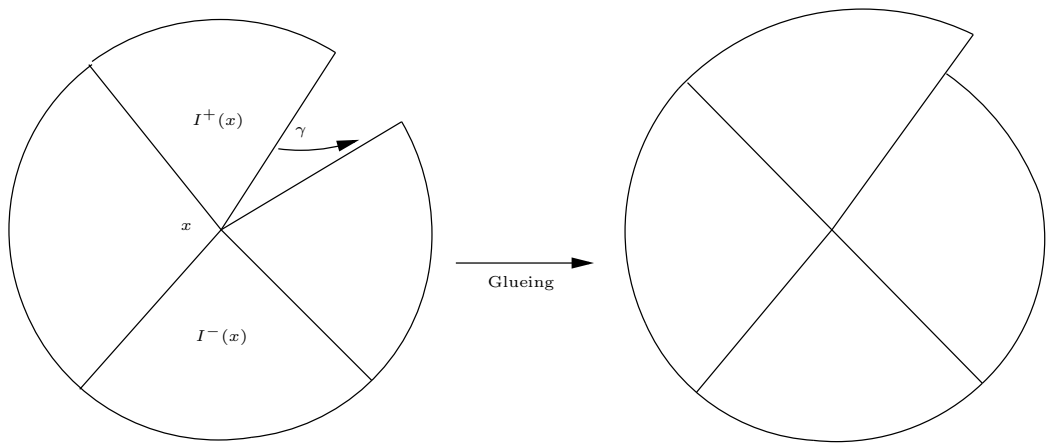

\begin{center}
\input Chirtachyon.pstex_t
\end{center}
\caption{Construction of a tachyon of positive mass.}
\label{fig.tachyon}
\end{figure}

A similar description holds for gravitons.

\end{remark}

\begin{remark}
\label{rk.cheminsing}
Let $[p_1, p_2]$ be an oriented arc in $\partial\HH^2_+$, and for every $x$ in $\HH^2$ consider the elliptic singularity
(with positive mass) obtained by removing the wedge comprising geodesic rays issued from $x$ and with extremity in $[p_1, p_2]$ and
glueing back by an elliptic isometry. Move $x$ until it reaches a point $x_\infty$ in $\partial\HH^2 \setminus [p_1, p_2]$.
It provides a continuous deformation of an elliptic singularity to a graviton, which can be continued further into $\dS^2$
by a continuous sequence of tachyons.
Observe that the gravitons (resp. tachyons) appearing in this continuous family are positive (resp. have positive mass).
\end{remark}

\subsubsection{Local future and past of singular points}
\label{sub.futpast}
In our list of singular lines in singular AdS spacetimes, we didn't consider singular lines in $e(\mathfrak{e}(L))$
where the $\R\PP^1$-circle underlying $L$ is of (even) degree $k \geq 4$ or of spacelike hyperbolic type. The reason for
this omission is that, as we will see now, they present pathological causal behavior.

In the singular AdS spacetimes we constructed, one can define timelike or causal curves, future or past
oriented: they are continuous paths, timelike or causal in the usual meaning outside the singularity, authorized
to go for a while along a massive particle or a graviton, and to cross a tachyon. Once introduced this notion
one can define the future $I^+(x)$ of a point $x$ as the set of final extremities of future oriented timelike
curves starting from $x$. Similarly, one defines the past $I^-(x)$, and the causal past/future $J^\pm(x)$.

Let $x$ be a point in a singular AdS-manifold $M$, and let $\Sigma_{x}$ be its link.
Let $\HH^+_{x}$ (resp. $\HH^-_{x}$) be the set of future (resp. past) timelike elements of the HS-surface $\Sigma_{x}$.
It is easy to see that the local future of $x$ in $e(\Sigma_{x})$, which is locally isometric to $M$, is the open domain
$e(\HH^+_{x}) \subset e(\Sigma_{x})$. Similarly, the past of $x$ in $e(\Sigma_{x})$ is $e(\HH^-_{x})$.
It follows that the causality relation in the neighborhood of point in a massive particle, a tachyon or a graviton presents
the same characteristic than in the neighborhood of a regular point: the local past and the local future are non-empty
connected open subset, bounded by lightlike geodesics, and contains no periodic causal curve. Furthermore, as
AdS-manifolds, the local future or past of a point lying on a tachyon or a graviton
are isometric to the local future or past of a regular point in $\uAdS$.

Points in the future singularity of a BTZ black-hole, extreme or not, have no future, and only one past
component. This past component is moreover isometric to the quotient of the past of a point in
$\uAdS$ by a hyperbolic (parabolic in the extreme case) isometry fixing the point. Hence, it is
homeomorphic to the product of an annulus by the real line.

If $L$ has degree $k \geq 4$, then the local future of a singular point in $e(\mathfrak{e}(L))$ admits $k/2$ components,
hence at least $2$, and the local past as well. This situation sounds very unphysical, we exclude it.

Points in spacelike hyperbolic singularities have no future, and no past. Besides, any neighborhood of such a point
contains closed timelike curves (CTC in short). Indeed, in that case, $\mathfrak{e}(L)$ is obtained by glueing the
two spacelike sides of a bigon entirely contained in the de Sitter region $\dS^2$ by some isometry $g$, and for every point $x$
in the past side $gx$ lies in the future of $x$: any timelike curve joining $x$ to $gx$ induces a CTC in $\mathfrak{e}(L)$.
But:

\begin{lemma}
\label{le.CTC}
Let $\Sigma$ be a singular HS-surface. Then the singular AdS-manifold $e(\Sigma)$ contains CTC if and only if
the de Sitter region of $\Sigma$ contains CTC. Moreover, if it is the case, every neighborhood of the vertex of $e(\Sigma)$
contains a CTC of arbitrarly small length.
\end{lemma}

\begin{proof}
Let $\bar{p}$ be the vertex of $e(\Sigma)$. For any real number $\epsilon$, let $f$ be the map associating to $v$ in
the de Sitter region of $\Sigma$
the point at distance $\epsilon$ to $p$ on the spacelike geodesic $r(v)$. Then the image of $f$ is a singular Lorentzian submanifold
locally isometric to the de Sitter space rescaled by a factor $\lambda(\epsilon)$. Moreover, $f$ is a conformal isometry: its
differential multiply by $\lambda(\epsilon)$ the norms of tangent vectors. The lemma follows by observing that
$\lambda(\epsilon)$ tends to $0$ with $\epsilon$.
\end{proof}

The definition of BTZ black-holes themselves is based on the prohibition of CTC (\cite{BTZ, barbtz1, barbtz2}).
Similarly, \textit{$e(\Sigma)$ contains closed causal curve (abbrev. CCC ) if and only if the de Sitter region of $\Sigma$
contains CCC.\/}
From now, we keep this convention and will restrict to HS-surfaces without CCC. Hence, we don't allow hyperbolic spacelike singularities.

\begin{defi}
A singular HS-surface is causal if it admits no singularity of degree $\geq 4$, no CTC and no CCC.
\end{defi}

\subsection{Link of an interaction}

\begin{defi}
An interaction is the suspension of a causal singular HS-surface homeomorphic to the 2-sphere with at least three singularities.
It is \textit{positive} if all particles and tachyons have positive mass, and all gravitons are positive.
\end{defi}

Let $\bar{p}$ be the vertex of an interaction; as we will see later the singular surface admits at least one timelike component,
let say a past component. The singularities in the past timelike region are massive particles, colliding at $\bar{p}$, and emitting
other massive particles (singularities in the future timelike region, if any), tachyons, gravitons, and/or creating black-holes (more precisely,
future singularities). Hence the classification of interaction types reduces to the classification of admissible HS-surfaces. It is the matter
of \S~\ref{sec.classificationHS}.

A singular AdS manifold with interaction is a manifold locally modeled on open subsets of interactions and open subsets of
singular AdS manifolds without interactions.

%\subsection{More on black holes}
% Horizon?

\section{Classification of positive causal HS-surfaces}
\label{sec.classificationHS}
In all this {\S} $\Sigma$ denotes a closed (compact without boundary) positive causal HS-surface. It decomposes
in three regions:

\begin{itemize}
\item \textit{Photons:} a photon is a point corresponding in every
HS-chard to points in $\partial\HH^2_\pm$. Observe that a photon might be singular,
i.e. corresponds to a graviton, or to the singularity of an extreme BTZ black-hole, i.e.
a past or future parabolic singularity. The set of photons, denoted
$\mathcal{P}(\Sigma)$, or simply $\mathcal{P}$ in the non-ambiguous situations,
is the disjoint union of a finite number of extreme black-holes singularities and of a compact embedded
one dimensional manifold, i.e. a finite union of circles.

\item \textit{Hyperbolic regions:} They are the connected components
of  the open subset $\HH^2(\Sigma)$ of $\Sigma$
corresponding to the timelike regions $\HH^2_\pm$ of $\HS^2$.
They are naturally hyperbolic surfaces with cone singularities. There
are two types of hyperbolic regions: the future and the past ones.
The boundary of every hyperbolic region is a finite union of circles of photons
and of cuspidal (parabolic) singularities.

\item \textit{De Sitter regions:} They are the connected components
of  the open subset $\dS^2(\Sigma)$ of $\Sigma$
corresponding to the timelike regions $\dS^2_\pm$ of $\HS^2$.
Alternatively, they are the connected components of $\Sigma \setminus \mathcal{P}$
that are not hyperbolic regions. Every de Sitter region is a singular dS surface,
whose closure is compact and with boundary made of circles of photons and of
a finite number of extreme parabolic singularities.

\end{itemize}

\subsection{Photons}
Let $C$ be a circle of photons. It admits two natural $\R\PP^1$-structures, which may not coincide, if
$C$ contains gravitons.

Consider a closed annulus $A$ in $\Sigma$ containing $C$ so that all HS-singularities in
$A$ lie in $C$.
Consider first the hyperbolic side, i.e. the component $A_{H}$ of $A \setminus C$ comprising timelike elements.
Reducing $A$ if necessary we can assume that $A_{H}$ is contained in a hyperbolic region. Then every path
starting from a point in $C$ has infinite length in $A_{H}$, and inversely every complete geodesic ray in $A_{H}$
accumulates on an unique point in $C$. In other words, $C$ is the conformal boundary at $\infty$ of $A_{H}$. Since
the conformal boundary of $\HH^2$ is naturally $\R\PP^1$ and that hyperbolic isometries extend as real projective transformations,
$C$ inherits, as conformal boundary of $A_{H}$, a $\R\PP^1$-structure that we call \textit{$\R\PP^1$-structure on $C$ from the hyperbolic side.}

Consider now the component $A_{S}$ in the de Sitter region adjacent to $C$. It is
is foliated by the lightlike lines. Actually, there are two such foliations (for more details, see
\ref{sub.dSclass} below).
An adequate selection of this annulus ensures that the leaf space of each of these foliations
is homeomorphic to the circle - actually, there is a natural identification
between this leaf space and $C$: the map associating to a leaf its extremity.
These foliations are transversely projective: hence they induce a $\R\PP^1$-structure
on $C$. This structure is the same for both foliations, we call it \textit{$\R\PP^1$-structure on $C$ from the de Sitter side.}

In order to sustain this
claim, we refer \cite[{\S} 6]{mess}: first observe that $C$ can be slightly pushed inside $A_{H}$
onto a spacelike simple closed curve (take a loop around $C$ following alternatively past oriented lightlike
segments in leaves of one of the foliations, and future oriented segments in the other foliation; and smooth it).
Then apply \cite[Proposition 17]{mess}.

If $C$ contains no graviton, $\R\PP^1$-structures from the hyperbolic and de Sitter sides coincide.
But it is not necessarily true if $C$ contains gravitons. Actually, the holonomy from one side is obtained
by composing the holonomy from the other side by parabolic elements, one for each graviton in $C$.
Observe that in general even the degrees may not coincide.

\subsection{Hyperbolic regions}

Every hyperbolic region is a complete hyperbolic surface with cone singularities (corresponding to massive particles)
and cusps (corresponding to cuspidal parabolic singularities) and of finite type, i.e. homeomorphic
to a compact surface without boundary with a finite set of points removed.

\begin{prop}
\label{pro.hypdegree0}
Let $C$ be a circle of photons in $\Sigma$, and $H$ the hyperbolic region adjacent to $C$.
Let $\bar{H}$ be the open domain in $\Sigma$ comprising $H$ and all cuspidal singularities
contained in the closure of $H$. Assume that $\bar{H}$ is not homeomorphic to the disk.
Then, as a $\R\PP^1$-circle defined by the hyperbolic side, the circle $C$ is hyperbolic of degree 0.
\end{prop}

\begin{proof}
The proposition will be proved if we find
an annulus in $H$ bounded by $C$ and a simple closed geodesic in $H$, and
containing no singularity.

Consider absolutely continuous simple loops in $H$ freely homotopic to $C$ in $H \cup C$.
Let $L$ be the length of one of them. There are two compact subsets $K \subset K' \subset \bar{H}$ such that
every loop of length $\leq 2L$ containing a point in the complement of $K'$ stay outside $K$ and
is homotopically trivial. It follows that every loop freely homotopic to $C$
of length $\leq L$ lies in $K'$:
by Ascoli and semi-continuity of the length, one of them has minimal length $l_0$
(we also use the fact that $C$ is not freely homotopic to a small closed loop around a cusp of $H$,
details are left to the reader).
It is obviously simple, and it contains no singular point, since every path containing a singularity
can be shortened. Hence it is a closed hyperbolic geodesic.

There could be several such closed simple geodesics of minimal length, but they are two-by-two disjoint,
and the annulus bounded by two such minimal closed geodesic must contain at leat one singularity since there
is no closed hyperbolic annulus bounded by geodesics. Hence, there is only a finite number of such
minimal geodesics, and for one of them, $c_0$, the annulus $A_0$ bounded by $C$ and $c_0$
contains no other minimal closed geodesic.

If $A_0$ contains no singularity,
the proposition is proved. If not, for every $r > 0$, let $A(r)$ be the
set of points in $A_0$ at distance $< r$ from $c_0$, and let $A'(r)$ be the complement of
$A(r)$ in $A_0$. For small value of $r$, $A(r)$ contains no singularity. Thus, it is isometric
to the similar annulus in the unique hyperbolic annulus containing a geodesic loop of length
$l_0$. This remarks holds as long as $A(r)$ is regular. Denote by $l(r)$ the length
of the boundary $c(r)$ of $A(r)$.

Let $R$ be the supremum of positive real numbers $r_0$ such that for every $r < r_0$
every essential loop in $A'(r)$ has length $\geq l(r)$. Since $A_0$ contains no closed
geodesic of length $\leq l_0$, this supremum is positive. On the other hand,
let $r_1$ be the distance between $c_0$ and the singularity $x_1$ in $A_0$ nearest to $c_0$.

We claim that $r_1 > R$. Indeed: near $x_1$ the surface is isometric
to a hyperbolic disk $D$ centered at $0$ with a wedge between two geodesic rays $l_1$, $l_2$ issued from $0$
of angle $2\theta$ removed. Let $\Delta$ be the geodesic ray issued from $0$ made of points at equal distance
from $l_1$ and from $l_2$.
Assume by contradiction $r_1 \leq R$. Then, $c(r_1)$ is a closed simple geodesic, containing $x_1$ and
minimizing the length of curves inside $A'(r_1)$. Singularities of cone angle $\theta \geq  \pi/2$ cannot be approached by closed loops
minimizing length, hence $\theta < \pi/2$. Moreover, we can assume without loss of generality that $c(r)$
near $x_1$ is the projection of a $C^1$-curve $\hat{c}$ in $D$ orthogonal to $\Delta$
at $0$, and such that the removed wedge between $l_1$, $l_2$, and the part of $D$ projecting into $A(r)$ are on opposite
sides  of this curve. For every $\epsilon > 0$, let $x_1^\epsilon$, $x^\epsilon_2$ be the points at distance
$\epsilon$ from $x$ in respectively $l_1$, $l_2$. Consider the geodesic $\Delta^\epsilon_i$
at equal distance from $x_i^\epsilon$ and $0$ ($i=1,2$): it is orthogonal to $l_i$, hence not
tangent to $\hat{c}$. It follows that, for $\epsilon$ small enough, $\hat{c}$ contains a point
$p_i$ closer to $x_i^\epsilon$ than to $0$. Hence, $c(r_1)$ can be shortened be replacing
the part between $p_1$ and $p_2$ by the union of the projections of the geodesics $[p_i, x_i^\epsilon]$.
This shorter curve is contained in $A'(r_1)$: contradiction.

Hence $R < r_1$. In particular, $R$ is finite. For arbitrarly small $\epsilon$, the annulus $A'(R+\epsilon)$ contains
an essential closed geodesic $c_\epsilon$ of minimal length $< l(R+\epsilon)$. Since it lies in $A'(R)$, this geodesic has
length $\geq l(R)$. It cannot be tangent to $c(R+\epsilon)$, hence it is disjoint from it. Moreover,
the annulus $A_\epsilon$ bounded by $c_\epsilon$ and $c(R+\epsilon)$ cannot be regular: indeed, if it was,
its union with $A(R+\epsilon)$ would be a regular hyperbolic annulus bounded by two closed geodesics.
Therefore, every $A_\epsilon$ contains a singularity. Up to a subsequence, the geodesics $c_\epsilon$ converges
when $\epsilon \to 0$ towards a closed geodesic $c_1$ of length $l(R)$, disjoint from $c(R)$, and
there is a singular point between $c_1$ and $c_R$.

Let $A_1$ be the annulus bounded by $C$ and $c_1$: every essential loop inside $A_1$ has length $\geq l(R)$
(since it lies in $A'(R)$). It contains strictly less singularities than $A_0$. If we restart the process
from this annulus, we obtain by induction an annulus bounded by $C$ and a closed geodesic inside $T$
with no singularity.
\end{proof}

\subsection{De Sitter regions}
\label{sub.dSclass}

Let $T$ be a de Sitter region of $\Sigma$. Future oriented isotropic directions defines two oriented line fields
on the regular part of $T$, defining two oriented foliations. Since tachyons are hyperbolic singularities of degree $2$,
these foliations extend continuously on tachyons (but not differentially) as regular oriented foliations.
Besides, in the neighborhood of every "black-hole" singularity $x$, the leaves of each of these foliations spiral around
$x$. They thus define two singular oriented foliations $\cF_{1}$, $\cF_{2}$, where the singularities are precisely
the "black-hole singularities", i.e. hyperbolic timelike ones, and have degree $+1$. By Poincar\'e-Hopf index formula
we immediatly get:

\begin{cor}
Every de Sitter region is homeomorphic to the  annulus, the disk or the sphere.
Moreover, it contains at most two timelike hyperbolic singularities: if it contains two singularities,
then it is homeomorphic to the 2-sphere, and if it contains exactly one, it is homeomorphic to the disk.
\end{cor}

Moreover, since by assumption $T$ contains no CCC, and by Poincar\'e-Bendixson Theorem:

\begin{cor}
\label{cor.Lclosed}
For every leaf $L$ of $\cF_{1}$ or  $\cF_{2}$, oriented by its time orientation, the $\alpha$-limit set
(resp. $\omega$-limit set) of $L$ is either empty or a past (resp. future) timelike hyperbolic singularity.
Moreover, if the $\alpha$-limit set (resp. $\omega$-limit set) is empty, the leaf accumulates in the past (resp. future)
direction to a past (resp. future) boundary component of $T$ that might be a point in a circle of photons,
or a extreme parabolic singularity.
\end{cor}

\begin{prop}
\label{pro.hyphyp}
Let $T$ be a de Sitter region adjacent to a hyperbolic region $H$ along a circle of photons $C$. If
the completion $\bar{H}$ of $H$ is
not homeomorphic to the disk, then
either $T$ is a disk containing a black-hole singularity, or the closure of $T$ in $\Sigma$ is the disjoint union
of $C$ and an extreme parabolic singularity.
\end{prop}

\begin{proof}
If $T$ is a disk, we are done. Hence we can assume
that $T$ is homeomorphic to the annulus. Reversing the time if necessary we also can assume
that $H$ is a past hyperbolic component. Let $C'$ be the other connected boundary component of $T$, i.e. its
future boundary. If $C'$ is an extreme parabolic singularity, the proposition is proved. Hence we are
reduced to the case where $C'$ is a circle of photons.

According to Corollary~\ref{cor.Lclosed} every leaf of $\cF_{1}$ or $\cF_{2}$ is a closed
line joining the two boundary components of $T$.
For every singularity $x$ in $H$, or every graviton in $C$, let $L_{x}$ be the future oriented half-leaf of $\cF_{1}$
emerging from $x$. Assume that $L_{x}$ does not contain any other singularity. Cut along
$L_{x}$: we obtain a $\dS^2$-surface $T^\ast$ admitting in its boundary two copies of $L_{x}$. Since $L_{x}$ accumulates until
$C'$ it develops in $\dS^2$ into a geodesic ray touching $\partial\HH^2$. In
particular, we can glue the two copies of $L_{x}$ in the boundary of $T^\ast$ by an isometry fixing their common
point $x$. For the appropriate choice of this glueing map, we obtain
a new $\dS^2$-spacetime where $x$ has been replaced by a regular point: we call this process,
well defined, \textit{regularization at $x$} (see figure~\ref{fig.regul}).

\begin{figure}[ht]
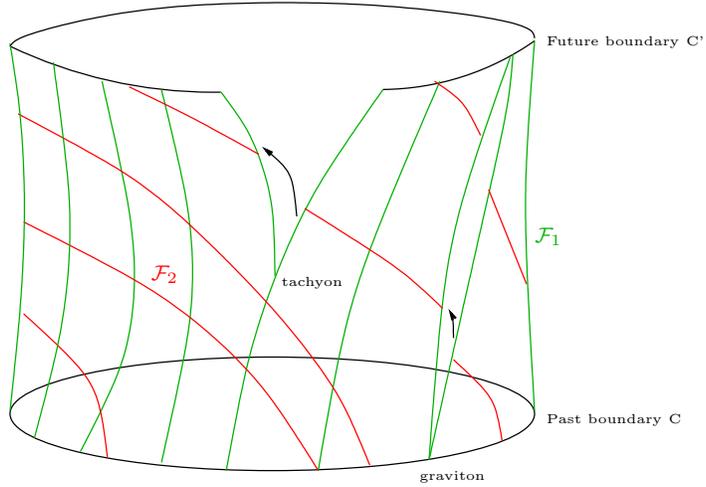

\begin{center}
\input regul.pstex_t
\end{center}
\caption{Regularization of a tachyon and a graviton.}
\label{fig.regul}
\end{figure}

After a finite number of regularizations, we obtain a regular $\dS^2$-spacetime
$T'$. Moreover, all these surgeries can actually be performed on $T \cup C \cup H$: the de Sitter annulus $A'$ can be glued
to $H \cup C$, giving rise to a HS-surface containing as circle of photons $C$ disconnecting the hyperbolic region $H$ from
the regular de Sitter region $T'$ (however, the other boundary component $C'$ have been modified and do not match anymore
the other hyperbolic region adjacent to $T$). Moreover, the circle of photons contains no graviton, hence its $\R\PP^1$-structure from
the de Sitter side coincide with the $\R\PP^1$-structure from the hyperbolic side. According to Proposition~\ref{pro.hypdegree0}
this structure is hyperbolic of degree $0$: it is the quotient of an interval $I$ of $\R\PP^1$ by an hyperbolic element $\gamma_{0}$,
with no fixed point inside $I$.

Denote by $\cF'_{1}$, $\cF'_{2}$ the isotropic foliations in $T'$.
Since we performed the surgery along half-leaves of $\cF_{1}$, leaves of $\cF'_{1}$ are still closed in $T'$. Moreover,
each of them accumulates at a unique point in $C$: the space of leaves of $\cF'_{1}$ is identified with $C$. Let
$\widetilde{T}'$ the universal covering of $T'$, and $\widetilde{\cF}'_{1}$ the lifting of $\cF_{1}$.
Recall that $\dS^2$ is naturally identified with $\R\PP^1 \times \R\PP^1 \setminus \kD$, where $\kD$ is the diagonal.
The developing map $\cD: \widetilde{T}' \to \R\PP^1 \times \R\PP^1 \setminus \kD$ maps every leaf of
$\widetilde{\cF}'_{1}$ into a fiber $\{ \ast \} \times \R\PP^1$. Besides, as affine lines, they are complete affine lines,
hence they still develop onto the entire geodesic $\{ \ast \} \times (\R\PP^1 \setminus \{ \ast \})$.
It follows that $\cD$ is a homeomorphism between $\widetilde{T}'$ and the open domain $W$ in $\R\PP^1 \times \R\PP^1 \setminus \kD$
comprising points with first component in the interval $I$, i.e. the region in $\dS^2$ bounded by two $\gamma_{0}$-invariant
isotropic geodesics. Hence $T'$ is isometric to the quotient of $W$ by $\gamma_{0}$, which is well understood (see figure~\ref{fig.W};
it has been more convenient to draw the lifting $W$ in the region in $\uRP \times \uRP$ between the graph of the identity
map and the translation $\delta$, region which is isomorphic to
the universal covering of $\R\PP^1 \times \R\PP^1 \setminus \kD$).

\begin{figure}[ht]
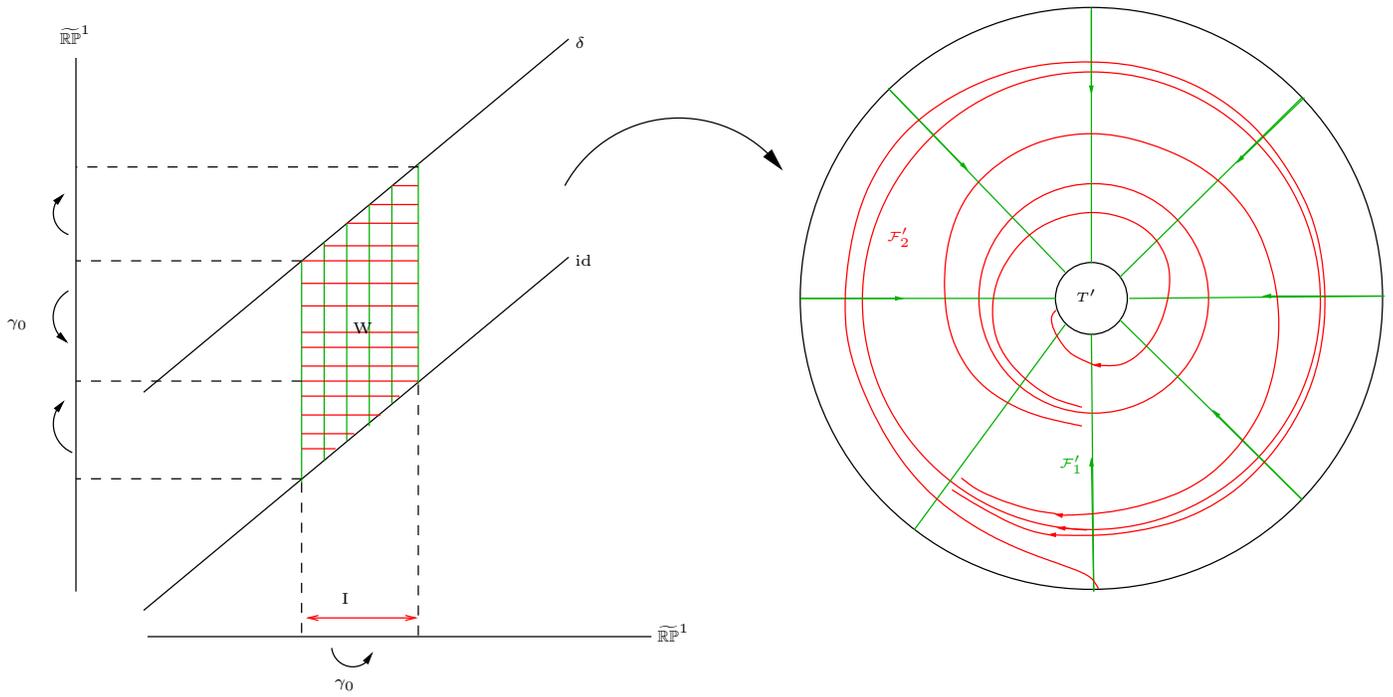

\begin{center}
\input W.pstex_t
\end{center}
\caption{The domain $W$ and its quotient $T'$.}
\label{fig.W}
\end{figure}

Hence the foliation $\cF_{2}$ admits two compact leaves.
These leaves are CCC, but it is not yet in contradiction with the fact that $\Sigma$ is causal,
since the regularization might create such CCC.

The regularization
procedure is invertible and $T$ is obtained from $T'$ by \textit{positive} surgeries along future oriented
half-leaves of $\cF_{1}$, i.e. obeying the rules described in Remark~\ref{rk.cuttachyon2}.
We need to be more precise: pick a leaf $L'_{1}$ of $\cF'_{1}$. It corresponds to a vertical line in $W$ depicted
in figure~\ref{fig.W}. We consider the first return $f'$ map from $L'_{1}$ to $L'_{1}$ along future oriented leaves of $\cF'_{2}$:
it is defined on an interval $]-\infty, x_{\infty}[$ of $L'_{1}$, where $-\infty$ corresponds to the end of $L'_{1}$ accumulating on
$C$. It admits two fixed points $x_{1} < x_{2} < x_{\infty}$, the former being attracting and the latter repelling.
Let $L_{1}$ be a leaf of $\cF_{1}$ corresponding, by the reverse surgery, to $L'_{1}$. We can assume
without loss of generality that $L_{1}$ contains no singularity. Let $f$ be the first return map from $L_{1}$ into itself
along future oriented leaves of $\cF_{2}$. There is a natural identification between $L_{1}$ and $L'_{1}$, and since all gravitons and tachyons in
$T \cup C$ are positive, \textit{the deviation of $f$ with respect to $f'$ is in the past direction,} i.e. for every $x$ in
$L_{1} \approx L'_{1}$ we have $f(x) \leq f'(x)$ (it includes the case where $x$ is not in the domain of definition of $f$,
in which case, by convention, $f(x) = \infty$). In particular, $f(x_{0}) \leq x_{0}$. It follows that the future oriented leaf
of $\cF_{2}$ is trapped below its portion between $x_{0}$, $f(x_{0})$. Since it is closed, it must accumulate on $C$.
But it is impossible since future oriented leaves near $C$ exit from $C$, intersect a spacelike loop, and cannot go back because
of orientation considerations. The proposition is proved.
\end{proof}

\begin{figure}[ht]
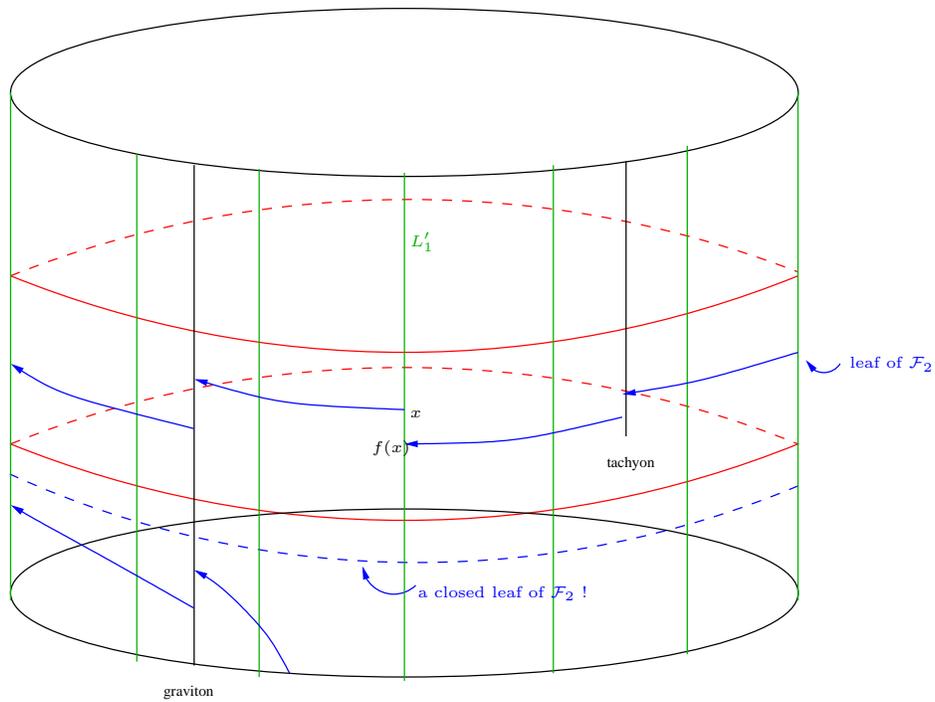

\begin{center}
\input return.pstex_t
\end{center}
\caption{First return maps. The identification maps along lines above tachyons and gravitons
compose the broken arcs in blue in leaves of $\mathcal{F}_2$. }
\label{fig.return}
\end{figure}

\begin{remark}
In Proposition~\ref{pro.hyphyp} the positive mass hypothesis is necessary. Indeed, consider a regular HS-surface
made of one annular past hyperbolic region connected to one annular future hyperbolic region by two de Sitter regions
isometric to the region $T' = W/\langle\gamma_{0}\rangle$ appearing in the proof of Proposition~\ref{pro.hyphyp}.
Pick up a photon $x$ in the past boundary of one of these de Sitter components $T$, and let $L$ be the leaf of $\cF_{1}$
accumulating in the past to $x$. Then $L$ accumulates in the future to a point $y$ in the future boundary component.
Cut along $L$, and glue back by a parabolic isometry fixing $x$ and $y$. The main argument in the proof above is that
if this surgery is performed in the positive way, so that $x$ and $y$ become positive tachyons, then the resulting spacetime
still admits two CCC, leaves of the foliation $\cF_{2}$. But if the surgery is performed in the negative way, with a sufficiently
big parabolic element, the closed leaves of $\cF_{2}$ in $T$ are destroyed, and every leaf of the new foliation $\cF_{2}$
in the new singular surface joins the two boundary components of the de Sitter region, which is therefore causal.
\end{remark}

\begin{theorem} \label{tm:thierry}
Let $\Sigma$ be a singular causal positive HS-surface, homeomorphic to the sphere. Then, it admits at most one past hyperbolic
component, and at most one future hyperbolic component. Moreover, we are in one of the following mutually exclusive situations:

\begin{enumerate}

\item \textit{Causally regular case: } There is a unique annular de Sitter component, connecting one past hyperbolic region
homeomorphic to the disk to a future hyperbolic region homeomorphic to the disk.

\item  \textit{Interaction of black-holes or white-holes: } There is no past or no future hyperbolic region, and every de Sitter region is a
either a disk containing a unique future BTZ singularity, or a disk with an extreme black-hole singularity removed.

\item \textit{Big-Bang and Big Crunch: } There is no de Sitter region, and only one timelike region, which is a
singular hyperbolic sphere - if the timelike region is a future one, the singularity is called a Big-Bang; if the timelike region is a
past one, the singularity is a Big Crunch.

\item \textit{Interaction of a white hole with a black-hole: } There is no hyperbolic region. The surface $\Sigma$ contains one past black-hole singularity
and a future black-hole singularity - these singularities may be extreme or not.

\end{enumerate}
\end{theorem}

\begin{proof}
If the future hyperbolic region and the past hyperbolic region is not empty, there must be a de Sitter annulus
connected one past hyperbolic component to a future hyperbolic component. By Proposition \ref{pro.hyphyp}
these hyperbolic components are disks: we are in the causally regular case.

If there is no future hyperbolic region, but one past hyperbolic region, and one Sitter region, every de Sitter component
cannot be annuli
connecting two hyperbolic regions: its closure is a closed disk. It follows that there is only one past hyperbolic component:
$\Sigma$ is an interaction of black-holes. Similarly, if there is a de Sitter region, a past hyperbolic region but no future hyperbolic region,
$\Sigma$ is an interaction of white holes.

The remaining situations are the cases where $\Sigma$ has no de Sitter region, or no hyperbolic region. The former
case corresponds obviously to the description (3) of Big-Bang or Big-Crunch , and the latter to the description (4) of
an interaction between a black-hole and a white-hole.
\end{proof}

\begin{remark}
It is easy to construct singular hyperbolic spheres, i.e. Big-Bang or Big-Crunch: take for example the double
of a hyperbolic triangle. The existence of interactions of white-holes with black-hole is slightly less obvious.
Consider the HS-surface $\Sigma_{m}$ associated to the black-hole $\mathcal{B}_{m}$. It can be described as follows:
take a point $p$ in $\dS^2$, let $d_{1}$, $d_{2}$ be the two projective circles in $\HS$ containing $p$, its opposite
$-p$, and tangent to $\partial\HH^2_{\pm}$. It decomposes $\HS^2$ in four regions. One of these components,
that we denote by $U$, contains the past hyperbolic region $\HH^2_{-}$. Then, $\Sigma_{m}$ is the quotient
of $U$ by the group generated by a hyperbolic isometry $\gamma_{0}$ fixing $p$, $-p$, $d_{1}$ and $d_{2}$.
Let $x_{1}$, $x_{2}$ be the points where $d_{1}$, $d_{2}$ are tangent to $\partial\HH^2_{-}$, and let $I_{1}$,
$I_{2}$ be the connected components of $\partial\HH^2_{-} \setminus \{ x_{1}, x_{2} \}$. We select the index
so that $I_{1}$ is the boundary of the de Sitter component $T_{1}$ of $U$ containing $x_{1}$. Now let $q$ be a point in
$T_{1}$ so that the past of $q$ in $T_{1}$ has a closure in $U$ containing a fundamental domain $J$ for the action of
$\gamma_{0}$ on $I_{1}$. Then there are two timelike geodesic rays starting from $q$ and accumulating at points
in $I_{1}$ which are extremities of a subintervall containing $J$. These rays project in $\Sigma_{m}$ onto two timelike
geodesic rays $l_{1}$ and $l_{2}$ starting from the projection $\bar{q}$ of $q$. These rays admit a first intersection point $\bar{q}'$ in the past of $\bar{q}$. Let $l'_{1}$, $l'_{2}$
be the subintervalls in respectively $l_{1}$, $l_{2}$ with extremities $\bar{q}$, $\bar{q}'$: their union is a circle disconnecting
the singular point $\bar{p}$ from the boundary of the de Sitter component. Remove the component adjacent to this boundary.
If $\bar{q}'$ is well-chosen, $l'_{1}$ and $l'_{2}$ have the same proper time. Then we can glue one to the other by a hyperbolic isometry. The resulting spacetime is as required an interaction between a black-hole corresponding to $\bar{p}$ with a white-hole corresponding
to $\bar{q}'$ - it contains also a tachyon of positive mass corresponding to $\bar{q}$.
\end{remark}

\section{From particle interactions to convex polyhedra}

This section describes a relationship between on 
interactions of particles in $3$-dimensional AdS manifolds, 
HS-structure on the sphere, and convex polyhedra in $\HS^3$, the natural
extension of the hyperbolic $3$-dimensional by the de Sitter space. 

Given a convex polyhedron in $\HS^3$, one can consider the
induced metric $g$ on its boundary, which is a HS-structure on $S^2$ with some cone
singularities, and then the cone over $(S^2,g)$, which is an
AdS metric with cone singularities at the vertex and along the 
lines corresponding to cone points of $g$. This is the metric
in a neighborhood of an interaction of particles. 

The converse also holds at least to some extend, under a technical 
hypothesis which
appears to be physically relevant. Let $M$ be an AdS manifold
with interacting particles, and let $x\in M$ be an interaction
point. The link of $x$ is homeomorphic to $S^2$, with a natural
HS-structure $g$ with cone singularities at the points corresponding
to the particles interacting at $x$. 

Under the hypothesis that
the interaction has ``positive mass'' -- a hypothesis which 
appears to make sense physically, as explained below -- 
this HS-structure $g$ should be realized
as the induced metric on the boundary of a unique convex polyhedron in 
$\HS^3$. This is proved here in most ``simple'' cases, using 
a previously known result on the induced metrics on the boundary of convex
polyhedra in $\HS^3$. 

Some hypothesis in the polyhedral result are precisely those
which are physically relevant in the context of interactions, 
for instance the condition of positive mass for massive particles
and for tachyons. Some other conditions in the polyhedral result
deal with interactions which are geometrically possible but 
more difficult to interpret, and it is not completely clear
what the ``physical'' meaning of those conditions is.

For technical reasons that will appear clearly below, we do not
consider in this section gravitons -- singularities along 
light-like lines -- and restrict our attention to massive
particles, tachyons, and black/white holes, as well as
big bangs/crunches.

\subsection{Convex polyhedra in $\HS^3$}

We recall here for completeness a (slightly incomplete) description of 
the induced metric on convex
polyhedra in $\HS^3$. The material here is from \cite{shu,cpt}.

\subsubsection{The space $\HS^3$}

The previous sections contain a description of the 2-dimensional
space $\HS^2$, which is a natural extension of the hyperbolic plane
$\HH^2$ by a quotient by $\Z/2\Z$ of the de Sitter plane $dS^2$. Its
double cover, $\HSt^2$, is simply connected. 

We consider here the same notion in dimension 3, the
object corresponding to $\HS^2$ 
is $\HS^3$, the natural extension of the hyperbolic space
$\HH^3$ by the quotient by $\Z/2\Z$ of the de Sitter space $dS^3$. 
Its double cover, $\HSt^3$, is simply connected, and is made of
two copies of $\HH^3$ and one of $dS^3$ (see \cite{shu,cpt} for more
details). 

An elementary but useful point is that $\HSt^3$ has a projective model
in the 3-sphere $S^3$. As a consequence, there is a well-defined notion
of polyhedron in $\HSt^3$, or in $\HS^3$. 

\subsubsection{Three kinds of polyhedra}

As already mentioned we do not consider here gravitons. Such particles
are cone singularities along a light-like line, they correspond
to cone singularities on the light cone of the link of the interaction
point. When this metric is the induced metric on the boundary of a convex polyhedron 
in $\HSt^3$, those cone singularities correspond to vertices on the
boundary at infinity of the hyperbolic part of $\HSt^3$. Such vertices
are not considered in the results of \cite{cpt} -- which will be
needed below -- and this explains the exclusion of gravitons from this
section. 

Note also that by ``polyhedra'' we mean here a polyhedron with at least
$3$ vertices. This allows for degenerate polyhedra which are reduced
to a triangle, but excludes segments, with only two vertices. Polyhedra
are then always contained in the complement of a plane in $\HS^3$, and
they can be lifted to a convex polyhedron in $\HSt^3$. It is slightly simpler to
consider polyhedra in $\HSt^3$, which is what we will do below.

Once vertices on the boundary of $\HH^3$ are excluded, there are three 
distinct types of polyhedra in $\HSt^3$, as shown in Figure \ref{fig:3poly}.

\begin{figure}[ht]
\begin{center}
\psfig{figure=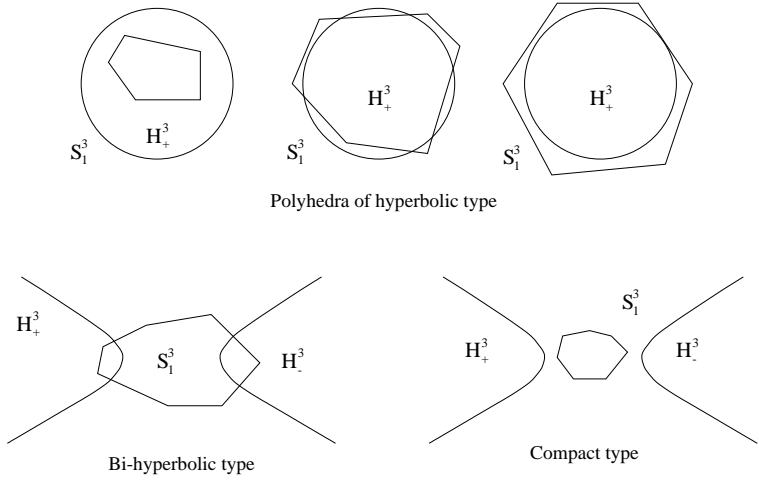,width=10cm}
\end{center}
\caption{Three types of polyhedra in $\HS^3$.}
\label{fig:3poly}
\end{figure}

The distinction is more easily understood by considering polyhedra
in $\HSt^3$. 

\begin{itemize}
\item The first family contains all polyhedra which intersect 
one of the hyperbolic spaces in $\HSt^3$ but not the other. 
Those polyhedra are of ``hyperbolic type'', they include compact, ideal or
hyperideal hyperbolic polyhedra, as well as some more exotic polyhedra
having more vertices in $dS^3$. The induced metrics on the boundaries of
those polyhedra are described in \cite{shu}, and each is obtained in exactly one
way (the result is recalled below).
\item The second family contains all 
polyhedra which intersect both hyperbolic spaces in $\HSt^3$, they are
of ``bi-hyperbolic type'', the induced metrics on their boundaries are described in
\cite{cpt} and each possible metric is obtained in exactly on way.
\item The third family contains all polyhedra which are contained in 
$dS^3$, but are not the dual of a hyperbolic polyhedron. The possible induced
metrics on their boundaries are described in \cite{cpt}, and it is shown that possible
metric is obtained on a unique polyhedron in $dS^3$ under a mild
additional hypothesis.
\end{itemize}
 
We will see below that each type of polyhedron has an interpretation in
terms of an interaction between particles in a Lorentz space-form.

\subsubsection{Topological data}

To describe properly the induced metrics on the boundaries of polyhedra in 
$\HSt^3$ of the three types described above, it is necessary to
include, in addition to the metric itself, an additional data
corresponding to which edges are ``degenerate'' timelike faces,
and which edges or vertices are ``degenerate'' space-like faces. 

\begin{defi}
Let $P\subset \HSt^3$, we call $\Sigma$ the union of: 
\begin{itemize}
\item the space like faces of $P$ in $dS^3$,
\item the space-like edge $e$ of $P$ which bound two timelike or light-like
faces, such that, in any neighborhood of a point of $e$, some space-like
geodesic intersects both,
\item the vertices $v$ of $P$ such that all geodesics of $P$ (for the the
induced metric) starting from $v$ are timelike,
\end{itemize}
from which are removed:
\begin{itemize}
\item the space-like edges $e$ of $P$ bounding two space-like faces such that,
in any neighborhood of any point of $e$, some timelike geodesic intersects both,
\item the vertices $v$ of $P$ such that the link of $P$ at $v$ is a polygonal
disk contained in the de Sitter part of $\HSt^2$.
\end{itemize}
We also define $\cT$ as the subset of points in the de Sitter part of $P$ which
are not in $\Sigma$, and $\cH$ the set of hyperbolic points in $P$, that is, 
the intersection of $P$ with the union of the two copies of $\HH^3$, 
considered as subsets of $\HSt^3$.
Given a polyhedron $P\subset \HSt^3$, the {\bf marked induced HS-structure} on the
boundary of $P$
is the induced HS-structure on the boundary $P$ along with $\Sigma$ ($\cT$ and $\cH$ are then
defined implicitly).
\end{defi}

The heuristic motivation for this definition is that $\Sigma$ is, in a natural
way, the union of all space-like faces of $P$, including those which are
reduced to either an edge or a vertex. This definition exhibits a kind of
topological stability, under deformations of convex polyhedra. For instance,
if a cube in $dS^3$ is deformed so that its upper face, which is space-like
throughout the deformation, is an edge in the limit, then that edge is
in $\Sigma$. $\cT$ is, in a similar ``natural'' way, the union of the
timelike or light-like faces of $P$ (in the de Sitter part of $\HS^3$),
including the ``faces'' which are reduced to an edge or to a vertex.

More specifically, the topology of $\Sigma$ and of $\cT$ takes a specific form
for each of the three types of polyhedra considered above. In all cases, $\cT$
is the disjoint union of a finite number of deformation retracts of annuli
(each might be a circle, or obtained by gluing disks and segments).
In addition:
\begin{itemize}
\item If $P$ is of hyperbolic type, there is no other topological constraint.
\item If $P$ is of bi-hyperbolic type, $\Sigma=\emptyset$, the hyperbolic
part of the metric has two connected components which are contractible, and $\cT$ is
a cylinder in which every timelike curve connects one of the hyperbolic
components to the other.
\item If $P$ is of compact type, the metric on the boundary of $P$ has no hyperbolic
component, $\Sigma$ has two connected components which are contractible,
and  $\cT$ is the deformation retract of an annulus 
in which every timelike curve connects one 
of the components of $\Sigma$ to the other, or a circle.
\end{itemize}

We will see below that each of those types can be interpreted in the
setting of collisions of particles, and each corresponds to a different
physical situation. 

\subsubsection{Lengths of geodesics}

The induced metrics on the boundaries of 
polyhedra in $\HS^3$ satisfy some metric conditions,
concerning the lengths of two kinds of space-like ``geodesics''. We first
introduce those two types of curves.

\begin{defi}
A {\bf $\Sigma$-geodesic} in $P$ is a polygonal curve contained in $\Sigma$,
geodesic in each face of $P$, such that, at each vertex $v$ and on each side,
either there is (in each neighborhood of $v$) an element of $\cT$, or 
the metric is concave. 
\end{defi}

The meaning of ``concave'' here should be clear, since it applies to the
sides of the curve on which there is no element of $\cT$, so that all
faces are spherical.

Suppose for instance that $P\subset dS^3$ is the dual of a hyperbolic
polyhedron, so that $\Sigma$ is the whole boundary of $P$. 
Then $\Sigma$-geodesics are polygonal
curves for which each side is concave, so they correspond to ``usual''
geodesics in the induced metric on the boundary of $P$. The $\Sigma$-geodesics have 
length bounded from below, a point appearing already, for the duals 
of hyperbolic polyhedra, in \cite{RH}.

\begin{lemma} \label{lm:length1}
  \begin{enumerate}
  \item Let $P\subset \HS^3$ be a convex polyhedron. Then all closed $\Sigma$-geodesics
on the boundary of $P$ have length $L\geq 2\pi$, with equality only when they bound a degenerate
domain in $\cT$.
\item If $P$ is of compact type, $\Sigma$-geodesic segments in both connected 
components of $\Sigma$ have length less than $\pi$.
  \end{enumerate}
\end{lemma}

We refer the reader to \cite{shu,cpt} for the proof. 

There is a similar notion of $\cT$-geodesics.

\begin{defi}
A {\bf $\cT$-geodesic} in $P$ is a polygonal curve contained in $\cT$,
geodesic and space-like in each face of $P$, such that, at each vertex $v$ and on each side,
either there is (in each neighborhood of $v$) an element of $\Sigma$, or 
the metric is concave. It is {\bf simple} if it intersects every timelike
curve in $\cT$ at most once.
\end{defi}

Here again the notion of concavity used should be straightforward: concavity
is a projective notion so that it can be used as soon as one has a connection,
here we use the Levi-Civita connection of the induced metrics on the faces, 
which is Lorentz since the faces are time-like. 

\begin{lemma} \label{lm:length2}
Let $P\subset \HS^3$ be a polyhedron, then all simple $\cT$-geodesics in $P$ 
have length $L\leq 2\pi$, with equality only when they bound a degenerate
domain in $\cT$.
\end{lemma}

Again we refer the reader to \cite{cpt} for a proof. This statement did not
have any clear meaning in the setting of isometric embeddings of polyhedra,
however it has a natural physical interpretation in the context of interactions
of particles considered below.

\subsubsection{Convexity conditions at the vertices of $P$} \label{ssc:convexity}

Given a convex polyhedron $P\subset \HS^3$, the induced metric satisfies a 
``convexity'' condition at each vertex. If $P$ is a hyperbolic polyhedron,
or the dual of a hyperbolic polyhedron, the condition takes a simple form,
but for the general case one has to consider a series of different conditions
depending on the position of the vertices relative to $\Sigma, \cT$ and $\cH$.
We repeat here the definition in \cite{cpt} (Section 3).

To define precisely those conditions, we consider a marked HS-structure 
$(\mu, \Sigma)$ and call $S_\cT$ the set of singular points $v\in \Sigmab\cap
\cT$ such that all geodesic rays starting from $v$ are space-like. It is 
necessary to define the interior angle at the face of a polygon in a plane
in $dS^3$ (or any Lorentz space-form), we follow here the notations in \cite{cpt}
and refer the reader there for precise definitions. The angle can be real 
(when the induced metric on the polygon is Riemannian) or of the form $k\pi/2+ir$,
where $r\in \R$ and $k\in \N$.

The conditions are then as follows, in the different cases. Each of those conditions
has or might have a ``physical'' interpretation in terms of interactions of particles.

\begin{enumerate}
\item $v\in \cH$: then $\mu$ is required to have positive singular curvature
at $v$ (that is, the sum of the angles of the faces is strictly less than $2\pi$.
\item $v$ is in the interior of $\Sigma$, then $\mu$ has negative singular 
curvature at $v$.
\item $v$ is in the interior of $\cT$, then the sum of the angles at $v$ of 
the incident faces is $2\pi+ir$, with $r>0$.
\item $v$ is an isolated point of $\Sigma$, there is then no condition.
\item $v\in S_\cT$, $\Sigma\setminus \{ v\}$ has two connected components in
the neighborhood of $v$, and, in each, the sum of the angles of the faces
incident to $v$ is in $[0,\pi)$; and $\cT\setminus \{ v\}$ has two connected
components in the neighborhood of $v$ and, in each, the sum of the
angles is in $i\R_{\geq 0}$.
\item $x\in \Sigmab\cap \cT\setminus S_\cT$, $\cT\setminus \{ v\}$ is connected
in the neighborhood of $v$, and the angles $\theta_i$ at $s$ of the faces in 
$\cT$ and the angles $\theta'_j$ at $v$ of the faces in $\Sigma$ satisfy:
$$ \sum_i \theta_i = \pi-ir_1~, ~~\sum_j \theta'_j=r_2~, $$
with $r_1\in \R$, $r_2\geq 0$, and either $r_1>0$ or $r_2<\pi$.
\item $v\in \Sigmab\cap \cT\setminus S_\cT$, $\cT \setminus \{ v\}$ is not 
connected in the neighborhood of $v$, and, for each connected component 
$C$ of $\cT\setminus \{ v\}$ in the neighborhood of $v$, the sum $\alpha$
of the angles at $v$ of the faces in $C$ is in $\pi-i\R_{>0}$,
or $\alpha=\pi$ and all faces in $C$ are light-like; and the sum of all
angles at $v$ is not $2\pi$.
\end{enumerate}

In cases (5) and (6), a sum of angles equal to $0$ corresponds to the case
where the corresponding angular domain is limited to a segment. 

\begin{defi} \label{df:cvx}
A marked HS-structure $(g,\Sigma)$ is {\bf convex} if those conditions are
satisfied.  
\end{defi}

\subsubsection{From marked HS-structures to convex polyhedra}

We are now ready to state a previously known result describing the induced metrics
on convex polyhedra in $HS^3$. 

\begin{theorem}[\cite{cpt}] \label{tm:cpt}
Let $(\sigma, \Sigma)$ be a marked HS-structure on $S^2$. Suppose that $(\sigma,
\Sigma)$ is induced on a convex polyhedron in $\HS^3$. Then $(\sigma, \Sigma)$
satisfies the following properties:
\begin{enumerate}
\item[(A)] $(\sigma, \Sigma)$ is convex (as in Definition \ref{df:cvx}) at its singular
points.
\item[(B)] Closed $\Sigma$-geodesic curves of $(\sigma, \Sigma)$ have length $L>2\pi$, or
$L=2\pi$ if they bound a degenerate domain in $\cT$.
\item[(C)] Closed, simple $\cT$-geodesic curves of $(\sigma, \Sigma)$ 
have length $L<2\pi$, or
$L=2\pi$ if they bound a degenerate domain in $\cT$.
\item[(D)] One of the following is true:
\begin{enumerate}
\item each timelike geodesic on $\cT$ joins $H$ to $\Sigma$
($(\sigma, \Sigma)$ is of hyperbolic type); 
\item $\Sigma=\emptyset$, $H$ has two connected components $H_+$ and $H_-$,
and each timelike geodesic in $\cT$ joins $H_+$ to $H_-$ ($(\sigma, \Sigma)$ is
of bi-hyperbolic type); 
\item $H=\emptyset$, $\Sigma$ has two connected components $\Sigma_+$ and $\Sigma_-$,
and each timelike geodesic in $\cT$ joins $\Sigma_+$ to $\Sigma_-$; and,
moreover, $\Sigma$-geodesic segments in $\Sigma_+$ and $\Sigma_-$ have length $L<\pi$
($(\sigma, \Sigma)$ is of compact type). 
\end{enumerate}
\end{enumerate}
Suppose now that $(\sigma, \Sigma)$ satisfies properties (A),
(B), (C), (D), and also:
\begin{enumerate}
\item[(E)] In case (D.c), $\Sigma_+$ and $\Sigma_-$ are
convex, with boundaries of length less than $2\pi$.   
\end{enumerate}
Then $(\sigma, \Sigma)$ is induced on a the boundary of a 
unique convex polyhedron in $\HS^3$.
\end{theorem}

Clearly this result is not quite complete since there is a slight
discrepancy, for polyhedra of compact type, between the conditions that
are known to be satisfied by marked HS-metrics induced on polyhedra in 
$\HS^3$, and the conditions which are necessary to insure that a marked
HS-metric is actually realized.

\subsection{From convex polyhedra to particle interactions}

The key point of this section is a direct relationship between 
particle interactions in a Lorentz space-form (in particular $AdS^3$),
convex polyhedra in $\HS^3$, and marked polyhedral metrics on the sphere.

\subsubsection{Constructing an interaction from a convex polyhedron}

Consider first a convex polyhedron $P\subset \HS^3$ (with no vertex on
the boundary at infinity of $\HH^3$). $\HS^3$ can be considered as the
link of a point in $AdS^4$ (or for that matter in any Lorentz 
4-manifold) with the quadric $Q$ which is the union of the two sphere
which are the boundaries of the two copies of $\HH^3$ corresponding
to the light cone. The converse works as follows.

\begin{remark}
Let $h$ be the metric on $\HS^3\setminus Q$, let $g$ be the metric 
on $\R_{>0}\times(\HS^3\setminus Q)$ defined as
$$ d(\epsilon r)^2 + \cos(\epsilon r)^2 h~, $$
with $r\in \R_{>0}$ and $\epsilon =i$ for points in the hyperbolic
part of $\HS^3$, and to $1$ in the de Sitter part. Then 
$(\R_{>0}\times(\HS^3\setminus Q), g)$ is isometric to the 
complement of the light cone of a point in the universal cover
of $AdS^4$.
\end{remark}

This remark of course extends to the construction of the 
de Sitter or Minkowski space as a cone over $\HS^3$, with
$\cos$ replaced by $\cosh$, resp. $r\rightarrow r^2$.

Now given a polyhedron $P\subset \HS^3$ as above, we can
consider the cone $P'$ over $P$ as a subset of $AdS^4$
(including the points in the light cone corresponding to 
the intersection of $P$ with the quadric $Q$). It is a
polyhedral cone in $AdS^4$, with a vertex which we call 
$x$. By construction $P$ is identified with the link of 
$P'$ at $x$. The induced metric on the boundary of $P'$, in the neighborhood
of $x$, is a locally AdS metric with cone particles, with 
one interaction point at $x$.

This interaction then satisfies some physically natural
conditions. For instance, any massive particle (singularity
along a timelike line) has positive mass, this is a 
consequence of the fact that the induced metric on the boundary 
of $P$ is ``convex'' (as in Definition \ref{df:cvx}), and
the same holds for tachyons for the same reason. 

\subsubsection{From interactions to convex polyhedra}

The converse is somewhat less clear but works at least to some
extend. Consider a particle 
interaction, that is, a singular point $x$
in an $AdS$ manifold with particles, such that there are at least $3$
singular segments arriving at $x$. Then we can consider the link $L_x$ of
$x$, it is naturally an HS-structure, with cone
singularities corresponding to the ``particles'' -- the singular segments
-- arriving at $x$. 

The convexity conditions appearing in subsection \ref{ssc:convexity}
take on a new meaning in this context, depending
on the position of a vertex $v$ of $L_x$. It is remarkable that the
simplest of those conditions at least are precisely those which are 
physically relevant. In particular, condition (1) is equivalent to 
the condition that all massive particles -- cone singularities
along timelike lines -- have positive mass, while condition (3) 
is the condition that tachyons have positive mass. So it is possible to 
apply Theorem \ref{tm:cpt}, when its conditions are satisfied, e.g. when
only massive particles and tachyons are present, this leads to a convex
polyhedron in $\HSt^3$.

Note also that in this setting, the condition in Lemma \ref{lm:length2},
which was somewhat mysterious in the polyhedral context, has a
physical meaning since it states that the interaction has ``positive
mass'' in a way which extends the usual condition for massive particles.
If there is a space-like disk $D$ containing
the interaction point $x$ and $D$ is totally geodesic outside $x$, then 
the holonomy of the disk, evaluated 
on a simple curve around $x$, is a rotation of 
angle less than $2\pi$. This is because the ``trace'' of $D$ on the link of
$x$ is a $\cT$-geodesic so that the condition in Lemma \ref{lm:length2}
applies.
The interpretation in terms of collisions of the condition on the length 
of the $\Sigma$-geodesics appearing in 
Lemma \ref{lm:length1} is less clear, but it appears only in 
fairly special situations (not in the ``causally regular'' case as defined
in Theorem \ref{tm:thierry}).

\section{Global hyperbolicity}

In previous sections, we considered local properties of AdS manifolds with particles. 
We already observed in \S~\ref{sub.futpast}
that the usual notions of causality (causal curves, future, past, time functions...) 
in regular Lorentzian manifolds still hold.
In this section, we consider the global character of causal properties of AdS manifolds with particles.
The main point presented here is that, as long as no interaction  appears, global hyperbolicity is
%%still a meaningfull notion for singular AdS spacetimes. 
still a meaningfull notion for singular AdS spacetimes. This notion will be necessary in sections
6-8.

In all this {\S} $M$ denotes a singular AdS manifold admitting 
as singularities only massive particules and no interaction.
%%The regular part of $M$ is denoted by $M^\ast$. 
The regular part of $M$ is denoted by $M^\ast$ in this section. 
Since we will consider other Lorentzian metrics on
$M$, we need a denomination for the singular AdS metric~: we denote it $g_0$.

\subsection{Local coordinates near a singular line}
\label{sub.localcoord}

Causality notions only depend on the conformal class of the metric, and
AdS is conformally flat. Hence, AdS spacetimes and flat spacetimes share
the same local causal properties. Every regular AdS spacetime admits
an atlas for which local coordinates have the form $(z, t)$,
where $z$ describes the unit disk $D$ in the complex plane, $t$
the interval $]-1, 1[$ and such that the AdS metric is conformally
isometric to:
%%$$-dt^2 + |dz|^2$$
$$-dt^2 + |dz|^2~. $$

For the singular case considered here, any point $x$ lying on
a singular line $l$ (a massive particule of mass $m$), the same
expression holds, but we have to remove a wedge $ \{ 2\alpha\pi <Arg(z) < 2\pi \}$ where
$\alpha=1-m$, and to glue the two sides of this wedge.
Consider the map $z \to z^\alpha$: it sends the disk with a wedge removed onto
the entire disk $D$, and is compatible with the glueing of the sides of the wedge.
Hence, a convenient
local coordinate system near $x$ is $(z, t)$ where $(z,t)$ still lies in $D \times ]-1, 1[$.
The singular AdS metric is then, in these coordinates, conformally isometric to:
%%$$ (1-m)^2 \frac{dz^2}{z^{2m}} - dt^2$$
$$ (1-m)^2 \frac{dz^2}{z^{2m}} - dt^2~. $$

In these coordinates, future oriented causal curves can be parametrized by
the time coordinate $t$, and satisfies:
%%$$\frac{z'(t)}{|z|^{m}} \leq \frac{1}{\alpha}$$
$$ \frac{z'(t)}{|z|^{m}} \leq \frac{1}{\alpha}~. $$

Observe that all these local coordinates define a differentiable atlas on the
topological manifold $M$ for which the AdS metric on the regular part is smooth.

\subsection{Achronal surfaces}
Usual definitions in regular Lorentzian manifolds still apply to the singular AdS spacetime $M$:

\begin{defi}
%%A subset $S$ of $M$ is achronal (resp. acausal) if there is no non-trivial timelike curve joining two points in $S$
%%(resp. causal). It is only locally achronal (resp. acausal) if every point in $S$ admits a neighborhood $U$ such that
A subset $S$ of $M$ is achronal (resp. acausal) if there is no non-trivial timelike (resp. causal) 
curve joining two points in $S$. 
It is only locally achronal (resp. acausal) if every point in $S$ admits a neighborhood $U$ such that
the intersection $U \cap S$ is achronal (resp. acausal) inside $U$.
\end{defi}

Typical examples of locally acausal subsets are spacelike surfaces, but the definition above
also includes non-differentiable "spacelike" surfaces, with only Lipschitz regularity. Lipschitz
spacelike surfaces provide actually the general case if one adds the {\it edgeless} assumption~:

\begin{defi}
%%A locally achronal subset $S$ is edgeless if for every point $x$ in $S$ admits a neighborhood
A locally achronal subset $S$ is edgeless if every point $x$ in $S$ admits a neighborhood
$U$ such that every causal curve in $U$ joining one point of the past of $x$ (inside $U$)
to a point in the future (in $U$) of $x$ intersects $S$.
\end{defi}

In the regular case, closed edgeless locally achronal subsets are embedded locally Lipschitz surfaces.
More precisely, in the coordinates $(z,t)$ defined in \S~\ref{sub.localcoord}, they are graphs
of $1$-Lipschitz maps defined on $D$.

%%This property still holds in $M$, except the locally Lipschitz property which is not anymore valid at
This property still holds in $M$, except the locally Lipschitz property which is not valid anymore at
%%singular points, but only a weaker weighted version:
singular points, but only a weaker weighted version holds:
closed edgeless acausal subsets containing $x$ corresponds to functions 
$f: D \to ]-1, 1 [$ differentiable almost everywhere and
satisfying:
%%$$\Vert d_zf \Vert < \alpha |z|^{-m} $$
$$\Vert d_zf \Vert < \alpha |z|^{-m}~. $$
It is sometimes more relevant to adopt the following point of view:
the coordinates $(z, t)$ have the the form $(\zeta^\alpha, t)$ where $\zeta$
describes the disk with the wedge $ \{ 2\alpha\pi <Arg(\zeta) < 2\pi \}$ removed.
The acausal subset is then the graph of a $1$-Lipschitz map $\varphi$ over
the disk minus the wedge. Moreover, the values of $\varphi$ on the boundary of
the wedge must coincide since they have to be send one to the other by the rotation
performing the glueing. Hence, for every $r < 1$:
%%$$\varphi(r)=\varphi(re^{i2\alpha\pi})$$
$$\varphi(r)=\varphi(re^{i2\alpha\pi})~. $$
We can extend $\varphi$ over the wedge by defining $\varphi(re^{i\theta}) = \varphi(r)$
for $2\alpha\pi \leq \theta \leq 2\pi$. This extension over the entire $D \setminus \{ 0 \}$ is then clearly
$1$-Lipschitz. It therefore extends to $0$. We have just proved:

\begin{lemma}
The closure of any closed edgeless achronal subset of $M^\ast$
is a closed edgeless achronal subset of $M$.
\end{lemma}

\begin{defi}
A spacelike surface $S$ in $M$ is a closed edgeless locally acausal subset whose
intersection with the regular part $M^\ast$ is a smooth embedded spacelike surface.
\end{defi}

\subsection{Time functions}
\label{sub.singtime}

As in the regular case, we can define time functions as maps $T: M \to \RR$ which
are increasing along any future oriented causal curve. For non-singular
spacetimes the existence is related to \textit{stable causality~:}

\begin{defi}
Let $g$, $g'$ be two Lorentzian metrics on the same manifold $X$.
Then, $g'$ dominates $g$ if
every causal tangent vector for $M$ is timelike for $g'$. We denote
this relation by $g \prec g'$.
\end{defi}

\begin{defi}
%%A Lorentzian metric $g$ is stably causal if there is metric $g'$ such that
A Lorentzian metric $g$ is stably causal if there is a metric $g'$ such that
$g \prec g'$, and such that $(X, g')$ is chronological, i.e. admits no periodic timelike
curve.
\end{defi}

\begin{theorem}[See \cite{beem}]
\label{thm.stabletime}
A Lorentzian manifold $(M, g)$ admits a time function iff it is
\textit{stably causal.} Moreover, when a time function exists, then
there is a smooth time function.
\end{theorem}

\begin{remark}
\label{rk.stablecausal}
In section~\ref{sub.localcoord} we defined some differentiable atlas on
the manifold $M$. For this differentiable structure, the null cones of $g_0$
degenerate along singular lines to half-lines tangent to the "singular" line
(which is perfectly smooth for the selected differentiable atlas). Obviously,
we can extend the definition of domination to the more general case
$g_0 \prec g$ where $g_0$ is our singular metric and $g$ a smooth regular
metric. Therefore, we can define the stable causality of in this context:
$g_0$ is stably causal if there is a smooth Lorentzian metric $g'$ which is
achronological and such that $g_0 \prec g$.
Theorem~\ref{thm.stabletime} is still valid in this more general context.
Indeed, there is a smooth Lorentzian metric $g$ such that $g_0 \prec g \prec g'$,
which is stably causal since $g$ is dominated by the achronal metric $g'$.
Hence there is a time function
$T$ for the metric $g$, which is still a time function for $g_0$ since $g_0 \prec g$:
causal curves for $g_0$ are causal curves for $g$.
\end{remark}

\begin{lemma}
\label{le.singstable}
The singular metric $g_0$ is stably causal if and only if its restriction
to the regular part $M^\ast$ is stably causal. Therefore, $(M, g_0)$ admits
a smooth time function if and only if $(M^\ast, g_0)$ admits a time function.
\end{lemma}

\begin{proof}
The fact that $(M^\ast, g_0)$ is stably causal as soon as $(M, g_0)$ is stably
causal is obvious. Let's assume that $(M, g_0)$ is stably causal:
let $g'$ be smooth achronological Lorentzian metric on $M^\ast$ dominating
$g_0$. On the other hand, using the local models around singular lines,
it is easy to construct an achronological Lorentzian metric $g''$ on a tubular neighborhood $U$
of the singular locus of $g_0$ (the fact that $g'$ is achronological implies that
the singular lines are not periodic). Reducing $g''$ if necessary, we can assume that
$g'$ dominates $g''$ on $U$. Let $U' $ be a smaller tubular neighborhood of
the singular locus so that $\overline{U}' \subset U$, and let $a$, $b$ be
a partition of unity subordonate to $U$, $M \setminus U'$. Then
$g_1 = ag'' + bg'$ is a smooth Lorentzian metric dominating $g_0$.
Moreover, we also have $g_1 \subset g'$ on $M^\ast$. Hence any timelike curve
for $g_1$ can be slightly perturbed to an achronological timelike curve for $g'$
avoiding the singular lines. It follows that $(M, g_0)$ is stably causal.
\end {proof}

\subsection{Cauchy surfaces}

\begin{defi}
A spacelike surface $S$ is a Cauchy surface if it is acausal and intersects every inextendible causal curve in $M$.
\end{defi}

Since a Cauchy surface is acausal, its future $I^+(S)$ and its past $I^-(S)$ are disjoint.

\begin{remark}
\label{rk.ghnongh}
The regular part of a Cauchy surface in $M$ is not a Cauchy surface in
the regular part $M^\ast$, since causal curves can exit the regular region through a
timelike singularity.
\end{remark}

\begin{defi}
A singular AdS spacetime is globally hyperbolic if it admits a Cauchy surface $M$.
\end{defi}

\begin{remark}
We defined Cauchy surfaces as smooth objects for further requirements in this paper,
but this definition can be generalized for non-smooth locally achronal closed subsets.
This more general definition leads to the same notion of globally hyperbolic spacetimes,
i.e. singular spacetimes admitting a non-smooth Cauchy surface also admits a smooth one.
\end{remark}

\begin{prop}
\label{pro.ghtime}
%%Let $M$ be a singular AdS spacetime without interaction and regular
Let $M$ be a singular AdS spacetime without interaction and singular
%%set consisting reduced to massive particles. Assume that $M$ is globally hyperbolic.
set reduced to massive particles. Assume that $M$ is globally hyperbolic.
Then $M$ admits a time function $T: M \to \RR$ such that every level
$T^{-1}(t)$ is a Cauchy surface.
\end{prop}

\begin{proof}
This is a well-known theorem by Geroch in the regular case, even for
general globally hyperbolic spacetimes without compact Cauchy surfaces.
But, the singular version does not follow immediately by applying this
regular version to $M^\ast$ (see Remark~\ref{rk.ghnongh}).

Let $l$ be an inextendible causal curve in $M$. It intersects the Cauchy surface $S$,
and since $S$ is achronal, $l$ cannot be periodic. Therefore, $M$ admits no periodic
causal curve, i.e. is \textit{acausal.}

%We now intend to prove that the regular part $M^\ast$
%is stably causal for the metric $g_0$.

Let $U$ be a small tubular neighborhood of $S$ in $M$, so that the boundary $\partial U$
is the union of two spacelike hypersurfaces $S_-$, $S_+$ with $S_- \subset I^-(S)$,
$S_+ \subset I^+(S)$,
and such that every inextendible future oriented causal curve in $U$ starts from
%%$S_-$, intersects $S$ and then hits $S^+$. If there was some acausal curve joining
$S_-$, intersects $S$ and then hits $S^+$. If there were some acausal curve joining
two points in $S_-$, then we could add to this curve some causal segment in $U$
and obtain some acausal curve in $M$ joining two points in $S$. Since $S$ is acausal,
this is impossible~: $S_-$ is acausal. Similarly, $S_+$ is acausal. It follows that
$S_\pm$ are both Cauchy surfaces for $(M, g_0)$.

For every $x$ in $I^+(S_-)$ and every past oriented $g_0$-causal tangent vector $v$,
the past oriented geodesic tangent to $(x,v)$ intersects $S$. The same property holds for
tangent vector $(x, v')$ nearby. It follows that there exists on $I^+(S_-)$ a smooth Lorentzian metric
$g'_1$ so that $g_0 \prec g'_1$ and such that every inextendible past oriented $g'_1$-causal
curve attains $S$. Furthermore, we can select $g'_1$ so that $S$ is $g'_1$-spacelike,
and such that every future oriented $g'_1$-causal
vector tangent at a point of $S$ points in the $g_0$-future of $S$. It follows that future oriented $g'_1$-causal curves
crossing $S$ cannot come back to $S$: $S$ is acausal, not only for $g_0$, but
also for $g'_1$.

We can also define $g'_2$ in the past of $S_+$ so that
$g_0 \prec g'_2$, every inextendible future oriented $g'_2$-causal
curve attains $S$, and such that $S$ is $g'_2$-acausal. We can now interpolate in the common region
$I^+(S_-) \cap I^-(S_+)$, getting a Lorentzian metric $g'$ on the entire $M$
so that $g_0 \prec g' \prec g'_1$ on $I^+(S_-)$, and
$g_0 \prec g' \prec g'_2$ on $I^-(S_+)$. Observe that even if it is not totally
obvious that the metrics $g'_i$ can be selected continuous, we have enough room to
pick such a metric $g'$ in a continuous way.

Let $l$ be a future oriented $g'$-causal curve starting from a point in $S$. Since $g' \prec g'_1$, this
curve is also $g'_1$-causal as long as it remains inside $I^+(S_-)$. But since $S$ is acausal for
%%$g'_1$, it implies that $l$ cannot cross anymore $S$: hence $l$ lies entirely in $I^+(S)$.
$g'_1$, it implies that $l$ cannot cross $S$ anymore: hence $l$ lies entirely in $I^+(S)$.
It follows that $S$ is acausal for $g'$.

By construction of $g'_1$, every past-oriented $g'_1$-causal curve starting from a point inside $I^+(S)$
must intersect $S$. Since $g' \prec g'_1$ the same property holds for $g'$-causal curves.
Using $g'_2$ for points in $I^+(S_-)$, we get that every inextendible $g'$-causal curve intersects
$S$. Hence, $(M, g')$ is globally hyperbolic. According to Geroch's Theorem in the regular case,
there is a time function $T: M \to \RR$ whose levels are Cauchy surfaces. The proposition
follows, since $g_0$-causal curves are $g'$-causal curves, implying that $g'$-Cauchy surfaces
are $g_0$-Cauchy surfaces and that $g'$-time functions are $g_0$-time functions.
\end{proof}

\begin{cor}
\label{cor.splitting}
If $(M, g_0)$ is globally hyperbolic, there is a decomposition $M \approx S \times \RR$
where every level $S \times \{ \ast \}$ is a Cauchy surface, and very vertical line
$\{ \ast \} \times \RR$ is a singular line or timelike.
\end{cor}

\begin{proof}
Let $T: M \to \RR$ be the time function provided by Proposition~\ref{pro.ghtime}.
Let $X$ be minus the gradient  (for $g_0$) of
$T$: it is a future oriented timelike vector field on $M^\ast$. Consider also a future oriented
timelike vector field $Y$ on a tubular neighborhood $U$ of the singular locus: using a partition
of unity as in the proof of Lemma~\ref{le.singstable}, we can construct a smooth timelike vector field
$Z = aY + bX$ on $M$ tangent to the singular lines. The orbits of the flow generated by $Z$
are timelike curves. The global hyperbolicity of $(M, g_0)$ ensures that each of these orbits
intersect every Cauchy surface, in particular, the levels of $T$. In other words, for every $x$ in $M$
the $Z$-orbit of $x$ intersects $S$ at a point $p(x)$. Then the map
$F: M \to S \times \RR$ defined by $F(x) = (p(x), T(x))$ is the desired diffeomorphism
between $M$ and $S \times \RR$.
\end{proof}

\subsection{Maximal globally hyperbolic extensions}

From now we assume that $M$ is globally hyperbolic, admitting a compact Cauchy surface $S$.
In this section, we prove the following facts, well-known in the case of regular globally hyperbolic solutions of the Einstein equation (\cite{gerochdependence}): \textit{there is exist a maximal extension, which is unique up to isometry.}

\begin{defi} An isometric embedding $i: (M, S) \to (M', S')$ is a Cauchy embedding if
$S'=i(S)$ is a Cauchy surface of $M'$.
\end{defi}

\begin{remark}
If $i: M \to M'$ is a Cauchy embedding then the image $i(S')$ of any Cauchy surface $S'$
of $M$ is also a Cauchy surface in $M'$. Indeed, for every inextendible causal curve $l$
in $M'$, every connected component of the preimage $i^{-1}(l)$ is an inextendible causal
curve in $M$, and thus intersects $S$. Since $l$ intersects $i(S)$ in exactly one point,
$i^{-1}(l)$ is connected. It follows that the intersection $l \cap i(S')$ is non-empty and reduced to a single point: $i(S')$ is a Cauchy surface.

Therefore, we can define Cauchy embeddings without reference to the selected Cauchy surface
$S$. However, the natural category is the category of \textit{marked} globally hyperbolic
spacetimes, i.e. pairs $(M, S)$.
\end{remark}

\begin{lemma}
\label{le.coincide}
Let $i_1: (M, S) \to (M', S')$, $i_2: (M, S) \to (M', S')$ two Cauchy embeddings into the same marked globally hyperbolic singular AdS spacetime $(M', S')$. Assume that
$i_1$ and $i_2$ coincide on $S$. Then, they coincide on the entire $M$.
\end{lemma}

\begin{proof}
If $x'$, $y'$ are points in $M'$ sufficiently near to $S'$, say, in the future of $S'$,
then they are equal if and only if the intersections $I^-(x') \cap S'$ and
$I^-(y') \cap S'$ are equal. Apply this observation to $i_1(x)$, $i_2(x)$ for $x$
near $S$: we obtain that $i_1$, $i_2$ coincide in a neighborhood of $S$.

Let now $x$ be any point in $M$.
Since there is only a finite number of singular lines
in $M$, there is a timelike geodesic segment $[y, x]$, where $y$ lies in $S$,
and such that $[y, x[$ is contained in $M^\ast$ ($x$ may be singular).
Then $x$ is the image by the exponential map of some $\xi$ in $T_yM$.
Then $i_1(x)$, $i_2(x)$ are the image by the exponential map of respectively
$d_yi_1(\xi)$, $d_yi_2(\xi)$. But these tangent vectors are equal, since $i_1 = i_2$
near $S$.
\end{proof}

\begin{lemma}
\label{le.causalconvex}
Let $i: M \to M'$ be a Cauchy embedding into a singular AdS spacetime. Then, the image of $i$ is
causally convex, i.e. any causal curve in $M'$ admitting extremities
in $i(M)$ lies inside $i(M)$.
\end{lemma}

\begin{proof}
Let $l$ be a causal segment in $M'$ with extremities in $i(M)$. We extend it as an inextendible causal curve $\hat{l}$.
Let $l'$ be a connected component of $\hat{l} \cap i(M)$: it is an inextendible causal curve
inside $i(M)$. Thus, its intersection with $i(S)$ is non-empty. But $\hat{l} \cap i(S)$ contains at most one point: it follows that $\hat{l} \cap i(M)$ admits only one connected component, which contains $l$.
\end {proof}

\begin{cor}
\label{cor.bordgh}
The boundary of the image of a Cauchy embedding $i: M \to M'$
is the union of two closed edgeless achronal subsets $S^+$, $S^-$ of $M'$,
and $i(M)$ is the intersection between the past of $S^+$ and the future of $S^-$.
\end{cor}

Each of $S^+$, $S^-$ might be empty, and is not necessarily connected.

\begin{proof}
This is a general property of causally convex open subsets: $S^+$ (resp. $S^-$) is the set of elements in the boundary of $i(M)$ whose past (resp. future) intersects $i(M)$. The proof is
straightforward and left to the reader.
\end{proof}

\begin{defi}
$(M, S)$ is maximal if every Cauchy embedding $i: M \to M'$ into a singular AdS spacetime is surjective, i.e. an isometric homeomorphism.
\end{defi}

\begin{prop}
$(M, S)$ admits a maximal singular AdS extension, i.e. a Cauchy embedding into
a maximal globally hyperbolic singular AdS spacetime $(\widehat{M}, \hat{S})$ without interaction.
\end{prop}

\begin{proof}
Let $\mathcal M$ be set of Cauchy embeddings $i: (M, S) \to (M', S')$.
We define on $\mathcal M$ the relation $(i_1, M_1, S_1) \preceq (i_2, M_2, S_2)$ if
there is a Cauchy embedding $i: (M_1, S_1) \to (M_2, S_2)$ such that $i_2 = i \circ i_1$.
It defines a preorder on $\mathcal M$. Let $\overline{\mathcal M}$ be the space
%%of Cauchy embeddings up to isometry, i.e. the quotient space of equivalence relation 
of Cauchy embeddings up to isometry, i.e. the quotient space of the equivalence relation 
identifying $(i_1, M_1, S_1)$ and $(i_2, M_2, S_2)$ if
there is an isometric homeomorphism $i: (M_1, S_1) \to (M_2, S_2)$ such that
$i_2 = i \circ i_1$. Then $\preceq$ induces on $\overline{\mathcal M}$ a preorder relation, that we still denote by $\preceq$. Lemma~\ref{le.coincide} ensures
that $\preceq$ is a partial order (if $(i_1, M_1, S_1) \preceq (i_2, M_2, S_2)$
and $(i_2, M_2, S_2) \preceq (i_1, M_1, S_1)$, then $M_1$ and $M_2$ are isometric
one to the other and represent the same element of $\overline{\mathcal M})$.
Now, any totally ordered subset
$A$ of $\overline{\mathcal M}$ admits an upper bound in $A$: the inverse limit
of (representants of) the elements of $A$. By Zorn Lemma, we obtain that $\overline{M}$ contains a maximal
element. Any representant in $\mathcal M$ of this maximal element is a maximal extension
of $(M, S)$.
\end{proof}

\begin{remark}
The proof above is sketchy: for example, we didn't justify the fact that the inverse
limit is naturally a singular AdS spacetime. This is however a straightforward verification,
%%the same than in the classical situation, and left to the reader.
the same as in the classical situation, and is left to the reader.
\end{remark}

\begin{prop}
The maximal extension of $(M, S)$ is unique up to isometry.
\end{prop}

\begin{proof}
Let $(\widehat{M}_1, S_1)$, $(\widehat{M}_1, S_1)$ be two maximal extensions of $(M,S)$.
Consider the set of globally hyperbolic singular AdS spacetimes $(M', S')$ for which there
is a commutative diagram as below, where arrows are Cauchy embeddings.

\begin{center}
$$ \xymatrix{
       &  & (\widehat{M}_1, S_1) \\
    (M, S) \ar[r] \ar[rru] \ar[rrd] & (M', S') \ar[ru] \ar[rd] & \\
       &  & (\widehat{M}_2, S_2)
  } $$
\end{center}

Reasoning as in the previous proposition, we get that this set admits a maximal element:
there is a marked extension $(M', S')$ of $(M,S)$, and Cauchy embeddings $\varphi_i: M' \to \widehat{M}_i$ which cannot be simultaneously extended.

Define $\widehat{M}$ as the union of $(\widehat{M}_1, S_1)$ and $(\widehat{M}_2, S_2)$,
identified along their respective embedded copies of $(M', S')$, through $\varphi:=\varphi_2 \circ \varphi_1^{-1}$, equipped with the quotient topology. The key point is to prove that
$\widehat{M}$ is Hausdorff. Assume not: there is a point $x_1$ in $\widehat{M}_1$,
a point $x_2$ in $\widehat{M}_2$, and a sequence $y_n$ in $M'$ such that $\varphi_i(y_n)$
converges to $x_i$, but such that $x_1$ and $x_2$ do not represent the same
element of $\widehat{M}$. It means that $y_n$ does not converge in $M'$, and that
$x_i$ is not in the image of $\varphi_i$. Let $U_i$ be small neighborhoods in $\widehat{M}_i$ of $x_i$.

Denote by $S^+_i$, $S^-_i$ the upper and lower boundaries of $\varphi_i(M')$ in $\widehat{M}_i$ (cf.
Corollary~\ref{cor.bordgh}). Up to time reversal, we can assume that $x_1$ lies
in $S^+_1$: it implies that all the $\varphi_1(y_n)$ lies in $I^-(S^+_1)$, and that,
if $U_1$ is small enough, $U_1 \cap I^-(x_1)$ is contained in $\varphi_1(M')$. It is an open subset, hence $\varphi$ extends to some AdS isometry $\overline{\varphi}$ between $U_1$ and $U_2$ (reducing the $U_i$ if necessary). Therefore, every $\varphi_i$ can be extended
to isometric embeddings $\overline{\varphi}_i$ of a spacetime $M''$ containing $M'$, so that
$$\overline{\varphi}_2 = \overline{\varphi} \circ \overline{\varphi}_1$$

We intend to prove that $x_i$ and $U_i$ can be chosen so that $S_i$ is a Cauchy surface in $\overline{\varphi}_i(M') \cup U_i$. Consider past oriented causal curves, starting from $x_1$, and contained in $S^+_1$. They are partially ordered by the inclusion. According to Zorn lemma,
there is a maximal causal curve $l_1$ satisfying all these properties.
Since $S^+_1$ is disjoint from $S_1$, and since every inextendible causal curve crosses
$S$, the curve $l_1$ is not inextendible: it has a final endpoint $y_1$ belonging
to $S^+_1$ (since $S^+_1$ is closed). 
%%Then, any past causal oriented causal curve starting from $y_1$ is disjoint from $S^+_1$ 
Therefore, any past oriented causal curve starting from $y_1$ is disjoint from $S^+_1$ 
(except at the starting point $y_1$).

We have seen that $\varphi$ can be extended over in a neighborhood of $x_1$: this extension
maps the initial part of $l_1$ onto a causal curve in $\widehat{M}_2$ starting from $x_2$ and contained in $S^+_2$. By compactness of $l_1$, this extension can be performed along the entire $l_1$, and the image is a causal curve admitting a final point $y_2$ in $S^+_2$. The points
$y_1$ and $y_2$ are not separated one from the other by the topology of $\widehat{M}$.
Replacing $x_i$ by $y_i$, we can thus assume that \textit{every past oriented causal curve
%%starting from $x_i$ is contained in $I^-(S^+_i)$.} Then, it follows that, once more reducing 
starting from $x_i$ is contained in $I^-(S^+_i)$.} It follows that, once more reducing 
$U_i$ if necessary, inextendible past oriented causal curves starting from points in $U_i$ 
and in the future of $S^+_i$ intersects $S^+_i$ before escaping from $U_i$. In other words,
inextendible past oriented causal curves in $U_i \cup I^-(S^+_i)$ are also 
inextendible causal curves in $\widehat{M}_i$, and therefore, intersect $S_i$. 
%%As required, $S_i$ is a Cauchy surface in $U_i \cup \varphi_i(M')$.
As required, $S_i$ is a Cauchy surface in $U_i \cup \overline{\varphi_i}(M')$.

Hence, there is a Cauchy embedding of $(M, S)$ into some globally hyperbolic spacetime
$(M'', S'')$, and Cauchy embeddings
$\overline{\varphi}_i: (M'', S'') \to \varphi_i(M') \cup U_i$, which are related by
some isometry $\overline{\varphi}: \varphi_1(M') \cup U_1 \to \varphi_2(M') \cup U_2$:
$$\overline{\varphi}_2 = \overline{\varphi} \circ \overline{\varphi}_1$$

It is a contradiction with the maximality of $(M', S')$. Hence, we have proved that
$\widehat{M}$ is Hausdorff. It is a manifold, and the singular AdS metrics on $\widehat{M}_1$, $\widehat{M}_2$ induce a singular AdS metric on $\widehat{M}$.
Observe that $S_1$ and $S_2$ projects in $\widehat{M}$ onto the same spacelike surface
$\widehat{S}$. Let $l$ be any extendible curve in $\widehat{M}$. Without loss of generality, we can assume that $l$ intersects the projection $W_1$ of $\widehat{M}_1$ in $\widehat{M}$.
Then every connected component of $l \cap W_1$ is an inextendible causal curve in
$W_1 \approx \widehat{M}_1$. It follows that $l$ intersects $\widehat{S}$. Finally,
if some causal curve links two points in $\widehat{S}$, then it must be contained
in $W_1$ since globally hyperbolic open subsets are causally convex. It would contradict
the acausality of $S_1$ inside $\widehat{M}_1$.

The conclusion is that $\widehat{M}$ is globally hyperbolic, and that $\widehat{S}$
is a Cauchy surface in $\widehat{M}$. In other words, the projection of $\widehat{M}_i$
into $\widehat{M}$ is a Cauchy embedding. Since $\widehat{M}_i$ is a maximal extension,
these projections are onto. Hence $\widehat{M}_1$ and $\widehat{M}_2$ are isometric.
\end{proof}

\begin{remark}
The uniqueness of the maximal globally hyperbolic AdS extension is no more true if
we allow interactions. Indeed, in the next {\S} we will see how, given some singular AdS spacetime without interaction, to define a surgery near a point in a singular line, introducing some collision or interaction at this point. The place where such a surgery can be performed is arbitrary.
\end{remark}  %global hyperbolicity

\section{Global examples}

The main goal of this section is to construct examples of globally hyperbolic
AdS manifolds with interacting particles, so we go beyond the local examples 
constructed in section 2. It will be shown in the next section that some of the examples
constructed in this section are ``good'' space-times, as in Definition \ref{df:good}.
It will therefore follow from Theorem \ref{tm:homeo} that there are large moduli
spaces of globally hyperbolic spaces with interacting particles, obtained by
deforming the fairly special cases constructed here.

\subsection{An explicit example} \label{ssc:explicit}

Let $S$ be a hyperbolic surface with one cone point $p$ of angle
$\theta$.   Denote by $\mu$ the corresponding
%%hyperbolic singular metric on $S$.
singular hyperbolic metric on $S$.

Let us consider the Lorentzian metric on $S\times (-\pi/2,\pi/2)$ given by
\begin{equation}\label{ads:eq}
   h= -{\rm d}t^2\ + \cos^2t \ \mu
 \end{equation}
 where $t$ is the real parameter of the interval $(-\pi/2,\pi/2)$.

 We denote by $M(S)$ the singular spacetime $(S\times(-\pi/2,\pi/2),h)$.
 \begin{lemma}
 $M(S)$ is an AdS spacetime with a particle corresponding to the
   singular line $\{p\}\times(-\pi/2,\pi/2)$.  The corresponding cone
   angle is $\theta$. Level surfaces $S\times\{t\}$ are orthogonal to
   the singular locus.
 \end{lemma}
\begin{proof}
First we show that $h$ is an AdS metric. The computation is local, so
we can assume $S=\mathbb H^2$. Thus we can identify $S$ to a geodesic
%%plane in $AdS_3$. 
plane in $AdS_3$. We consider $AdS_3$ as embedded in $\R^{2,2}$, as mentioned in the introduction.
Let $n$ be the normal direction to $S$ then we can
consider the normal evolution
\[
      F:S\times (-\pi/2,\pi/2)\ni(x,t)\mapsto \cos t x +\sin t n\in AdS_3
\]
The map $F$ is a diffeomorphism onto an open domain of $AdS_3$ and the
pull-back of the $AdS_3$-metric takes the form (\ref{ads:eq}).

To prove that $\{p\}\times (-\pi/2,\pi/2)$ is a conical singularity of
angle $\theta$, take a geodesic plane $P$ in $\mathcal P_\theta$
orthogonal to the singular locus. Notice that $P$ has exactly one cone
point $p_0$ corresponding to the intersection of $P$ with the singular
%%line of $\mathcal P_\theta$. 
line of $\mathcal P_\theta$ (here $P_\theta$ is the singular model space
defined in subsection \ref{sub.singular_lines}). 
Since the statement is local, it is
sufficient to prove it for $P$. Notice that the normal evolution of
%%$P\setminus\{p_0\}$ is well- defined for any $t$. Moreover, such
$P\setminus\{p_0\}$ is well-defined for any $t\in (-\pi/2,\pi/2)$. Moreover, such
evolution can be extended to a map on the whole
$P\times(-\pi/2,\pi/2)$ sending $\{p_0\}\times (-\pi/2,\pi/2)$ onto
%%the singular line. This map, in turn, is a diffeomorphism of $P\times
the singular line. This map is a diffeomorphism of $P\times
(-\pi/2,\pi/2)$ with an open domain of $\mathcal P_\theta$. Since the
pull-back of the $AdS$-metric of $\mathcal P_\theta$ on
$(P\setminus\{p_0\})\times(-\pi/2,\pi/2)$ takes the form
(\ref{ads:eq}) the statement follows.
\end{proof}

Let $\Sigma$ be a causally regular HS-sphere with an elliptic future
singularity at $p$ of angle $\theta$ and two elliptic past
singularities, $q_1,q_2$ of angles $\eta_1,\eta_2$.

Let $r$ be the future singular ray in $e(\Sigma)$. For a given
$\epsilon>0$ let $p_\epsilon$ be the point at distance $\epsilon$ from
the interaction point. Consider the geodesic disk $D_\epsilon$ in
$e(\Sigma)$ centered at $p_\epsilon$, orthogonal to $r$ and with
radius $\epsilon$.

The past normal evolution $n_t:D_\epsilon\rightarrow e(\Sigma)$ is
well-defined for $t\leq\epsilon$.  In fact, if we restrict to the
annulus $A_\epsilon=D_\epsilon\setminus D_{\epsilon/2}$, the evolution
can be extended for $t\leq\epsilon'$ for some $\epsilon'>\epsilon$.

%\begin{figure}
%\begin{center}
%\input exa.pstex_t
%\includegraphics{exa.pstex}
%\end{center}
%\end{figure}

Let us set
\[
\begin{array}{l}
%%   U_\epsilon=\{n_t(p)~|~p\in D_\epsilon,
   U_\epsilon=\{n_t(p)|p\in D_\epsilon,
%%   t\in(0,\epsilon)\}\\ \Delta_\epsilon=\{n_t(p)|p\in
   t\in(0,\epsilon)\}\\ \Delta_\epsilon=\{n_t(p)~|~p\in
   D_\epsilon\setminus D_{\epsilon/2}, t\in (0,\epsilon')\}
   \end{array}
\]
Notice that the interaction point is in the closure of $U_\epsilon$.
It is possible to contruct a neighbourhood $\Omega_\epsilon$ of the
interaction point such that
\begin{itemize}
\item
$U_\epsilon\cup\Delta_\epsilon\subset\Omega_\epsilon$;
\item
$\Omega_\epsilon$ admits a foliation in achronal disks
  $(D(t))_{t\in(0,\epsilon')}$ such that
\begin{enumerate}
\item
 $D(t)=n_t(D_\epsilon)$ for $t\leq\epsilon$
\item
  $D(t)\cap\Delta_t=n_t(D_\epsilon\setminus D_{\epsilon/2})$ for
  $t\in(0,\epsilon')$
 \item
  $D(t)$ is orthogonal to the singular locus.
  \end{enumerate}
  \end{itemize}

%% added the next sentence.
Consider now the space $M(S)$ as in the previous lemma.
%%For small $\epsilon$ the disk $D_\epsilon$ embeds in $M(S)$ sending
For small $\epsilon$ the disk $D_\epsilon$ embeds in $M(S)$, sending
%%$p_\epsilon$ to $\{p\}\times\{0\}$.
$p_\epsilon$ to $(p,0)$.

Let us identify $D_\epsilon$ with its image in $M(S)$. The normal
%%evolution on $D_\epsilon$ in $M(S)$ is well-defined for $t<\pi/2$ and
evolution on $D_\epsilon$ in $M(S)$ is well-defined for $0<t<\pi/2$ and
in fact coincides with the map
\[
   n_t(x,0)=(x,t)\,.
\]
%% Thus it follows that the map
It follows that the map
 \[
    F:(D_\epsilon\setminus D_{\epsilon/2})\times
    (0,\epsilon')\rightarrow \Delta_\epsilon
 \]
 defined by $F(x,t)=n_t(x)$ is an isometry.

%% added next figure
\begin{figure}[ht]
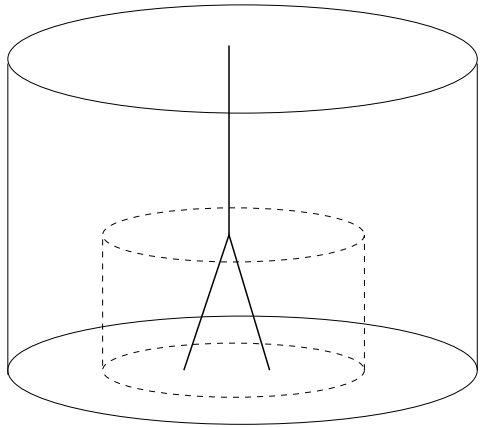

\input interaction.pstex_t
\caption{Surgery to add a collision.}
\label{fig.collision}
\end{figure}

 Thus if we glue $(\Sigma\setminus D_{\epsilon/2})\times
 (0,\epsilon')$ to $\Omega_\epsilon$ by identifying
 $D_{\epsilon}\setminus D_{\epsilon/2}$ to $\Delta_\epsilon$ via $F$
 we get a spacetime
 \[
   \hat M=(\Sigma\setminus D_{\epsilon/2})\times(0,\epsilon)\cup_F
   \Omega_\epsilon
 \]
  such that
  \begin{enumerate}
  \item
%%  topologically $\hat M\cong\Sigma\times\mathbb R$
  topologically, $\hat M$ is homeomorphic to $S\times\mathbb R$,
  \item
%%  in $\hat M$ two particles collide producing a particle
  in $\hat M$, two particles collide producing one particle only,
  \item
%%  $\hat M$ admits a foliation in spacelike surfaces orthogonal to the
  $\hat M$ admits a foliation by spacelike surfaces orthogonal to the
    singular locus.
  \end{enumerate}

  We say that $\hat M$ is obtained by a surgery on $M(S)$.

  \subsection{Surgery}\label{surg:sec}

%%In this section we will try to get a generalization of the
%%construction we explained in the previous section.  In particular we
%%will show how to do a surgery on a spacetime with conical singularity
%%in order to obtain a spacetime with collision also more complicated
%%than that described in the previous section.
In this section we get a generalization of the
construction explained in the previous section.  In particular we
show how to do a surgery on a spacetime with conical singularity
in order to obtain a spacetime with collision more complicated
than that described in the previous section.

\begin{lemma}\label{surgery-key}
Let $\Sigma$ be a causally regular HS-sphere containing only elliptic
singularities.

If $C_-$ and $C_+$ denote the circles of photons of $\Sigma$, the
respective projective structures are isomorphic.

%%Suppose moreover that such projective structures are elliptic of angle $\theta$.
Suppose moreover that those projective structures are elliptic of angle $\theta$.
Then $e(\Sigma)\setminus (I^+(p_0)\cup I^-(p_0))$ embeds in $\mathcal P_\theta$.
\end{lemma}
\begin{proof}

Let $D$ be the de Sitter part of $\Sigma$, and consider the developing map
\[
  d:\tilde D\rightarrow\tilde{dS_2}
\]
As in Section~\ref{sec.classificationHS}, $\tilde {dS}_2$ can be
completed by two lines of photons, say $R_+, R_-$ that are
projectively isomorphic to $\tilde{\mathbb R\mathbb P^1}$.

Consider the two isotropic foliations of $\tilde {dS}_2$. Each leaf of
each foliation has an $\alpha$-limit in $R_-$ and an $\omega$-limit on
$R_+$. Moreover every point of $R_-$ (resp. $R_+$) is an
$\alpha$-limit (resp. $\omega$-limit) of exactly one leaf of each
foliation.  Thus two identifications arise, 
$\iota_L:\iota_R:R_-\rightarrow R_+$, 
%%that turn to be projective equivalence.  These
that turn out to be projective equivalences.  These
identifications are natural in the following sense. For each
$\gamma\in\mathrm{Isom}(\tilde {dS}_2)$ we have
\[
   \iota_L\circ\gamma=\gamma\circ\iota_L\qquad\iota_R\circ\gamma=
   \gamma\circ\iota_R
\]

Now the developing map
\[
   d:\tilde D\rightarrow\tilde{dS}_2
\]
extends to a map
\[
  d:\tilde D\cup\tilde C_-\cup\tilde C_+\rightarrow \tilde{dS}_2\cup
  R_-\cup R_+
 \]
 and the restriction of $d$ to $\tilde C_-$ (resp. $\tilde C_+$) is
%% the developing map of $C_-$ and $C_+$.  In particular such
%% restriction are bijective. Thus the identification
 the developing map of $C_-$ and $C_+$.  In particular those
 restrictions are bijective. Thus the identification
 $\iota_L:R_-\rightarrow R_+$ induces an identification
 $\iota_L:\tilde C_-\rightarrow\tilde C_+$, that, by naturality,
 produces an identification between $C_-$ and $C_+$.

 Assume now that $C_\pm$ is elliptic of angle $\theta$.
 Notice that
 \[e(D)=e(\Sigma)\setminus (I^+(p_0)\cup I^-(p_0))
 \]
Thus, to prove that $e(D)$ embeds in $\mathcal P_\theta$ it is
sufficient to prove that $D$ is isometric to the de Sitter part of the
HS sphere $\Sigma_\theta$ that is the link of a singular point of
$\mathcal P_\theta$. Such de Sitter surface is the quotient of
$\tilde{dS}_2$ under an elliptic transformation of $\tilde{SO}(2,1)$
of angle $\theta$.

%%So the statement is equivalent to prove that  the developing map
So the statement is equivalent to proving that  the developing map
\[
  d:\tilde D\rightarrow \tilde{dS_2}
\]
is a diffeomorphism.

 Since $d$ is a local isometry, the map $d$ sends the double foliation
 of $\tilde D$ to the double foliation of $\tilde{dS_2}$. Given a point
 $p$, consider the $\omega$-limit points of the two leaves passing
 through $p$, say $r_L(p)$ and $r_R(p)$. They are sent by $d$ to the
 $\omega$-limit points of the leaves passing through $d(p)$.

Since the restriction of $d$ on $C_+$ is injective, in order to prove
%%that $d$ is injective it is sufficient to show that points $r_L(p)$
that $d$ is injective it is sufficient to showing that points $r_L(p)$
and $r_R(p)$ determines $p$. This is equivalent to show that a left
and right leaves of $\tilde D$ meet at most in one point. By
contradiction assume that there exist two leaves $l$ and $r$ meeting
%%twice. By standard argument, there is an embedded disk $\Delta$ in
twice. By standard arguments, there is an embedded disk $\Delta$ in
$\tilde D$ such that $\partial\Delta=l_1\cup r_1$, where $l_1$
(resp. $r_1$) is a segment on $l$ (resp. $r$). Notice that each leaf
$l'$ of the left foliation passing through $\Delta$ meets
$\partial\Delta$ at least twice. Since it cannot meet $l_1$, it
follows that it meets $r_1$ at least twice. Thus there is a left leaf
that is tangent to a right leaf. Since in the model such leaves do not
exist, we get a contradiction. Thus the injectivity is proved.

To prove that $d$ is surjective notice that $d(\tilde D)$ is a simply
connected open domain of $\tilde{dS}_2$, such that
\begin{itemize}
\item
$R_-$ and $R_+$ are in its closure;
\item
$d(\tilde D)$ is invariant under the holonomy $\gamma$ of $D$ that acts like a translation
\end{itemize}
Since $d(\tilde D)$ is connected, it contains an arc, $c_0$, joining
$R_-$ to $R_+$.  By the invariance for the holonomy, we see that it
contains infinitely many arcs $c_k=\gamma^k(c_0)$ for $k\in\mathbb
Z$. Now given any point $p$ in $\tilde{dS}_2$, we can find $k>>0$ such
that $p$ is contained in the rectangle of $d(\tilde D)$ bounded by
$c_{-k}$ and $c_k$. Since the boundary of such a rectangle is
contained in $d(\tilde D)$ that is simply connected, the whole
rectangle is contained in $d(\tilde D)$.
\end{proof}

\begin{defi}
Let $M$ be a singular spacetime homeomorphic to $S\times\mathbb R$
%%and $p\in M$. A neighbourhood $U$ of $p$ is said cyclindrical if
and let $p\in M$. A neighbourhood $U$ of $p$ is said to be {\bf cylindrical} if
\begin{itemize}
\item $U$ is topologically a ball;
\item $\partial_\pm C:=\partial U\cap I^\pm(p)$ is a spacelike disk;
%%\item there are two disjoint closed spacelike slices $S_-, S_+\cong S$
\item there are two disjoint closed spacelike slices $S_-, S_+$ homeomorphic to $S$
  such that $S_-\subset I^-(S_+)$ and $I^\pm(p)\cap \partial U=
  \partial_\pm C$.
\end{itemize}
  \end{defi}

\begin{remark}
\-
\begin{itemize}
\item
If a spacelike slice through $p$ exists then
cylindrical neighbourhoods form a fundamental family of neighbourhoods.
\item
There is an open retract $M'$ of $M$ whose boundary is $S_-\cup S_+$.
\end{itemize}
\end{remark}

\begin{cor}
Let $\Sigma$ be a HS-sphere as in Lemma \ref{surgery-key}.  Given an
AdS spacetime $M$ homeomorphic to $S\times\mathbb R$ containing a
particle of angle $\theta$, let us fix a point $p$ on it and suppose
that a spacelike slice through $p$ exists. There is a cylindrical
neighborhood $C$ of $p$ and a cylindrical neighbourhood $C_0$ of the
interaction point $p_0$ in $e(\Sigma)$ such that $C\setminus (
I^+(p)\cup I^-(p))$ is isometric to $C_0\setminus(I^+(p_0)\cup
I^-(p_0)$.
\end{cor}

Take an open deformation retract $M'\subset M$ with spacelike boundary
such that $\partial_\pm C\subset\partial M'$.  Thus let us glue
$M'\setminus ( I^+(p)\cup I^-(p))$ and $C_0$ by identifying
$C\setminus (I^+(p)\cup I^-(p))$ to $C_0\cap e(D)$. In this way we get
a spacetime $\hat M$ homeomorphic to $\Sigma\times\mathbb R$ with an
interaction point modelled on $e(\Sigma)$. We say that $\hat M$ is
obtained by a surgery on $M$.

%% next sentence has been added.
The following proposition is a kind of converse to the previous construction.

\begin{prop} \label{pr:example}
Let $\hat M$ be a spacetime with conical singularities homeomorphic to
$S\times\mathbb R$ containing only one interaction between
%%particles. Suppose moreover that the interaction point is modelled on
particles. Suppose moreover that a neighbourhood of the interaction point is isometric to
an open subset in 
$e(\Sigma)$, where $\Sigma$ is a HS-surface as in
%%Lemma~\ref{surgery-key}.  Then, up to shrinking $M$, it is obtained by
Lemma~\ref{surgery-key}.  Then a subset of $M$ is obtained by
a surgery on a spacetime without interaction.
\end{prop}

\begin{proof}
Let $p_0$ be the interaction point. There is an HS-sphere $\Sigma$ as
in Lemma~\ref{surgery-key} such that a neighborhood of $p_0$ is
isometric to a neighborhood of the vertex of $e(\Sigma)$. In
particular there is a small cylindrical neighborhood $C_0$ around
$p_0$. According Lemma~\ref{surgery-key}, for a suitable cylindrical
neighborhood $C$ of a singular point $p$ in $\mathcal P_\theta$ we
have
\[
    C\setminus (I^+(p)\cup I^-(p))\cong C_0\setminus(I^+(p_0)\cup I^-(p_0))
 \]
 Taking the retract $M'$ of $M$ such that $\partial_\pm C_0$ is in the
 boundary of $M'$, the spacetime $M'\setminus(I^+(p_0)\cup I^-(p_0))$
 can be glued to $C$ via the above identification. We get a spacetime
%% $M^0$ with only a singular line. Clearly the surgery on $M^0$ of
 $M^0$ with only one singular line. Clearly the surgery on $M^0$ of
 $C_0$ produces $M'$.
  \end{proof}

  \subsection{Spacetimes containing black holes singularities}
  \label{sub.bhfrancesco}

%%  In this section we will describe a class of spacetimes containing
  In this section we describe a class of spacetimes containing
  black holes singularities.

  First, consider two hyperbolic transformations $\gamma_1,\gamma_2\in
  PSL(2,\mathbb R)$ with the same translation length. There are
  exactly $2$ spacelike geodesics $l_1,l_2$ in $AdS_3$ that are left
  invariant under the action of $(\gamma_1,\gamma_2)\in PSL(2,\mathbb
  R)\times PSL(2,\mathbb R)=Isom(AdS_3)$.

  Moreover we can suppose that points of $l_1$ are fixed by
%%  $(\gamma_1,\gamma_2)$ whereas it acts by translation on $l_2$. The
  $(\gamma_1,\gamma_2)$ whereas it acts by pure translation on $l_2$. 
%% next sentence added.
$l_1$ and $l_2$ are then dual lines, that is, any point in $l_1$ can be
joined to any point in $l_2$ by a time-like segment of length $\pi/2$.
The union of the timelike segments with past end-point on $l_2$ and
  future end-point on $l_1$ is a domain $\Omega_0$ in $AdS_3$
  invariant under $(\gamma_1,\gamma_2)$.  The action of
  $(\gamma_1,\gamma_2)$ on $\Omega_0$ is proper and free and the
  quotient $M_0(\gamma_1,\gamma_2)=\Omega_0/(\gamma_1,\gamma_2)$ is a
  spacetime homeomorphic to $S^1\times\mathbb R^2$.

%%%% we could add a drawing here !

%%  We state there exists a spacetime with singularities 
  We state that there exists a spacetime with singularities 
  $\hat M_0(\gamma_1,\gamma_2)$ such that $M_0(\gamma_1,\gamma_2)$ is
  isometric to the regular part of $\hat M_0(\gamma_1,\gamma_2)$ and
  it contains a future black hole singularity.
%% We define
 Define
 \[
 \hat M_0(\gamma_1,\gamma_2)= (\Omega_0\cup l_1)/(\gamma_1,\gamma_2)
 \]

 To show that $l_1$ is a future black hole singularity, let us
 consider an alternative description of $\hat M_0(\gamma_1,\gamma_2)$.
 Notice that a fundamental domain in $\Omega_0\cup l_1$ for the action
 of $(\gamma_1,\gamma_2)$ can be constructed as follows. Take on $l_2$
 a point $z_0$ and put $z_1=(\gamma_1,\gamma_2)z_0$.  Then consider
 the domain $P$ that is the union of timelike geodesic joining a point
 on the segment $[z_0,z_1]\subset l_2$ to a point on $l_1$. $P$ is
 clearly a fundamental domain for the action with two timelike
 faces. $\hat M_0(\gamma_1,\gamma_2)$ is obtained by gluing the faces
 of $P$.

Now let us fix a surface $\Sigma$ with some boundary component and
%%genus negative Euler characteristic. Consider on $\Sigma$ two
negative Euler characteristic. Consider on $\Sigma$ two
hyperbolic metrics $\mu_l$ and $\mu_r$ with geodesic boundary such
that each boundary component has the same length with respect to those
metrics.

Let $h_l,h_r:\pi_1(\Sigma)\rightarrow PSL(2,\mathbb R)$ be the
corresponding holonomy representations.  The pair
%%$(h_1,h_2):\pi_1(\Sigma)\rightarrow PSL(2,\mathbb R)\times
$(h_l,h_r):\pi_1(\Sigma)\rightarrow PSL(2,\mathbb R)\times
PSL(2,\mathbb R)$ induces an isometric action of $\pi_1(\Sigma)$ on
$AdS_3$.

In \cite{barbtz1,barbtz2, bks} it is proved that there exists a convex
domain $\Omega$ in $AdS_3$ invariant under the action of
$\pi_1(\Sigma)$ and the quotient $M=\Omega/\pi_1(\Sigma)$ is a
strongly causal manifold homeomorphic to $\Sigma\times\mathbb R$. We
refer to those papers for the description of the geometry of $M$.

\begin{prop}
There exists a manifold with singularities $\hat M$ such that
\begin{enumerate}
\item The regular part of $\hat M$ is $M$.
\item There is a future black hole singularity and a past black hole
  singularity for each boundary component of $M$.
\end{enumerate}
 \end{prop}
 \begin{proof}
 %We refer to the description of the domain $\Omega$ given in~\ref{th,bsk}.  

Let $c\in\pi_1(\Sigma)$ be a loop representing a
   boundary component of $\Sigma$ and let $\gamma_1=h_L(c)$,
   $\gamma_2=h_R(c)$.

 By hypothesis, the translation lengths of $\gamma_1$ and $\gamma_2$
 are equal, so there are geodesics $l_1$ and $l_2$ as above. Moreover
 the geodesic $l_2$ is contained in $\Omega$ and is in the boundary of
 the convex core $K$ of $\Omega$. By~\cite{bks,bb_hand}, there exists
 a face $F$ of the past boundary of $K$ that contains $l_2$. The dual
 point of such a face, say $p$, lies in $l_1$. Moreover a component of
 $l_1\setminus\{p\}$ contains points dual to some support planes of
 the convex core containing $l_2$. Thus there is a ray $r=r(c)$ in
 $l_1$ with vertex at $p$ contained in $\partial_+\Omega$ (and
 similarly there is a ray $r_-=r_-(c)$ contained in
 $l_1\cap\partial_-\Omega$).

 Now let $U(c)$ be the set of timelike segments in $\Omega$ with past
 end-point in $l_2$ and future end-point in $r(c)$. Clearly
 $U(c)\subset\Omega(\gamma_1,\gamma_2)$.
%% Now the stabilizer of $U(c)$ in $\pi_1(\Sigma)$ is the group
The stabilizer of $U(c)$ in $\pi_1(\Sigma)$ is the group
 generated by $(\gamma_1,\gamma_2)$. Moreover we have
 \begin{itemize}
 \item
 for some $a\in\pi_1(\Sigma)$ we have $a\cdot U(c)=U(aca^{-1})$,
 \item
%% if $d$ is another peripheral loop $U(c)\cap U(d)=\varnothing$.
 if $d$ is another peripheral loop, $U(c)\cap U(d)=\varnothing$.
 \end{itemize}

 So if we  put
 \[
%% \hat M=(\Omega\cup\bigcup r(c)\cup\bigcup r_-(c))/(\pi_1(\Sigma))
 \hat M=(\Omega\cup\bigcup r(c)\cup\bigcup r_-(c))/\pi_1(\Sigma)
 \]
%% a neighbourhood of $r(c)$ in $\hat M$ is isometric to a neighbourhood
 then a neighbourhood of $r(c)$ in $\hat M$ is isometric to a neighbourhood
%% of $l_1$ in $M(\gamma_1,\gamma_2)$ and thus is a black hole
 of $l_1$ in $M(\gamma_1,\gamma_2)$, and is thus a black hole
 singularity (and analogously $r_-(c)$ is a white hole singularity).
 \end{proof}

%%%% this concludes the case where $\gamma_1,\gamma_2$ have the same 
%%%% translation lengths, but what about different lengths ??

 \subsection{Surgery on spacetimes containing black holes singularities}

 Now we illustrate how to get spacetimes $\cong S\times\mathbb R$
 containing two particles that collide producing a black hole
 singularity.  Such examples are obtained by a surgery operation
 similar to that implemented in Section~\ref{surg:sec}. The main
 difference with that case is that the boundary of these spacetimes is
 not spacelike.

 Let $M$ be a spacetime $\cong S\times\mathbb R$ containing a
 black-hole singularity $l$ of mass $m$ and fix a point $p\in l$. Let
 us consider a HS-surface $\Sigma$ containing a black hole singularity
 $p_0$ of mass $m$ and two elliptic singularities $q_1,q_2$. A small
 disk $\Delta_0$ around $p_0$ is isomorphic to a small disk $\Delta$
 in the link of the point $p\in l$.

 Let $B$ be a ball around $p$ and $B_\Delta$ be the intersection of
 $B$ with the union of segments starting from $p$ with velocity in
 $\Delta$.  Clearly $B_\Delta$ embeds in $e(\Sigma)$, moreover there
 exists a small disk $B_0$ around the vertex of $e(\Sigma)$ such that
%% $e(\Delta_0)\cap B_0$ is equal to the image of $B_\Delta$ in $B_0$.
 $e(\Delta_0)\cap B_0$ is isometric to the image of $B_\Delta$ in $B_0$.

%% Now we have that $\Delta'=\partial B\setminus B_\Delta$ is a disk in
 Now $\Delta'=\partial B\setminus B_\Delta$ is a disk in
 $M$. So there exists a topological surface $S_0$ in $M$ such that
 \begin{itemize}
 \item $S_0$ contains $\Delta'$;
 \item $S_0\cap B=\varnothing$;
%% \item $M\setminus S_0$ is the union of two copies of $S\times\mathbb R$
 \item $M\setminus S_0$ is the union of two copies of $S\times\mathbb R$.
 \end{itemize}
 Notice that we do not require $S_0$ to be spacelike.

%% Now, let $M_1$ be the component of $M\setminus S_0$ that contains
 Let $M_1$ be the component of $M\setminus S_0$ that contains
 $B$.  Consider the spacetime $\hat M$ obtained by gluing
 $M_1\setminus(B\setminus B_\Delta)$ to $B_0$ identifying $B_\Delta$
 to its image in $B_0$.
Clearly $\hat M$ contains two particles that collide giving a BH
singularity and topologically $\hat M\cong S\times\mathbb R$.

\section{Stereographic pictures}

In this section we extend to manifolds with interacting particles 
the result of Mess \cite{mess}, who showed how
one can associate to a GHM AdS manifold two Riemann surfaces. 
This correspondence between couples of hyperbolic surfaces and
3D AdS manifolds was already extended to manifolds with particles
(with angles less than $\pi$, and therefore no interaction) in 
\cite{cone}, each 3-manifold then yields two hyperbolic surfaces
with cone singularities. Here we consider 3-manifolds with 
interacting particles of positive mass -- but no tachyon or other 
exotic particle -- and show that one can associated to it a
series of ``stereographic pictures'', each one made of two 
hyperbolic surfaces with cone singularities. In the next section
we show that this ``stereographic movie'' defines a local parameterization 
of the moduli space of 3D metrics by the 2D pictures. 

\subsection{The left and right flat connections}

The constructions described below in this section can be understood
in a fairly simple manner through two connections on the space of
unit timelike tangent vectors on an AdS 3-manifold. In this subsection
we consider an AdS manifold $M$, which could for instance be the regular
part of an AdS manifold with particles. 

\subsubsection{The left and right connections}

\begin{defi}
We call $T^{1,t}M$ the bundle of unit
timelike vectors on $M$, and we define the left and right connections
on $T^{1,t}M$ as follows: if $u$ is a section of $T^{1,t}M$ and $x\in TM$
then 
$$ D^l_xu = \nabla_xu + u\times x~ , ~ ~ D^r_xu = \nabla_xu - u\times x~, $$
where $\nabla$ is the Levi-Civita connection of $M$.
\end{defi}

Here $\times$ is 
the cross-product in $AdS^3$ -- it can be defined by $(x\times y)^*=*(x^*\wedge
y^*)$, where $x^*$ is the 1-form dual to $x$ for the AdS metric and $*$ is the
Hodge star operator.

\begin{lemma}
$D^l$ and $D^r$ are flat connections.
\end{lemma}

\begin{proof}
Let $x,y$ be vector fields on $M$, and let $u$ be a section of $T^{1,t}M$.
Then the curvature tensor of $D^l$ is equal to
\begin{eqnarray*}
R^l_{x,y}u & = & D^l_xD^l_y u - D^l_yD^l_x u - D^l_{[x,y]} u \\
& = & \nabla_x(D^l_yu) + x\times (D^l_yu) - \nabla_y(D^l_xu) - y\times (D^l_xu) -
\nabla_{[x,y]}u - [x,y]\times u \\
& = & \nabla_x\nabla_yu + \nabla_x(y\times u) + x\times \nabla_yu + x\times (y\times u) \\
& & - \nabla_y\nabla_xu - \nabla_y(x\times u) - y\times \nabla_xu - y\times (x\times u) \\
& & - \nabla_{[x,y]}u - [x,y]\times u \\
& = & R_{x,y}u + (\nabla_xy)\times u - (\nabla_yx)\times u - [x,y]\times u 
+  x\times (y\times u) - y\times (x\times u) \\
& = &  R_{x,y}u + x\times (y\times u) - y\times (x\times u)~.
\end{eqnarray*}
A direct computation, using the fact that $M$ has constant curvature $-1$,
shows that this last term vanishes.
The same proof can be applied to $D^r$.
\end{proof}

Recall that we denote by $\isomz$ the 
isometry group of the quotient of $AdS^3$ by the antipodal map, and it 
is isomorphic to $PSL(2,\R)\times PSL(2,\R)$.

Suppose now that $M$ is time-oriented. Note that the fiber of 
$T^{1,t}M$ at each point is isometric to the hyperbolic plane $\HH^2$. 
Therefore, the flat connections $D^l$ and $D^r$ lead to the definition of two 
holonomy representations $\rho_l,\rho_r:\pi(M)\rightarrow PSL(2,\R)$
from the fundamental group of the AdS manifold $M$ to the isometry group
of $\HH^2$. In addition we also have a holonomy representation 
$\rho:\pi_1(M)\rightarrow \isomz$. 

\begin{lemma} \label{lm:rho}
$\rho=(\rho_l,\rho_r)$ under the canonical identification of
$\isomz$ with $PSL(2,\R)\times PSL(2,\R)$.
\end{lemma}

\begin{proof}
Let $\gamma:[0,1]\rightarrow M$ be a closed curve in $M$, i.e., with 
$\gamma(0)=\gamma(1)$. We can lift it to a segment $\gammab:[0,1]\rightarrow AdS^3$ 
in $AdS^3$. There is then a natural isometry $\phi_\gamma:T_{\gammab(0)}AdS^3\rightarrow 
T_{\gammab(1)}AdS^3$ obtained by compositions of the projections
$T_{\gammab(0)}AdS^3\rightarrow T_{\gamma(0)}M\rightarrow T_{\gammab(1)}AdS^3$. By
definition, $\phi_\gamma$ is the restriction to $T_{\gammab(0)}AdS^3$ of the holonomy
of $M$ on $\gamma$.
 
To prove the lemma, we need to check that the identifications between 
$T^{1,t}_{\gammab(0)}AdS^3$ and $T^{1,t}_{\gammab(1)}AdS^3$ provided by the
parallel transport for $D^l$ (resp. $D^r$) along $\gammab$ is identical
to the identification obtained between the same spaces which goes through
the intersection of the boundary at infinity of $T^{1,t}_{\gammab(0)}AdS^3$ and 
$T^{1,t}_{\gammab(1)}AdS^3$ with the family of lines $\cL_l$ (resp. $\cL_r$). 
Note that there is a natural identification of $T^{1,t}_{\gammab(0)}AdS^3$
with the plane dual to $\gammab(0)$ in $AdS_{3,+}$ (the space of points connected
to $\gammab(0)$ by a time-like geodesic segment of length $\pi/2$), which is
a totally geodesic space-like plane.

So we consider the following setup: in addition to $\gammab$, we choose
a unit timelike vector field $x$ along $\gammab$, and a unit space-like
vector field $v$ orthogonal to $x$ along $\gammab$, considered as a vector
tangent to $T^{1,t}_{\gammab(s)}AdS^3$ at $x$. To each $s\in [0,1]$
we then associate the point $y(s)$ in the boundary at infinity of 
$T^{1,t}_{\gammab(s)}AdS^3$ which is the endpoint of the geodesic ray starting from 
$x(s)$ with velocity $v(s)$, and the line $l_l(s)$ in $\cL_l$ (resp. $l_r(s)$
in $\cL_r$) containing
that point (using the natural embedding of  $T^{1,t}_{\gammab(s)}AdS^3$ in $AdS^3$).

The boundary at infinity of $AdS^3$ is canonically identified with the projectivization
of this cone. In this model, $y(s)$ corresponds
to the line generated by $\gammab(s)+v(s)$ in the light-cone of $0$ in $\R^{2,2}$,
$$ C=\{ x\in \R^{2,2}~|~ \langle x,x\rangle=0\}~. $$
The lines in $\dr_\infty AdS^3$ correspond to 2-dimensional planes
containing the origin contained in $C$. There are two such planes containing 
$\vect(\gamma(s)+v(s))$, they can be described using the vector $w(s)=x(s)\times v(s)$ as:
$$ \vect(\gamma(s)+v(s), x(s)-w(s))~,~~\vect(\gamma(s)+v(s), x(s)+w(s))~. $$
Those planes correspond to $l_l(s)$ and $l_r(s)$, respectively. 

It remains to show that, if $x$ and $v$ are parallel for $D^l$ (resp. $D^r$) then
$l_l(s)$ (resp. $l_r(s)$) is constant. One way to prove this is to show that
if $x,v$ are parallel for $D^l$ then $\gammab(s)+v(s)$ remains in the same 2-dimensional
plane in $C$. To do this computation note that the derivative of $v(s)$ as a 
vector in the flat space $\R^{2,2}$ is equal to $v'(s)=\nabla_{\gammab'(s)}v(s) +
\langle v(s),\gammab'(s)\rangle \gammab(s)$, because $AdS^3$ is umbilic in $\R^{2,2}$.
Therefore, if $v(s)$ is parallel for $D^l$, then
$$ \gammab'(s)+v'(s) = \gammab'(s) - \gammab'(s)\times v(s) + 
\langle v(s),\gammab'(s)\rangle \gammab(s)~. $$
To check that this is contained in $\vect(l_l(s))$ we check 3 cases, using
the basic properties of the cross product $\times$:
\begin{itemize}
\item $\gammab'(s)=v(s)$: then $\gammab'(s)+v'(s) = v(s) + \gammab(s)$,
\item $\gammab'(s)=w(s)$: then $\gammab'(s)+v'(s)=w(s) - w(s)\times v(s) = w(s)-x(s)$,
\item $\gammab'(s)=x(s)$: then $\gammab'(s)+v'(s)=x(s) - x(s)\times v(s) = x(s)-w(s)$.
\end{itemize}
The general case follows by linearity, and $l_l(s)$ remains constant. 
The same argument works, with minors differences in the signs, when $v(s)$ is parallel
for $D^r$ and for $l_r(s)$. 
\end{proof}

\subsubsection{The left and right metrics}

The flat connections $D^l$ and $D^r$ can be used to define two natural
degenerate metrics on the bundle of unit timelike vectors on $M$. 
We use a natural identification based on the Levi-Civita connection $\nabla$ of $M$:
$$ \forall (x,v)\in T^{1,t}M, T_{(x,v)}(T^{1,t}M)\simeq T_xM\times v^\perp\subset T_xM\times T_xM~. $$
In this identification, given $v'\in v^\perp$, the vector $(0,v')$, considered as a vector
in $T(T^{1,t}M)$, corresponds to a ``vertical'' vector, fixing $x$ and moving
$v$ according to $v'$.
And, given $x'\in T_xM$, the vector $(x',0)$, considered as a vector
in $T(T^{1,t}M)$, corresponds to a ``horizontal'' vector, 
moving $x$ according to $x'$ while doing a parallel
transport of $v$ (for the connection $\nabla$).

\begin{defi}
We call $M_l$ and $M_r$ the two degenerate metrics (everywhere of rank 2)
defined on $T^{1,t}M$ as follows. Let $(x,v)\in T^{1,t}M$, so that $x\in M$
and $v\in T_xM$ is a unit timelike vector. Let $(x',v')\in T(T^{1,t}M)$,
so that $x'\in T_xM$ and $v'\in v^\perp\subset T_xM$. Then
$$ M_l((x',v'),(x',v')) = \| v' + v\times x'\|^2~,~~ 
    M_r((x',v'),(x',v')) = \| v' - v\times x'\|^2~. $$
\end{defi}

By construction $M_l$ and $M_r$ are symmetric bilinear forms on the 
tangent space of $T^{1,t}M$, and they are semi-positive, of rank $2$ at
every point. 

\begin{defi}
$G(M)$ is the space of timelike maximal geodesics in $M$.
\end{defi}

\begin{lemma}
$M_l$ and $M_r$ are invariant under the
geodesic flow of $M$.
\end{lemma}

\begin{proof}
Let $x\in M$, let $v\in T^{1,t}_xM$ be a timelike vector, and let $\gamma:I\rightarrow M$ be
the timelike geodesic (parameterized by arc-length) with $\gamma(0)=x$ and $\gamma'(0)=v$.
Let $(x',v')\in T_{(x,v)}T^{1,t}M$, it extends under the geodesic flow on $M$ as a Jacobi
field along $\gamma$ and is therefore of the form
$$ x'(\gamma(t))=u_0 + u_1\cos(t)+u_2 \sin(t)~, ~~ v'(\gamma(t))= -u_1 \sin(t) + u_2 \cos(t)~, $$
with $u_0,u_1,u_2$ parallel vector fields along $\gamma$, where 
$u_0$ is parallel to $\gamma$ and $u_1, u_2$ are orthogonal to $\gamma$.
Therefore,
\begin{eqnarray*}
\| v'+v\times x'\|^2 & = & \| -u_1 \sin(t) + u_2 \cos(t) + v\times (u_1\cos(t)+u_2 \sin(t))\|^2 \\
& = & \| \sin(t) (-u_1+v\times u_2)+\cos(t) (u_2+v\times u_1)\|^2 \\
& = & \| \sin(t) (-u_1+v\times u_2)+\cos(t) v\times (-u_1+v\times u_2)\|^2 \\
& = & \| -u_1+v\times u_2\|^2~. 
\end{eqnarray*}
This shows that $M_l$ is invariant under the geodesic flow. The same argument, with a few
sign changes, shows the same result for $M_r$.
\end{proof}

\begin{defi}
$m_l$ and $m_r$ are the degenerate metrics induced by $M_l$ and
$M_r$, respectively, on $G(M)$.  
\end{defi}

\begin{lemma}
$m_l\oplus m_r$ is locally isometric to $\HH^2\times \HH^2$.
\end{lemma}

\begin{proof}
Choose $x\in M$ and $v\in T^{1,t}_xM$, let $P$ be the totally geodesic
plane in $M$ containing $x$ and orthogonal to $v$ (considered in a 
neighborhood of $x$). 
Consider the section $s^l_v$ (resp. $s^r_v$) of $T^{1,t}M$ defined in the
neighborhood of $x$ such that $s^l_v(x)=v$ (resp. $s^r_v(x)=v$), and which is
parallel for $D^l$ (resp. $D^r$). The integral curves of $s^l_v$ (resp. $s^r_v$)
form a two-dimensional subspace of the space of timelike geodesics near $x$,
parameterized by the intersection point with $P$.
Clearly $m_r$ (resp. $m_l$) vanishes on this subspace, while the restriction
to this subspace of $m_l$ (resp. $m_r$) is twice the pull-back of the metric
on $P$ by the map sending a geodesic to its intersection with $P$. This proves
that, in the neighborhood in $G(M)$ of the geodesic parallel to $v$ at $x$, 
$m_l\oplus m_r$ is isometric to $\HH^2\times \HH^2$.
\end{proof}

Note that there is another possible way to obtain the same hyperbolic metrics
$m_l$ and $m_r$ in another way, using the identification of $H^2\times H^2$
with $PSL(2,\R)\times PSL(2,\R)/O(2)\times O(2)$. We do not elaborate on this point
here since it appears more convenient to use local considerations.

\subsection{Transverse vector fields}

The construction of the left and right hyperbolic metric of a 
3-manifold is based on the use of a special class of 
surfaces, endowed with a unit timelike vector field behaving 
well enough, in particular with respect to the singularities.

\begin{defi}
Let $S\subset M$ be a space-like surface, and let $u$ be a field of timelike
unit vectors defined along $S$. It is {\bf transverse} if 
\begin{itemize}
\item $u$ is parallel to the particles at their intersection points with
$S$,
\item for all
$x\in S$, the maps $v\mapsto D^l_vu$ and $v\mapsto D^r_vu$ have rank
$2$.
\end{itemize}
\end{defi}

It is not essential to suppose that $S$ is space-like, and the weaker
topological assumption that $S$ is isotopic in $M$ to a space-like
surface would be sufficient. The definition is restricted to space-like
surface for simplicity.

\begin{defi} \label{df:610}
Let $S\subset M$ be a surface, and let $u$ be a transverse vector
field on $S$. Let $\delta:S\rightarrow G(M)$ be the map sending a
point $x\in S$ to the timelike geodesic parallel to $u$ at $x$. 
We call $\mu_l:=\rho^*m_l$ and $\mu_r:=\rho^*m_r$.
\end{defi}

Since $D^l$ and $D^r$ are flat connections, $\mu_l$ and $\mu_r$ 
can be defined locally as pull-backs of the metric on the hyperbolic
plane -- the space of unit timelike vectors at a point -- by some
map. Therefore $\mu_l$ and $\mu_r$ are hyperbolic metrics (with
cone singularities at the intersections of $S$ with the singularities). 
They are called the {\it left and right hyperbolic metrics on $S$}.

From this point on we suppose that $M$ is time-orientable, so that 
$u$ can be chosen among positively oriented unit timelike vector 
fields.

\begin{lemma} \label{lm:surface}
Given $S$, $\mu_l,\mu_r$ do not depend (up to isotopy) on the choice of the transverse
vector field $u$. Moreover, $\mu_l, \mu_r$ do not change (again up to 
isotopy) if $S$ is 
replaced by another surfaces isotopic to it.
\end{lemma}

The proof uses a basic statement on hyperbolic surfaces with
cone singularities. Although this result might be well known,
we provide a simple proof for completeness.

\begin{lemma} \label{pr:hypersurface}
A closed hyperbolic surface with cone singularities of angle less than
$2\pi$ is uniquely determined by its holonomy. 
\end{lemma}

\begin{proof}
Let $S$ be a closed surface with marked points $x_1,\cdots, x_n$,
let $\theta_1, \cdots, \theta_n\in (0,2\pi)$, and let $h_0$ be a 
hyperbolic metric on $S$ with cone singularities of angles $\theta_i$
on the $x_i$, $1\leq i\leq n$. 

Given angles $\theta'_1,\cdots, \theta'_n$ compatible with the
Gauss-Bonnet relation 
$$ \sum_i (2\pi-\theta'_i) > -2\pi \chi(S)~, $$
there is a unique metric on $S$, conformal to $h_0$, with cone 
singularity of angle $\theta'_i$ at $x_i$, see \cite{troyanov}. 
In particular there is a unique one-parameter family $(h_t)_{t\in [0,1]}$
of metrics conformal to $h_0$ such that $h_t$ has a cone singularity
of angle $\theta_i/(1+t)$ at $x_i$. It is known that a hyperbolic
metric with cone singularities of angle less than $\pi$ is determined
by its holonomy \cite{dryden-parlier}. Therefore the lemma will be proved if
we can show that, given a hyperbolic metric $h$ with cone singularities
of angle less than $2\pi$, the first-order deformations of $h$ are
uniquely determined by the first-order variations of its holonomy. 

Let $h$ be such a metric (with cone singularities of angle less than
$2\pi$). There exists a triangulation $T$ of $S$ which has as vertices 
exactly the singular points of $h$ and for which the edges are realized
by pairwise disjoint (except at the endpoints) geodesic segments of $h$. 

The end of the proof rests on the following two facts. 
\begin{enumerate}
\item The cone angle at a singular point $v$ is determined by the holonomy
of the oriented boundary of a small neighborhood of $v$.
\item The length of a geodesic segment $e$ with endpoints two singular 
points $v_-, v_+$ is determined by the cone angles at $v_-, v_+$ and
by the holonomy of the oriented neighborhood of $e$.
\end{enumerate}
The first point is self-evident. For the second point let $\theta_-,
\theta_+$ be the cone angles at $v_-, v_+$, let $l$ be the length of
$e$, and let $\alpha$ be the
holonomy of a small oriented neighborhood of $e$. Consider the
hyperbolic surface with boundary obtained by cutting $(S,h)$ open
along $e$. There is a unique way to extend $h$ by gluing at the
boundary either a disk with one cone singularity or a cylinder, in
such a way that the singularities at the boundary disappear. The
hyperbolic surface which is glued can be either:
\begin{itemize}
\item the double of a hyperbolic triangle, with one edge of length
$l$ and the angles of its endpoints equal to $\pi-\theta_-/2$ and
$\pi-\theta_+/2$. Then the angle at the third vertex is determined
by $\alpha$, and the three angles determine $l$, because a hyperbolic
triangle is completely determined by its angles.
\item the double of a hyperbolic triangle with one ideal vertex,
again with one edge of length $l$ and with the same angles
as in the previous case. Then $l$ is determined by the angles
of this triangle, because a hyperbolic triangle with one ideal
vertex is completely determined by its two non-zero angles. 
\item the double of a hyperideal vertex, i.e., of a hyperbolic
domain with boundary made of one geodesic segment and two
geodesic rays, and of infinite area. Then the angle at the
hyperideal vertex corresponds to (twice) the distance between the 
geodesic rays in the boundary of this domain, and it is determined
by the holonomy $\alpha$; in turns it determines the length $l$,
because a hyperbolic triangle with one hyperideal vertex is
completely determined by its angles. 
\end{itemize}

In all cases the holonomy $\alpha$ determines $l$, and this proves
that the lengths for $h$ of the edges of the triangulation $T$
are determined by the holonomy of $h$. For hyperbolic metrics close
to $h$, $T$ remains a triangulation with edges which can be 
realized by geodesic segments. Therefore hyperbolic metrics close
to $h$ are uniquely determined by their holonomy, this concludes
the proof of the lemma.
\end{proof}

\begin{proof}[Proof of Lemma \ref{lm:surface}]
For the first point consider another transverse vector field $u'$
on $S$, and let $\mu'_l,\mu'_r$ be the hyperbolic metrics defined 
on $S$ by the choice of $u'$ as a transverse vector field. 
Let $\gamma$ be a closed curve on the complement of the singular
points in $S$. The holonomy of $\mu'_l$ (resp. $\mu'_r$) on 
$\gamma$ is equal to the holonomy
of $D^l$ (resp. $D^r$) acting on the hyperbolic plane, identified
with the space of oriented timelike unit vectors at a point of 
$S$. So $\mu_l$ and $\mu'_l$ (resp. $\mu_r$ and $\mu'_r$) have the
same holonomy, so that they are isotopic by Lemma \ref{lm:surface}.

The same argument can be used to prove the second part of the lemma.
Let $\gamma_1$ be a closed curve on $S_1$ which does not intersect the 
singular set of $M$, and let $\gamma_2$ be a closed curve on $S_2$ 
which is isotopic to $\gamma_1$ in the regular set of $M$. The holonomy
of $M$ on $\gamma_1$, $\rho(\gamma_1)$, is equal to the holonomy of 
$M$ on $\gamma_2$, $\rho(\gamma_2)$. But $\rho=(\rho_l,\rho_r)$ by
Lemma \ref{lm:rho}, and $\rho_l,\rho_r$ are the holonomy representations
of the left and right hyperbolic metrics on $S_1$ and on $S_2$ by
Lemma \ref{lm:rho}. Therefore, $(S_1,\mu_l)$ has the same holonomy
of $(S_2,\mu_l)$, and $(S_1,\mu_r)$ has the same holonomy as 
$(S_2,\mu_r)$. The result therefore follows by Proposition 
\ref{pr:hypersurface}.
\end{proof}

Note that a weaker version of this proposition is proved
as \cite[Lemma 5.16]{minsurf} by a different argument.
The notations $\mu_l,\mu_r$ used here are the same as in 
\cite{cone}, while the same metrics
appeared in \cite{minsurf} under the notations $I^*_\pm$. Those 
metrics already appeared, although implicitly only, in Mess' paper
\cite{mess}. There one considers globally hyperbolic AdS manifolds,
which are the quotient of a maximal convex subset $\Omega$ of $AdS^3$ by 
a surface group $\Gamma$ acting by isometries on $\Omega$. The identification
of $\isomz$ with $PSL(2,\R)\times PSL(2,\R)$ then determines
two representations of $\Gamma$ in $PSL(2,\R)$, and it is proved in
\cite{mess} that those representations have maximal Euler 
number, so that they define hyperbolic metrics. It is proved in 
\cite{minsurf} that those two hyperbolic metrics correspond precisely
to the left and right metrics considered here.

\subsubsection{Note.} 

The reason this section is limited to manifolds with particles
-- rather than more generally with interacting singularities -- is that we
do not at the moment have good analogs of those surfaces with transverse
vector fields when other singularities, e.g. tachyons, are present.

\subsubsection{Example: good surfaces}

The previous construction admits a simple special case, when the
timelike vector field is orthogonal to the surface (which then
has to be space-like). 

\begin{defi}
Let $M$ be an AdS manifold with interacting particles. Let $S$ be a 
smooth space-like surface. $S$ is a {\it good} surface if:
\begin{itemize}
\item it does not contain any interaction point,
\item it is orthogonal to the particles,
\item its induced metric has curvature $K<0$.
\end{itemize}
\end{defi}

Note that, given a good surface $S$, one can consider the equidistant
surfaces $S_r$ at distance $r$ on both side. For $r$ small enough 
(for instance, if $S$ has principal curvature at most $1$,
when $r\in (-\pi/4,\pi/4)$), $S_r$ is a smooth
surface, and it is also good. So from one good surface one gets a
foliation of a neighborhood by good surfaces. 

The key property of good surfaces is that their unit normal vector
field is a transverse vector field, according to the definition given
above. This simplifies the picture since the left and right metrics
are defined only in terms of the surface, without reference to a
vector field. However the construction of a good surface seems to
be quite delicate in some cases, so that working with a more general
surface along with a transverse vector field is simpler.

\begin{lemma}
Let $S$ be a good surface, let $u$ be the unit normal vector field on
$S$, then $u$ is a transverse vector field.
\end{lemma}

\begin{proof}
Let $x\in S$ and let $v\in T_xS$. By definition, 
$$ D^l_v u = \nabla_vu + u\times v = -Bv + Jv~, $$   
$$ D^r_v u = \nabla_vu - u\times v = -Bv - Jv~, $$
where $B$ is the shape operator of $S$ ad $J$ is the complex structure 
of the induced metric on $S$. If $S$ is a good surface then its 
induced metric has curvature $K<0$. But
$\det(-B\pm J) = \det(B)+1 = -K$, so that $D^l_vu$ and $D^r_vu$
never vanish for $v\neq 0$. This means precisely that $u$ is a
transverse vector field.
\end{proof}

\subsubsection{Example.} 

Let $s_0$ be a space-like segment in $AdS^3$ of length 
$l>0$. Let $d_0, d_1$ be disjoint timelike lines containing the endpoints
of $s_0$ and orthogonal to $s$, 
chosen so that the angle between the (timelike) plane $P_0$
containing $s_0, d_0$ and the (timelike) plane $P_1$ containing $s_0$ 
and $d_1$ is equal to $\theta$, for some $\theta>0$. 
Let $W_0$ (resp. $W_1$) be wedges with 
axis $d_0$ (resp. $d_1$) not intersecting $s_0$ or $d_1$ (resp $d_0$).

Let $M_{\theta}$ be the space obtained from $AdS^3\setminus W_0\cup W_1$ by
gluing isometrically the two half-planes in the boundary of $W_0$
(resp. $W_1$), and let $M_{ex}:=M_\theta$ for $\theta=l$. 
Then $M_{ex}$ does not contain any good surface, or 
even any surface with a transverse vector field. 

To see this remark that for $\theta<l$, $M_{\theta}$ does contain a 
space-like surface with a transverse vector field (we leave the
construction to the interested reader) but with a left hyperbolic
metric, say $\mu_l(\theta)$, which as two cone singularities which 
``collide'' as $\theta\rightarrow l$. (This can be seen easily by 
taking a surface which contains $s_0$.) If $M_{ex}$ admitted a
surface with a transverse vector field, it could have only one cone
singularity (as is seen by considering the limit $M_\theta\rightarrow 
M_{ex}$, this is impossible.

Note that $M_{ex}$ is obviously not globally hyperbolic, and it contains no
closed space-like surface, it was chosen for its simplicity.

\subsection{From space-like surfaces to diffeomorphisms}

There are simple relations between on the one hand good space-like surfaces in
$M$ (or good equidistant foliations, or more generally surfaces with a transverse
vector field) and on the other hand diffeomorphisms between 
the corresponding left and right hyperbolic 
metrics. We consider good equidistant foliations first, then the more general
case of surfaces with a transverse vector field.

\begin{prop}\label{pr:area-preserving}
{~}
\begin{enumerate}
\item Let $(S_t)_{t\in I}$ be an equidistant foliation of a domain $\Omega\subset M$ by
good surfaces. We can associate to $(S_t)_{t\in I}$ an area-preserving diffeomorphism
$\phi:(S,\mu_l)\rightarrow (S,\mu_r)$ which fixes the singular points.
\item Given $\phi$ and $(S_t)_{t\in I}$ as above, any small deformation $\phi'$ of 
$\phi$, among area-preserving diffeomorphisms preserving the singular points, 
is obtained from a unique equidistant foliation $(S'_t)_{t\in I}$ by good
surfaces.
\end{enumerate}
\end{prop}

\begin{proof}
Let $t\in I$. By definition of $\mu_l$ and $\mu_r$, those metrics can be defined
in terms of the induced metric $I_t$ on $S_t$ as
$$ \mu_l(v,v)=I(-Bv+Jv, -Bv+Jv)~, ~~ \mu_r(v,v)=I(-Bv-Jv, -Bv-Jv)~. $$
The identity map on $S_t$ defines a diffeomorphism $\phi$ between $(S,\mu_l)$ and
$(S,\mu_r)$. The construction of $\mu_l$ and $\mu_r$ above -- through the 
metrics $M_l$ and $M_r$ on the space of timelike geodesics -- shows that this
diffeomorphism does not depend on $t$, since different values of $t$ determine 
the same surface in the space $G(M)$ of timelike geodesics in $M$. Moreover, 
checking that $\phi$ is area-preserving amounts to checking that
$$ \det(-B-J)=\det(-B+J)~, $$
or (multiplying by $\det(J)=1$) that 
$$ \det(I-JB)=\det(I+JB)~. $$
This is always the case because $B$ is self-adjoint for $I$ so that $\tr(JB)=0$.

For the second point, let $\phi'$ be a small deformation of $\phi$, it corresponds,
through the description above of $G(M)$ as a product, to a surface in the space of 
timelike geodesics in  $M$, close to the family of timelike geodesics normal to
$(S_t)_{t\in I}$ so that those timelike geodesics foliate a domain 
$\Omega'\subset \Omega$. Let $n$ be the unit future-oriented timelike vector
field on $\Omega'$ parallel to those geodesics, let $B:n^\perp\rightarrow n^\perp$
be defined, at each point $x\in \Omega'$, by $Bx = D_xn$, then the argument in 
the first part of the proof shows that the differential of $\phi'$ can be
written as $(-B-J)\circ (-B+J)^{-1}$. The fact that $\phi'$ is area-preserving
then means that $\tr(JB)=0$ and therefore that $B$ is self-adjoint for the
induced metric $I$ on $n^\perp$. If $x,y$ are two vector fields on $\Omega'$
both normal to $n$ then
$$ \langle [x,y],n\rangle = \langle D_xy-D_yx,n\rangle = - \langle y, D_xn\rangle
+ \langle x, D_yn\rangle = -\langle y,Bx\rangle + \langle x, By\rangle = 0~. $$
This means that the distribution of planes normal to $n$ is integrable, and point
(2) follows.
\end{proof}

When one considers only a surface with a transverse vector field -- a more general
case since we have seen that the  unit normal vector field on a good surface is
transverse -- the diffeomorphism is still well-defined, but it is not area-preserving
any more. This statement is weaker, but more flexible, than Proposition 
\ref{pr:area-preserving}.

\begin{prop} \label{pr:diffeo}
\begin{enumerate}
\item Let $S\subset M$ be a space-like surface, and let $u$ be a transverse vector field
on $S$. There is a diffeomorphism $\phi:(S,\mu_l)\rightarrow (S,\mu_r)$ associated
to $S$. 
\item $S$ and $u$ are uniquely determined by $\phi$ in the sense that, if $S'$
is another space-like surface such that the domain bounded by $S$ and $S'$ 
contains no interaction of singularities, and if $u'$  is a transverse vector
field on $S'$ so that $(S', u')$ is associated to $\phi$, then all timelike
geodesics which are parallel to $u$ at their intersection point with $S$ are also
parallel to $u'$ at their intersection point with $S'$.  
\end{enumerate}
\end{prop}

\begin{proof}
The existence of the diffeomorphism $\phi$ is proved like the first point of Proposition 
\ref{pr:area-preserving}: since $u$ is a transverse vector field, 
$S$ inherits from $u$ two metrics $\mu_l$ and $\mu_r$, and the identity defines
a diffeomorphism between $(S,\mu_l)$ and $(S,\mu_r)$. 

The second point then follows from the considerations of the previous sub-section,
because $\phi$ determines a map from $S$ to the space $G(M)$ of timelike geodesics in 
$M$, which determines $\mu_l$ and $\mu_r$. 
If two couples $(S,u)$ and $(S', u')$ determine the same map to $G(M)$
then all timelike geodesics which are parallel to $u$ at their intersection point
with $S$ have to be parallel to $u'$ at their intersection point with $S'$.
\end{proof}

\subsection{Surgeries at collisions} \label{ssc:74}

We now wish to understand how the left and right hyperbolic metrics
change when an interaction occurs. We use the term ``collision'' here,
rather than the more general ``interaction'' seen above, since we
only consider particles, rather than more exotic singularities like
tachyons or black holes. 

\subsubsection{Good spacial slices}

The first step in understanding AdS manifolds with colliding particles
is to define more easily understandable pieces. 

\begin{defi}
Let $M$ be an AdS manifold with colliding particles. A {\bf spacial slice}
in $M$ is a subset $\Omega$ such that 
\begin{itemize}
\item there exists a closed surfaces $S$
with marked points $x_1,\cdots, x_n$ and a homeomorphism 
$\phi:S\times [0,1]\rightarrow \Omega$,
\item $\phi$ sends $\{ x_1,\cdots,x_n\}\times [0,1]$ to the singular set of 
$\Omega$,
\item $\phi(S\times \{ 0\})$ and $\phi(S\times \{ 1\})$ are space-like 
surfaces orthogonal to the singular set of $M$.
\end{itemize}
$\Omega$ is a {\bf good} spacial slice if in addition
\begin{itemize}
\item it contains a space-like surface with a transverse vector field.
\end{itemize}
\end{defi}

It is useful to note that Lemma \ref{lm:surface}, along with its proof, applies 
also to surfaces with boundary, with a transverse vector field, 
embedded in a good spacial slice. Such surfaces 
determine the holonomy of the restriction of the left and right metrics
to surfaces with boundary, as explained in the following remark.

\begin{remark} \label{rk:surface}
Let $\Omega$ be a good spacial slice, let $D\subset \Omega$ be a space-like
surface with boundary, and let $u'$ be a transverse vector field on $D$.
Then $u'$ determines a left and a right hyperbolic metric, $\mu'_l, \mu'_r$
on $D$, as for closed surfaces above. Moreover for any closed curve
$\gamma$ contained in $D$, the holonomy of $\mu'_l$ and $\mu'_r$ 
on $\gamma$ is equal to the holonomy on $\gamma$ of the left and the
right hyperbolic metrics of $\Omega$.
\end{remark}

The proof is a direct consequence of the arguments used in the proof of Lemma
\ref{lm:surface}.

The second step is to identify in the left and right hyperbolic metrics
the patterns characterizing particles which are bound to collide.

\begin{prop} \label{pr:disk}
Let $\Omega\subset M$ be a good spacial slice, which contains particles
(singular timelike segments) $p_1, \cdots, p_n$ intersecting at a
point $c\in \dr\Omega$ -- we can suppose for instance that $c$ is in
the future boundary of $\Omega$. Let $S$ be a space-like surface in $\Omega$ and
let $u$ be a transverse vector field on $S$.
There is then a topological disk $D$ embedded in
the past hyperbolic component of the link of $c$, and topological
disks $D_l, D_r$ embedded in $(S, \mu_l)$ and in $(S, \mu_r)$, 
such that $D$ is isometric to $D_l$ and to $D_r$. 
\end{prop}

\begin{proof}
Let $C$ be the intersection with $\Omega$ of a cone with vertex at $c$,
containing the $p_i, 1\leq i\leq n$. (Note that $C$ can be chosen as 
embedded in $\Omega$ since the $p_i$ do not intersect in $\Omega$.)
Let $D':=C\cap S$, and let $u'$ be the timelike unit vector field defined 
on $D'$ as the direction of the geodesic line in $C$ going through $c$.

$u'$ is then a transverse vector field on $D$. To prove this let $x\in D'$,
let $v\in T_xD'$, and let $w$ be the orthogonal projection of $v$ on ${u'}^\perp$,
the plane perpendicular to $u'$ in $T_x\Omega$. If $r$ is the length of the
timelike geodesic segment going from $x$ to $c$ through $C$, 
$$ D^l_vu' = D^l_wu' = -\cotan(r) w + u'\times w~, $$
so that $D^l_vu'\neq 0$ if $v\neq 0$. The same argument (up to a sign) shows
that $D^r_vu'\neq 0$ if $v\neq 0$. This means precisely that $u'$ is transverse.

Now let $\mu'_l, \mu'_r$ be the left and right hyperbolic metrics defined 
on $D'$ by $u'$. Let $D_L$ be the restriction of the past link of $c$ to $C$,
that is, the space of geodesic rays with endpoints at $c$ contained in $C$.
$D_L$ is naturally endowed with a (non-complete) hyperbolic metric, and
there is a canonical projection $\rho:D'\rightarrow D_L$ sending a point 
in $D'$ to the geodesic ray starting from $c$ containing it. If $x\in D$
is at distance $r$ from $c$, and if $v\in T_xS$ and $w$ is again the
orthogonal projection of $v$ on ${u'}^\perp$, then the geometry of 
the past cone of a point in $AdS^3$ as a warped product shows that
$$ \| w\|^2 = \sin^2(r) \| \rho_*(v) \|^2~, $$
while the definition of $D^l$ indicates that
$$ \| D^l_vu'\|^2 = \| D_vu' + u'\times v \|^2 =
\| D_wu' + u'\times w \|^2 = $$ 
$$ = \| -\cotan(r) w + u'\times w \|^2
= \cotan^2(r) \| w\|^2 + \| w\|^2 = \frac{1}{\sin^2r} \| w\|^2~. $$
It is therefore clear that $\rho$ is an isometry from $(D',\mu'_l)$ to
$D_L$. 

The same argument works (with only one sign change) for $\mu'_r$,
and the proposition follows.
\end{proof}

\subsubsection{Example}

If only two particles, $p_1$ and $p_2$, collide, the corresponding
cone points are at the same distance in the left and right hyperbolic
metric of $\Omega$; more precisely, there are two segments of the same
length, one in the left and one in the right hyperbolic metric of 
$\Omega$, joining the cone points corresponding to $p_1$ and to $p_2$.
Moreover the length of those segments is equal to the ``angle'' 
between $p_1$ and $p_2$ at $c$, i.e., to the distance between the
corresponding points in the link of $c$.

\subsubsection{Surgeries on the left and right metrics} \label{sssc:surgeries}

It is now possible to consider in details how the left and right metrics
change when a collision occurs. In this regard it is necessary to define
a more precise setting. 
Let $\Omega$ be an AdS manifold with interacting particles, 
containing exactly one collision point $m$, which is the future
endpoint of $n$ particles $s_1, \cdots, s_n$ and the past
endpoint of $m$ particles $s'_1, \cdots, s'_m$. Let $\theta_1,\cdots,
\theta_n$ be the cone singularities at the $s_i$, and let let
$\theta'_1,\cdots, \theta'_m$ be the cone singularities at the $s'_j$. 

Suppose that $\Omega$ is the union of two good space-like slices
$\Omega_-$ and $\Omega_+$, with disjoint interior, with $\Omega_-$ containing
$s_1, \cdots, s_n$ and $\Omega_+$ containing $s'_1, \cdots, s'_m$. We call
$S_-$ a space-like surface in $\Omega_-$ with a transverse vector field $u_-$,
and $S_+$ a space-like surface in $\Omega_+$ with a transverse vector field 
$u_+$. Let 
$\mu^\pm_l,\mu^\pm_r$ be the left and right hyperbolic metrics defined on 
$S_\pm$ by $u'_\pm$.

\begin{lemma} \label{lm:collision}
Under those conditions, 
\begin{enumerate}
\item $\mu^-_l,\mu^-_r$ have (among others) $n$ cone singularities of angle 
$\theta_1,\cdots, \theta_n$ at points $x_i^{-,l}, x_i^{-,r}, 1\leq i\leq n$
respectively,
\item there exist embedded disks $D^-_l$ and $D^-_r$ in $S_-$, containing the
$x_i^{-,l}$ and the $x_i^{-,r}$ respectively, such that $(D^-_l,\mu^-_l)$ is isometric
to $(D^-_r,\mu^-_r)$,
\item $\mu^+_l$ and $\mu^+_r$ are obtained from $\mu^-_l$
and $\mu^-_r$, respectively, by removing $D^-_l$ and $D^-_r$, respectively, 
and gluing instead a topological disk $D$ with a hyperbolic metric with 
cone singularities $h_+$
such that some neighborhood of the boundary of $(D,h_+)$ is isometric to
some neighborhood of the boundary in $(D^-_l,\mu^-_l)$ (or, equivalently,
in $(D^-_r, \mu^-_r)$. 
\end{enumerate}
The metric $(D,h_+)$ has $m$ cone singularities of angle $\theta'_1,\cdots,
\theta'_m$. 
\end{lemma}

\begin{proof}
The first point is a basic property of the left and right metrics associated
to a transverse vector field, see above. 
The second point is a reformulation of Proposition \ref{pr:disk}. Moreover,
the same proposition shows -- when applied to $\Omega_+$ with the time
reversed -- that there exist disks $D^+_l\subset (S_+,\mu^+_l)$ and 
$D^+_r\subset (S_+,D^+_r)$ which are isometric.

It follows from Lemma \ref{lm:rho} that the holonomy for $\mu^-_l$ 
of any closed curve $\gamma_-$ in the complement of $S_-\setminus D^-_l$ is equal to the
holonomy of $\mu^+_l$ on a closed curve $\gamma_+$ which is isotopic to $\gamma_-$
in the regular set of $M$. This implies that the restriction of $\mu^+_l$ on the
complement of some disk $D^+_l$ is isometric to the restriction of $\mu^+_-$ to the
complement of $D^-_l$. The same description applies to $\mu^+_r$ in the complement
of some disk $D^+_r$. The restriction of $\mu^+_l$ to $D^+_l$ is isometric to 
the restriction of $\mu^+_r$ to $S^+_r$ by the first part of the proof, this 
terminates the proof.
\end{proof}

%\subsection{Existence of good space-times with collisions}

%[Construction of example with one interaction, replacing one particle
%with two, with small angle, so that there is a surface with a transverse
%vector field both before and after the collision]

\subsection{Transverse vector fields after a collision}

It might be interesting to remark that the description made in 
subsection \ref{ssc:74} of the surgery on the left and right hyperbolic metrics corresponding
to a collision only holds -- and actually only makes sense -- if there is a
space-like surface with a transverse vector field both before and after the
collision. However the existence of such a surface before the collision does
not ensure the existence of one after the collision, even for simple
collisions.

\subsubsection{Collisions with no transverse vector}

A simple example of such a phenomenon can be obtained by an extension of 
the example given above of an AdS space with two particles containing no
space-like surface with a transverse vector field. Consider the space
$M_\theta$ described in that example, with $\theta<l$, so that $M_\theta$
contains a space-like surface with a transverse vector field. This space
has two cone singularities, $d_0$ and $d_1$, each containing one of the
endpoints of $s_0$. It is now possible to perform on this space a simple
surgery as described in section 5, replacing the part of $d_1$ in
the past of its intersection with $s_0$ by two cone singularities, 
say $d_2$ and $d_3$, 
intersecting at the endpoint of $s_0$. This can be done in such a way
that the angle between the plane containing $s_0$ and $d_2$ and the
plane containing $s_0$ and $d_0$, is equal to $l$. The argument given
above for $M_{ex}$ then shows that there is no space-like surface
with a transverse vector field in a spacial slice before the collision.

\subsubsection{Collisions with transverse vector fields} \label{sssc:transverse}

Consider now the example constructed in subsections \ref{ssc:explicit} 
and \ref{surg:sec} by surgery
on a globally hyperbolic space-time with non-interacting massive particles,
in particular in Proposition \ref{pr:example}. Consider a space-time $M$ as
constructed in that proposition, with only one collision point $p$.
Let $\Sigma_-,\Sigma_+$ be closed, space-like surfaces in $M$ such that 
$p$ is in the future of $\Sigma_-$ and in the past of $\Sigma_+$. 
We suppose moreover that there is a tranverse vector field $v_+$ on $\Sigma_+$:
this is the case for the examples considered in subsection \ref{ssc:explicit},
which contain space-like surfaces which are ``almost'' totally geodesic (with arbitrarily
small curvature) after the collision point.

As seen above, $v_+$ corresponds to a diffeomorphism $\phi_+$ 
between the left and right hyperbolic metrics, $\mu^+_l$ and $\mu^+_r$, on $\Sigma_+$,
sending the singular points in $\mu^+_l$ to the corresponding singular points in $\mu^+_r$.
The effect of the collision on the left and right metrics is to replace two 
isometric disks $(D^+_l,\mu^+_l)$ and $(D^r_r,\mu^+_r)$, respectively in 
$(\Sigma^+,\mu^+_l)$ and in $(\Sigma^+, \mu^+_r)$, by two other isometric disks.
Clearly there exists a diffeomorphism $\phi_-$ between 
$(\Sigma^-,\mu^-_l)$ and in $(\Sigma^-, \mu^-_r)$, that is, after this surgery
is made. This shows that there is a transverse vector field along $\Sigma^-$.

\subsection{The graph of interactions}

The previous subsection contains a description of the kind of surgery on
the left and right hyperbolic metrics corresponding to a collision of
particles. Here a more global description is sought, and we will associate
to an AdS manifold with colliding particles a graph describing the relation
between the different spacial slices. In all this part we fix an AdS manifold
with colliding particles, $M$.

We need some simple definitions. First, we define an isotopy in $M$ as a
homeomorphism $\phi:M\rightarrow M$ such that there exists a one-parameter
family $(\phi_t)_{t\in [0,1]}$ of homeomorphisms from $M$ to $M$, with 
$\phi_0=Id, \phi_1=\phi$, such that each $\phi_t$ sends the singular set of
$M$ to itself. Two domains in $M$ are isotopic if there is an isotopy sending
one to the other. 

Let $\Omega, \Omega'$ be two spacial slices in $M$. They are {\bf equivalent}  
if each space-like surface in $\Omega$ is isotopic to a space-like surface
in $\Omega'$. Note that this clearly defines an equivalence relation on 
the spacial slices in $M$. 

\begin{defi} \label{df:good}
$M$ is a {\bf good} AdS manifold with colliding particles if any spacial
slice in $M$ is equivalent to a good spacial slice.   
\end{defi}

Clearly if two good spacial slices are equivalent then their holonomies
are the same, so that their left and right hyperbolic metrics are the
same by Proposition \ref{pr:hypersurface}. 

The content of \ref{sssc:transverse} shows that some of the examples
constructed by Proposition \ref{pr:example} are indeed good AdS manifolds
with colliding particles.

\begin{defi}
Let $\Omega_-$ and $\Omega_+$ be two spacial slices in $M$.
They are {\bf adjacent} if the union of the compact connected
components of the complement of the interior of 
$\Omega_-\cup \Omega_+$ in $M$ contains exactly one collision.
We will say that $\Omega_-$ is {\bf anterior} to $\Omega_+$ if 
this collision is in the future of $\Omega_-$ and in the
past of $\Omega_+$.
\end{defi}

Note that this relation is compatible with the equivalence
relation on the spacial slices: if $\Omega_-$ is adjacent to
$\Omega_+$ and $\Omega'_-$ (resp. $\Omega'_+$) is equivalent
to $\Omega_-$ (resp. $\Omega_+$) then $\Omega'_-$ is adjacent
to $\Omega'_+$. Moreover if $\Omega_-$ is anterior to $\Omega_+$
then $\Omega'_-$ is anterior to $\Omega'_+$. 

\begin{defi}
The {\bf graph of spacial slices} is the oriented graph associated
to a good AdS manifold with colliding particles $M$ in the following way.
\begin{itemize}
\item The vertices of $G$ correspond to the equivalence classes of
spacial slices in $M$.
\item Given two vertices $v_1, v_2$ of $G$, there is an edge between 
$v_1$ and $v_2$ if the corresponding spacial slices are adjacent.
\item This edge is oriented from $v_1$ to $v_2$ if the spacial
slice corresponding to $v_1$ is anterior to the spacial slice corresponding
to $v_2$.
\end{itemize}
\end{defi}

\subsection{The topological and geometric structure added to the graph of interactions}

Clearly the graph of spacial slice is not in general a tree -- there might
be several sequence of collisions leading from one spacial slice to
another one. A simple example is given in Figure \ref{fg:graph}, where the
graph of a manifold with colliding particles is shown together with a 
schematic picture of the collisions.

\begin{figure}[ht]
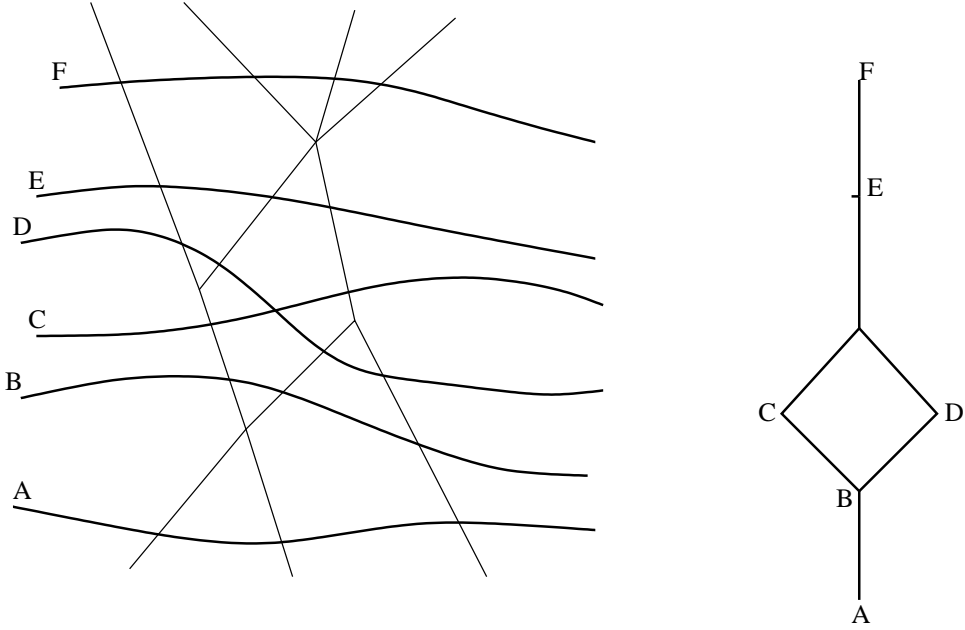

\begin{center}
\input graph.pstex_t
\end{center}
\caption{The graph of spacial slices.}
\label{fg:graph}
\end{figure}

The graph of spacial slices is clearly not sufficient to recover an AdS
manifold with colliding particles, additional data are needed. 

\begin{defi}
A {\bf topological data} associated to an oriented graph is the choice
\begin{itemize}
\item For each vertex $v$, of a closed surface $\Sigma_v$ with $n$ 
marked points $p_1, \cdots, p_n$, and a $n$-uple $\theta_v= 
(\theta_1,\cdots, \theta_n)\in (0,2\pi)^n$.
\item For each edge $e$ with vertices $e_-$ and $e_+$:
  \begin{enumerate}
  \item a homotopy class of disks $D_{e,+}\in \Sigma_{e_+}$, where the 
homotopies are in the complement of the marked points $p_i$ (or equivalently
they are homotopies of the complements of the $p_i$ in $\Sigma_v$),
  \item a homotopy class of disks $D_{e,-}\in \Sigma_{e_-}$, where again 
the homotopies fix the marked points,
  \item an isotopy class $i_e$ of homeomorphisms from 
$\Sigma_{e_-}\setminus D_{e,-}$ to $\Sigma_{e_+}\setminus D_{e,+}$. 
  \end{enumerate}
\end{itemize}
\end{defi}

\begin{defi}
A {\bf geometric data} associated to an oriented graph endowed with
a topological data is the choice, for each vertex $v$, of two 
hyperbolic metrics $\mu_l(v),\mu_r(v)$ on $\Sigma_v$, with a cone singularity 
of angle $\theta_i$ at $p_i$, so that, for each edge $e$ with endpoints $e_-$
and $e_+$:
\begin{itemize}
\item $\mu_l(e_-)$ and $\mu_r(e_-)$ coincide on $D_{e_-}$, while $\mu_l(e_+)$
and $\mu_r(e_+)$ coincide on $D_{e,+}$, 
\item $i_e$ is isotopic to an isometry between $(\Sigma_{e_-}\setminus D_{e,-},
\mu_l(e_-))$ and  $(\Sigma_{e_+}\setminus D_{e,+},
\mu_l(e_+))$, and to an isometry between $(\Sigma_{e_-}\setminus D_{e,-},
\mu_r(e_-))$ and $(\Sigma_{e_+}\setminus D_{e,+}, \mu_r(e_+))$.
\end{itemize}
\end{defi}

Given a good AdS space with colliding particles $M$ we can consider its graph
of collisions $\Gamma$, there is a natural topological and geometric structure associated
to $M$ on $\Gamma$. Given a vertex $v$ of the graph of collisions $\Gamma$, it
corresponds to a good space-like slice $S_v$ in $M$, and we take as $\Sigma_v$ a
space-like surface in $S_v$. The marked points correspond to the intersections 
of $\Sigma_v$ with the particles in $S_v$. By definition of a good spacial slice,
$S_v$ admits a transverse vector field, so one can define the left and right 
hyperbolic metrics $\mu_l(v)$ and $\mu_r(v)$ on $\Sigma_v$ through Definition 
\ref{df:610}. The fact that those two metrics are well defined follows from 
Lemma \ref{lm:surface}.

Now consider an edge of $\Gamma$, that, is a collision between particles. 
Let $e_-$ corresponds to the good spacial slice $\Omega_-$ in the past of the collision,
and $e_+$ to the good spacial slice $\Omega_+$ in the future of the collision. The existence
of a homotopy class of disks $D_-$ on $\Sigma_{e_-}$ on which 
$\mu_l(e_-)$ and $\mu_r(e_-)$ are isometric follows from Proposition \ref{pr:disk},
as well as the existence of a homotopy class of disks $D_+\subset \Sigma_{e_+}$
on which $\mu_l(e_+)$ and $\mu_r(e_+)$ are isometric. The existence of 
an isometry between $(\Sigma_{e_-}\setminus D_-,\mu_l(e_-))$ and 
$(\Sigma_{e_+}\setminus D_+,\mu_l(e_+))$ follows from Lemma \ref{lm:collision}, as
does the existence of an isometry between $(\Sigma_{e_-}\setminus D_-,\mu_r(e_-))$ and 
$(\Sigma_{e_+}\setminus D_+,\mu_r(e_+))$. Those isometries are in the isotopy class
of homeomorphisms between space-like surfaces in $\Omega_-$ and $\Omega_+$, respectively,
obtained by following a timelike vector field.

\section{Local parameterization of 3D metrics}

%% the next sentence has been added.
This section contains a local deformation result, Theorem \ref{tm:homeo}: 
it is shown that given
a globally hyperbolic space-time with interacting particles, its small
deformations are parameterized by small deformations of the geometric
data on its graph of interactions.

\subsection{Maximal good spacetimes}  

Let $M$ be a good spacetime with collision.  The oriented graph
associated to $M$ has a unique vertex $v_i$ with only outcoming edges
(the initial vertex) and a unique vertex $v_f$ with only incoming edges
(the final vertex).

Let us consider a surface with a transverse vector field
 $\Sigma_i$ (resp. $\Sigma_f$) in $M$
corresponding to $v_i$ ($v_f$). Let $M_i$ (resp. $M_f$) be the MGH
spacetime with particles and without collision, containing $\Sigma_i$
(resp. $\Sigma_f$).

The past of $\Sigma_i$ in $M$ can be thickened to a globally
hyperbolic spacetime containing $\Sigma_i$. This shows that
$I^-_M(\Sigma_i)$ isometrically embeds in $M_i$. In general its image
in $M_i$ is contained in the past of $\Sigma_i$ in $M_i$, but does not
coincide with it.

We say that $M$ is a $m$-spacetime, if the following conditions hold.

- $I^-_M(\Sigma_i)$ isometrically embeds in $M_i$ and coincides with
$I^-_{M_i}(\Sigma_i)$;

- $I^+_M(\Sigma_f)$ isometrically embeds in $M_f$ and coincides with
$I^+_{M_f}(\Sigma_f)$.

\begin{lemma}\label{mm:lem}
Every good spacetime embeds in a $m$-spacetime
\end{lemma}
\begin{proof}
Let $\Sigma_i$ (resp. $\Sigma_f$) be a spacelike surface corresponding to
$v_i$ (resp. $v_f$) and denote by $U_i$ (resp. $U_f$) the past
(resp. the future) of $\Sigma_i$ ($\Sigma_f$) in $M$.  Clearly $U_i$
embeds in $M_i$. Let $V_i$ be the past of $U_i$ in $M_i$
(and analogously let $V_f$ be the future of $U_f$ in $M_f$).

Then $V_i$ and $V_f$ can be glued to $M$ by identifying $U_i$ to its
image in $V_i$ and $U_i$ to its image in $U_f$.  The so obtained
spacetime, say $M'$, is clearly an $m$-spacetime.
\end{proof}

Let $M$ be a spacetime with collision and $T$ be the singular locus of $M$.
We say that $M$ is maximal, if every isometric embedding 
\[
    (M,T)\rightarrow (N,T')
\]
that restricted to $T$ is a bijection with the singular locus $T'$ of
%%$N$ , is actually an isometry.
$N$, is actually an isometry.

\begin{lemma}
Every $m$-spacetime is maximal. Every good spacetime isometrically
embeds in a unique $m$-spacetime.
\end{lemma}
\begin{proof}
We sketch the proof, leaving details to the reader. 

 Let $(M,T)$ be a good spacetime and $\pi:(M,T)\rightarrow(M_{max},
 T_{max})$ be the maximal extension constructed in Lemma \ref{mm:lem}.

We prove that given any isometric embedding
\[
    \iota:(M,T)\rightarrow (N,T')
\]
there is an embedding of $\pi':(N,T')\rightarrow (M_{max}, T_{max})$
such that $\pi=\pi'\circ\iota$.

Let $\Sigma_i,\Sigma_f$ be spacelike surfaces in $M$ as in Lemma
\ref{mm:lem}.  The embedding $\iota$ identifies $\Sigma_i$ and
$\Sigma_f$ with disjoint spacelike surfaces in $N$ (that with some
abuse we will still denote by $\Sigma_i,\Sigma_f$). Moreover
$\Sigma_i$ is in the past of $\Sigma_f$ in $N$.
 
Now the closure of the domains $\Omega=I^+_M(\Sigma_i)\cap
I^-_M(\Sigma_f)$ and $\Omega'=I^+_N(\Sigma_i)\cap I^-_N(\Sigma_f)$ are
both homeomorphic to $\Sigma\times[0,1]$ and $\iota$ sends
$\overline\Omega$ to $\overline\Omega'$ and $\partial\Omega$ onto
$\partial\Omega'$. A standard topological argument shows that
$\iota(\Omega)=\Omega'$.

Finally if $M_i$ (resp. $M_f$) denotes the MGH spacetime containing
$\Sigma_i$ ($\Sigma_f$), we have that $I^-_{N}(\Sigma_i)$
(resp. $I^+_N(\Sigma_f)$) embeds in $I^-_{M_i}(\Sigma_i)$
(resp. $I^+_{M_f}(\Sigma_f)$).

Thus we can construct the map $\pi':N\rightarrow M_{max}$ in such a way that
\begin{itemize}
\item on $I_N^+(\Sigma_i)\cap I_N^-(\Sigma_f)$ we have $\pi'=\iota^{-1}$;
\item on $I_N^-(\Sigma_i)$ it coincides with the embedding in
  $I^-_{M_i}(\Sigma_i)$;
\item on $I_N^+(\Sigma_f)$ it coincides with the embedding in
  $I^+_{M_f}(\Sigma_f)$.
\end{itemize}
\end{proof}

\subsection{Deformation of the graph of interactions}

Let us consider an $m$-spacetime $M=(\Sigma_g\times\mathbb R, \mu_0)$
%%and singular locus $T$.  
and singular locus $T$ made of colliding massive particles.
For every edge $e$ of $T$, let $\theta_e$ be
the corresponding cone angle, and denote by $\theta=(\theta_e)$ the
collection of all these angles. Moreover notice that the time orientation on $M$
induces on $T$ a natural orientation.

We consider singular AdS metrics $\mu$ on $\Sigma_{g}\times\mathbb R$ that
make $\Sigma_g\times\mathbb R$ an $m$-spacetime with singular locus
equal to $T$ such that
\begin{itemize}
%%\item each edge $e$ of $T$ corresponds to a particle with
\item each edge $e$ of $T$ corresponds to a massive particle with
cone angle $\theta_e$;
\item the orientation on $T$ induced by $\mu$ agrees with the
  orientation induced by $\mu_0$.
\end{itemize}

The following proposition ensures that the isometry type of an admissible metric is
determined by the isometry type of the complement of the singular locus.
Thus, in order to study the moduli space of admissible metrics, it will be sufficient to study
their restriction on the complement of the singular locus. The main advantage in considering
the AdS metrics on the complement that extends to admissible metrics, is that we can use the
theory of $(G,X)$-structures.

\begin{prop}\label{ext:prop}
Let $\mu,\nu$ be two singular metrics on $\Sigma\times\mathbb R$ 
with singular locus equal to $T$.

Then any isometry
\[
%%   \phi:(\Sigma\times\mathbb R\setminus
   \psi:(\Sigma\times\mathbb R\setminus
   T,\mu)\rightarrow(\Sigma\times\mathbb R\setminus T,\nu)
\]
%%extends to an isometry $\bar\phi:(\Sigma\times\mathbb
extends to an isometry $\bar\psi:(\Sigma\times\mathbb
R,\mu)\rightarrow(\Sigma\times\mathbb R,\nu)$.
\end{prop}

\begin{proof}
Let us take $p\in T$.  Consider some small $\mu$-geodesic
$c:[0,1]\rightarrow\Sigma\times\mathbb R$ such that $c(1)=p$ and
%%$c([0.1))\cap T=\emptyset$. If $c$ is small enough, we can find two
$c([0,1))\cap T=\emptyset$. If $c$ is small enough, we can find two
  points $r_-$ and $r_+$ in $\Sigma\times\mathbb R\setminus T$ such
  that $c[0,1)\subset I^+_\mu(r_-)\cap I^-_\mu(r_+)$.

Now let us consider the $\nu$-geodesic path $c'(t)=\psi(c(t))$ defined
in $[0,1)$.  Notice that $c'$ is an inextensible geodesic path in
  $\Sigma\times\mathbb R\setminus T$.

We know that $c'$ is contained in $I^+_\nu(\psi(r_-))\cap
I^-_\nu(\psi(r_+)$. Thus if $\Sigma_\pm$ is a Cauchy surface through
$r_\pm$, we have that $c'\subset I^+_\nu(\Sigma_-)\cap
I^-_\nu(\Sigma_+)$ that is a compact region in $\Sigma\times\mathbb
R$.
Thus $c'(t)$ converges to some point in $T$ as $t\rightarrow 1$.
We define $\hat c=\lim_{t\rightarrow 1}c'(1)$. 

To prove that $\psi$ can be extended on $T$ we have to check that
$\hat c$ only depends on the endpoint $p$ of $c$. In other words, if
$d$ is another geodesic arc ending at $p$, we have to prove that $\hat
c$ is equal to $\hat d=\lim_{t\rightarrow1}\psi\circ d(t)$.  By a
standard connectedness argument, there is no loss of generality if we
assume that $d$ is  close to $c$. In particular we may assume
that there exists the geodesic triangle $\Delta$ with vertices $c(0),
d(0), p$.

Consider now the $\mu$-geodesic segment in $\Delta$, say $I_t$, with
endpoints $c(t)$ and $d(t)$.  The image $\psi(I_t)$ is a
$\nu$-geodesic segment contained in $\Sigma\times\mathbb R\setminus T$. 
Arguing as above, we can prove that
all these segments $(\psi(I_t))_{t\in[0,1)}$ are contained in some
compact region of $\Sigma\times\mathbb R$.  Thus either they converge
to a point (that is the case $\hat c=\hat d$), or they converge to
some geodesic path in $T$ with endpoints $\hat d$ and $\hat c$.

On the other hand, the $\nu$-length of $\psi(I_t)$ goes to zero as
$t\rightarrow 1$.  Thus either $\psi(I_t)$ converges to a point or it
converges to a lightlike path. Since $T$ does not contain any
lightlike geodesic it follows that $\psi(I_t)$ must converge to a
point.  Thus $\hat d=\hat c$.

Finally we can define
\[
   \psi(p)=\hat c
\]
where $c$ is any $\mu$-geodesic segment with endpoint equal to $p$. 

Let us prove now that this extension is continuous.
For a sequence of points $p_n$ converging to $p\in T$ we have to 
check that $\psi(p_n)\rightarrow\psi(p)$.
%%We can reduce to consider  two cases
We can reduce to consider  two cases:
\begin{itemize}
\item $(p_n)$ is contained in $T$,
\item $(p_n)$ is contained in the complement of $T$.
\end{itemize}
In the former case we consider a point $q$ in the complement of $T$.
Let us consider the $\mu$-geodesic segment $c$ joining $q$ to $p$ and
%%the segment $c_n$ joining $q$ to $p_n$. Clearly for every $t\in (0,1]$ the points
the segments $c_n$ joining $q$ to $p_n$. Clearly for every $t\in (0,1]$ the points
$c_n(t)$ and $c(t)$
are timelike related and their Lorentzian distance converges to the
distance between $p_n$ and $p$ as $t\rightarrow 1$. On the other hand, since
$\psi(c_n(t))$ ($\psi(c(t))$) converges to $\psi(p_n)$ ($\psi(p)$) as
$t\rightarrow 1$, we can conclude that
\[
    d(\psi(p_n),\psi(p))=d(p_n,p)
\]
%%where $d$ denotes the Lorentzian distance. Clerly this equation implies that
where $d$ denotes the Lorentzian distance along $T$. Clerly this equation implies that
$\psi(p_n)\rightarrow\psi(p)$ as $n\rightarrow+\infty$.

%%Let us suppose now that $p_n$ is contained in the complement of $T$.
Let us suppose now that the points $p_n$ are contained in the complement of $T$.
%%We can take $r_+$ nad $r_-$ such that $p_n\in I^+_\mu(r_-)\cap I^-_\mu(r_+)$ for $n\geq n_0$.
We can take $r_+$ and $r_-$ such that $p_n\in I^+_\mu(r_-)\cap I^-_\mu(r_+)$ for $n\geq n_0$.
Thus the same argument used to define $\psi(p)$ shows  that $(\psi(p_n)) $ is contained in 
%%some compact subset of $\Sigma\times\mathbb R$. To conlude it is sufficient to prove that
some compact subset of $\Sigma\times\mathbb R$. To conclude it is sufficient to prove that
if $\psi(p_n)\rightarrow x$ then $x=\psi(p)$.
Clearly $x\in
T$. Moreover either $x$ coincides with $\psi(p)$ or there is a
%%gedoesic segment in $T$ connecting $x$ to $\psi(p)$. Since the length of such
%%gedoesic should be equal to the limit of $d(p_n,p)$, that is $0$, we
geodesic segment in $T$ connecting $x$ to $\psi(p)$. Since the length of this
geodesic should be equal to the limit of $d(p_n,p)$, that is $0$, we
conclude that $x=\psi(p)$.

Eventually we have  to check that the map $\psi$ is an isometry. Let us note
that $\psi$ induces a map
\[
\psi^\#:S_p\rightarrow S'_{\psi(p)}
\]
%%where $S_p$ and $S'_{\psi(p)}$ are respectly the link of $p$ with
where $S_p$ and $S'_{\psi(p)}$ are respectively the link of $p$ with
respect to $\mu$ and the link of $\psi(p)$ with respect to $\nu$.
Simply, if $c$ is the tangent vector of a geodesic arc $c$ at $p$, we
define $\psi^\#(v)=w$ where $w$ is the tangent vector to the arc
$\psi\circ c$ at $\psi(p)$.  Notice that $\psi$ is an isometry around
$p$ iff $\psi^\#$ is a $HS$-isomorphism.

Clearly $\psi^\#$ is bijective and is an isomorphism of $HS$ surfaces
far from the singular locus. On the other hand, since singularities
are contained in the hyperbolic region, a standard completeness
argument shows that it is a local isomorphism around singular points.
\end{proof}

Let $\mathrm{Diffeo}_0$ be the group of diffeomorphisms of the pair
$(\Sigma\times\mathbb R, T)$ whose restriction to $\Sigma\times\mathbb
R\setminus T$ is isotopic to the identity.  It is not difficult to
show that $\mathrm{Diffeo}_0$ acts on the space of admissible metrics
by pull-back.

Let us denote by $\Omega(g,T,\theta)$ the   quotient space: an element of
$\Omega(g, T,\theta)$ is  a singular metric with properties
described above, up to the action of $\mathrm{Diffeo}_0$.
Notice that the graph of collision and the corresponding
topological data of any structure in $\Omega(g,T,\theta)$ coincide
with those of $M$.

There is a natural forgetting map from $\Omega(g, T,\theta)$ to the set of
AdS structures on $\Sigma_g\times\mathbb R\setminus T$ up to
isotopy. 
From Proposition \ref{ext:prop}  this map is injective, so
$\Omega(g,T,\theta)$ can be identified
to a subset of anti de Sitter structures on $\Sigma\times\mathbb
R\setminus T$.  Thus $\Omega(g, T,\theta)$ inherits from this
structure space a natural topology (see \cite{epstein} for a
discussion on the topology of the space of $(G,X)$-structures on a
fixed manifold).

Let $X$ be the graph of collisions associated to $M$ equipped with the
topological data.  We denote by $\mathcal D(X)$ the set of the
\emph{admissible geometric data}, that are those geometric data which
are compatible with the topological data. More precisely an element of
$\mathcal D(X)$ is given by associating to every vertex $v$ two
hyperbolic metrics, say $\mu_L(v),\mu_R(v),$ on $\Sigma_v$ such that
%%the following conditions are satisfied
the following conditions are satisfied.
\begin{enumerate}

\item Every marked point $p_i\in\Sigma_v$ is a cone point of angle
  $\theta_i$ for both $\mu_l(v),\mu_r(v)$.
\item For each edge, $e$, with vertices $e_-, e_+$, there is a disk
  $\Delta^l_{e_-}$ in $\Sigma_v$ in the isotopy class of $D_{e_-}$ and
  an embedded disk $\Delta^l_{e_+}$ in $\Sigma_w$ in the isotopy class
  of $D_{e_+}$ such that there is an isometry $(\Sigma_{e_-}
  \setminus\Delta^l_{e_-},\mu_l(v))\rightarrow (\Sigma_{e_+}\setminus
  \Delta^l_{e_+},\mu_l(w))$ homotopic to $\alpha_e$.
\item For each edge, $e$, with vertices $e_-, e_+$, there is a disk
  $\Delta^r_{e_-}$ in $\Sigma_v$ in the isotopy class of $D_{e_-}$ and
  an embedded disk $\Delta^r_{e_+}$ in $\Sigma_w$ in the isotopy class
  of $D_{e_+}$ such that there is an isometry $(\Sigma_{e_-}
  \setminus\Delta^r_{e_-},\mu_l(v))\rightarrow (\Sigma_{e_+}\setminus
  \Delta^r_{e_+},\mu_l(w))$ homotopic to $\alpha_e$.
\end{enumerate} 

\begin{remark}
If $N\in\Omega(g,T,\theta)$, its geometric data is an element of
$\mathcal D(X)$.
\end{remark}

Notice that $\mathcal D(X)$ can be regarded as a subset of the product
%%of suitable Teichmuller spaces corresponding to vertices of $X$. Thus
of suitable Teichm\"uller spaces corresponding to vertices of $X$. Thus
%%$\mathcal D(X)$ is equipped with a natural topology Clearly the map
$\mathcal D(X)$ is equipped with a natural topology. Clearly the map
\[
  \Phi:\Omega(g,T,\theta)\rightarrow\mathcal D(X)
\]  
associating to each spacetime with collision the corresponding
geometric data is continuous.

%%The remaining part of this section is devoted to prove the following result. 
The remaining part of this section is devoted to proving the following result. 

\begin{theorem} \label{tm:homeo}
The map $\Phi$ is a local homeomorphism.
\end{theorem}

Let us illustrate the strategy to prove this theorem.

We consider representations $\rho: \pi_1(\Sigma_g\times\mathbb
R\setminus T)\rightarrow\mathrm{Isom}(AdS_3)$ such for every meridian
loop $c_e$ around an edge of $T$ we have that $\rho(c_e)$ is a pair of
elliptic transformations of angle $\theta_e$ in $PSL(2,\mathbb R)$ (we
are implicitly using the identification
%%$\mathrm{Isom}(AdS_3)=PSL(2,\mathbb R)\times PSL(2,\mathbb R)$).  Let
$\isomz=PSL(2,\mathbb R)\times PSL(2,\mathbb R)$).  Let
us denote by $\mathcal R$ the space of such representations up to
conjugation. Clearly the holonomy of a structure in $\Omega(g, T,
\theta)$ is an element of $\mathcal R$.

Now the proof is divided in two steps.

\begin{enumerate}
\item 
We show that the geometric data associated to $M$ determines the
holonomy. Moreover we will see that any admissible geometric data
determines a representation in $\mathcal R$.
\item
We show that the holonomy map is a local homeomorphism between
$\Omega(g,T,\theta)$ and $\mathcal R$.
\end{enumerate}

\subsection{From a deformation of the geometric data to a 
deformation of the holonomy}
Let $M$ be an $m$-spacetime with collision and $T$ be the singular
locus.  In the collision graph $X$ of $M$ consider a path starting
from the initial vertex and ending at the final vertex.  Let
$v_1,\ldots,v_n$ be the vertices meeting such a segment. For any $v_i$
let $\Sigma_i$ be the corresponding surface.

For any edge $e_i=[v_i,v_{i+1}]$, let $\Sigma_i=\Sigma_{v_i}$ and
$\Sigma_{i+1}=\Sigma_{v_{i+1}}$ be the corresponding spacial slices and
let $\Omega_i$ be the neighbourhood of collision point $p_i$ such that
$\Delta_{i_-}=\Omega_i\cap\Sigma_i$ and $\Delta_{i_+}=\Omega_i\cap\Sigma_{i+1}$.  
%%Consider eventually the identification
Consider the identification
\[
   \alpha_i:\Sigma_i\setminus \Delta_{i_-}\rightarrow
   \Sigma_{i+1}\setminus \Delta_{i_+}
 \]
  given by the flow of a  timelike vector field tangent to $\partial\Omega_i$.
  
 Let $\pi_1(M)$ denote the fundamental group of $\Sigma_g\times\mathbb
 R\setminus T$ and $\pi_1(\Sigma_i)$ denote the fundamental group of
 the punctured surface $\Sigma_i\setminus T$.  The following
 proposition allows to compute $\pi_1(M)$.

\begin{prop}\label{vk:prop}
The inclusion  homomorphism
\[
  i_*: \pi_1(\Sigma_1)*\pi_1(\Sigma_2)*\ldots*\pi_1(\Sigma_n)\rightarrow\pi_1(M)
\]
is surjective and $\ker i_*$ is the normal subgroup generated by
$\alpha_i(\gamma)\gamma^{-1}$ for every $i=1,\ldots,n-1$ and
%%$\gamma\in\pi_1(\Sigma_i\setminus\Delta_{i_-})$
$\gamma\in\pi_1(\Sigma_i\setminus\Delta_{i_-})$.
\end{prop}

\begin{proof}
First, let us consider the case $n=2$.  In this case we can cut
$\Delta_{1_-}$ from $\Sigma_1$ and glue a new disk to obtain a surface
$\Sigma$ that is isotopic to $\Sigma_2$.  Notice that $(\Sigma_1\cup
\Sigma)\setminus T$ is a deformation retract of $\Sigma_g\times\mathbb
R\setminus T$.  Thus by applying Van Kampen Theorem to $(\Sigma_1\cup
\Sigma)\setminus T$ we get the result.

Repeating inductively the argument one gets the general case.
\end{proof}

\begin{cor}
The holonomy of $M$ is determined by $X$.
\end{cor}

\begin{cor}
If $Y$ is an admissible geometric data then there is a representation
%%$\pi_1(M)\rightarrow Isom(AdS_3)=PSL(2,\mathbb R)\times PSL(2,\mathbb
$\pi_1(M)\rightarrow \isomz=PSL(2,\mathbb R)\times PSL(2,\mathbb
R)$ such that for each $i$ the composition
\[
    \pi_1(\Sigma_i)\rightarrow\pi_1(M)\rightarrow PSL(2,\mathbb
    R)\times PSL(2,\mathbb R)
 \]
 coincides with the holonomy of the left and right metric.
 \end{cor}   
\begin{proof}
Let $h_i^\pm:\pi_1(M)\rightarrow PSL(2,\mathbb R)$ be the holonomy of
$\mu_i^\pm$.  Since
\[
\alpha_1: \Sigma_1\setminus\Delta_{1_-}\rightarrow \Sigma_2\setminus\Delta_{2_+}
\]
is isotopic to an isometry, up to changing $h_2$ by conjugation we can suppose
\[
   h_1^\pm(\gamma)=h_2^\pm(\alpha_1(\gamma))
\]
 for every $\gamma\in\pi_1(\Sigma_1\setminus\Delta_{1_-}$.
 
 Inductively one shows that
 \[
   h_i^\pm(\gamma)=h_{i+1}^\pm(\alpha_i(\gamma))
 \]
 and thus by the presentation of the fundamental group of $M$ we get
%% that $h_i^\pm$'s glue together to a representation
 that the $h_i^\pm$'s glue together to a representation
 \[
    h^\pm:\pi_1(M)\rightarrow PSL(2,\mathbb R)\,.
\]      
\end{proof}

The following proposition summarizes the results of this section.
\begin{prop}\label{hol:prop}
There is a natural map
\[
   H:\mathcal D(X)\rightarrow \mathcal R
\]
that is an open injection. Moreover if $Y$ is the graph of collisions
of some $m$-spacetime $N$, then $H(Y)$ coincides with the holonomy of
$N$.
\end{prop}
\begin{proof}
The only point that remains to check is that $H$ is open. 
Take an admissible geometric data $\mu=(\mu_l(v),\mu_r(v))$ and denote by
$h=H(\mu)$.

Notice that the composition
\[
\begin{CD}
  \pi_1(\Sigma_v)@>>>\pi_1(M)@>>h> PSL_2(\mathbb
  R)\times PSL_2(\mathbb R)
\end{CD}
\]
produces the pair of holonomies of $\mu_l(v)$ and $\mu_r(v)$.

Since the set of holonomies of surfaces with cone points is open we have that
if $h'$ is sufficently close to $h$ in $\mathcal R$, then
for every vertex $v$ of $X$ the composition
\[
\begin{CD}
  \pi_1(\Sigma_v)@>>>\pi_1(M)@>>h'> PSL_2(\mathbb
  R)\times PSL_2(\mathbb R)
\end{CD}
\]
is a pair of holonomies of two hyperbolic metrics with cone angles 
(say $\mu_l'(v)$ and $\mu_r'(v)$) on $\Sigma_v$.

By Proposition \ref{vk:prop}, it is not difficult to check that metrics
$(\mu_l'(v),\mu_r'(v))$ form a geometric data $\mu'$ and $H(\mu')=h'$.
\end{proof}
\subsection{From the deformation of the holonomy 
to the deformation of the spacetime}

 In this section we prove that
small deformations of the holonomy $h$ of $M$ are holonomies for some
$m$-spacetime.  More precisely we prove the following theorem.
\begin{theorem}\label{struct:prop}
The holonomy map
\[
\Omega(g,T,\theta)\rightarrow \mathcal R
\]
is a local homeomorphism.
\end{theorem}

To prove this proposition we will use the following well-known fact
about $(G,X)$-structures on compact manifolds with boundary
\cite{epstein}.
\begin{lemma}\label{ep:lem}
Let $N$ be a smooth compact manifold with boundary and let $N'\subset
N$ be a submanifold such that $N\setminus N'$ is a collar of $N$.
\begin{itemize}
\item 
Given a $(X,G)$-structure $M$ on $N$ let $hol(M):\pi_1(N)\rightarrow G$
be the corresponding holonomy (that is defined up to conjugacy). 
Then, the holonomy map from the space of $(X,G)$-structures on $N$ to
  the space of representations of $\pi_1(N)$ into $G$ (up to conjugacy)
\[
    M\mapsto hol(M)
\]
is an open map.
\item
Let $M_0$ be an $(X,G)$-structure on $N$ and denote by $M_0'$ the
restriction of $M_0$ on $N'$.

There is a neighbourhood $\mathcal U$ of $M_0$ in the set of
$(X,G)$-structures on $N$ and a neighbourhood $\mathcal V$ of $M'_0$
in the set of $(X,G)$-structures on $N'$ such that for every structure
$M'\in\mathcal V$ there is a unique $M\in\mathcal U$ such that
$hol(M')=hol(M)$. Moreover, there is an embedding as $(X,G)$-manifolds
\[
    M'\hookrightarrow M.
\]
\end{itemize}
\end{lemma}    

First we prove Theorem \ref{struct:prop} assuming just one collision
in $M$.  Let $p_0$ be the collision point of $M$ and $S_0$ be the link
of $p_0$ in $M$ (that is a HS-surface).  Denote by
$G_0<\pi_1(\Sigma_g\times\mathbb R\setminus T)$ the fundamental group
of $S_0\setminus T$.

\begin{lemma}\label{bo:lem}
There is a neighbourhood $\mathcal U_0$ of $S_0$ in the space of
HS-surfaces homeomorphic to $S_0$ such that the holonomy map on
$\mathcal U_0$ is injective.

Moreover, there is a neighbourhood $\mathcal V$ of $h$ in $\mathcal R$ such
that for every $h'\in\mathcal V$ there is an HS-surface in 
%%$\mathcal U_0$ say $S(h)$ such that the holonomy of $S(h)$ is $h|_{G_0}$.
$\mathcal U_0$, say $S(h)$, such that the holonomy of $S(h)$ is $h|_{G_0}$.
\end{lemma}

\begin{proof}
Around each cone  point $q_i$ of $S_0$ take small disks
\[
   \Delta_1(i)\supset\Delta_2(i)
\]
Let now $S,S'$ be two HS-surfaces close to $S_0$ sharing the same
holonomy.  By Lemma \ref{ep:lem}, up to choosing $\mathcal U_0$
%%sufficiently small, there is an isometric emebdding
sufficiently small, there is an isometric embedding
\[
  f:(S\setminus\bigcup\Delta_2(i))\rightarrow S'.
\]
 Moreover, $\Delta_1(i)$ equipped with the structure induced by $S$
%% embeds in $S'$ (this beacuse the holonomy locally determines the
 embeds in $S'$ (this because the holonomy locally determines the
 singular points of HS surfaces). It is not difficult to see that
 such an inclusion coincides with $f$ on
 $\Delta_1(i)\setminus\Delta_2(i)$ (basically this depends on the fact
%% that an isometry of an hyperbolic annulus into a disk containing a
 that an isometry of a hyperbolic annulus into a disk containing a
 cone point is unique up to rotations).  Thus gluing such maps we
 obtain an isometry between $S$ and $S'$.

To prove the last part of the statement, let us consider for each cone
point a smaller disk $\Delta_3(i)\subset\Delta_2(i)$.  Let
$U=S_0\setminus\bigcup\Delta_3(i)$. Clearly we can find a
neighbourhood $\mathcal V$ of $h$ such that if $h'\in \mathcal V$ then
there is a structure $U'$ on $U$ close to he original one with holonomy
%%$h'$. On the other hand it is clear there exists a structure, say
$h'$. On the other hand it is clear that there exists a structure, say
$\Delta'_1(i)$ on $\Delta_1(i)$ with cone singularity with holonomy
given by $h'$ and close to the original structure.  By Lemma
\ref{ep:lem}, if $h'$ is sufficiently close to $h$, then
$\Delta_2(i)\setminus\Delta_3(i)$ equipped with the structure given by
$U'$ embeds in $\Delta'_1(i)$. Moreover $\partial \Delta_2(i)$ bounds
in $\Delta'_1(i)$ a disk $\Delta(i)$ containing the cone point. Thus
we can glue $\Delta_1(i)$'s to $U'$ obtained the HS surface with
holonomy $h'$.
\end{proof}

Consider now two spacelike surfaces $\Sigma_1,\Sigma_2$ in $M$
orthogonal to the singular locus that are disjoint and such that
$p_0\in I^+(\Sigma_1)\cap I^-(\Sigma_2)$.  Let $M_0= I^+(\Sigma_1)\cap
I^-(\Sigma_2)$. Clearly $\Sigma_1$ is the past boundary of $M$ and
$\Sigma_2$ is the future boundary.

Now let $U_1\supset U_2\supset U_3$ be regular neighborhoods of the
singular locus of $M_0$.  By Lemma~\ref{bo:lem} there is a
neighborhood $\mathcal V$ of $h$ in $\mathcal R$, such that for any
%%$h'\in\mathcal V$ there is a structure with collision say $U_1(h')$ on
$h'\in\mathcal V$ there is a structure with collision, say $U_1(h')$, on
$U_1$ close to the original one with holonomy $h$.

%%Moreover, there is a structure on $M_0\setminus U_3$, say $N_0(h')$
Moreover, there is a structure on $M_0\setminus U_3$, say $N_0(h')$,
with holonomy $h'$. By Lemma ~\ref{ep:lem}, $U_2\setminus U_3$ with
the metric induced by $N_0(h')$ can be embedded in
%%$U_1(h')$. Moreover, up to shrinking $N_0(h')$, we can suppose there
$U_1(h')$. Moreover, up to shrinking $N_0(h')$, we can suppose that there
%%exists an open set in $U_1(h')$ say $U$ such that
exists an open set in $U_1(h')$, say $U$, such that:
\begin{enumerate}
\item $U$ contains the collision point,
\item the frontier of $U$ is the union of the image of frontier of
  $U_2$ and spacelike disks orthogonal to the singular locus.
\end{enumerate}

The spacetime obtained by gluing $N_0(h')$ to $U$, by identifying
%%$U_2\setminus U_3$ with its image is a spacetime with collision with
$U_2\setminus U_3$ with its image, is a spacetime with collision with
holonomy $h'$. Its maximal extension, say $M(h')$, is an $m$-spacetime
with holonomy $h'$.

To conclude we have to prove that if $h'$ is sufficiently close to
$h$, then $M(h')$ is unique in a neighborhood of $M_0$.

In fact, it is sufficient to show that given an $m$-spacetime $M$ with
holonomy $h'$ close to $M_0$, there is a smaller spacetime, $N$,
containing the collision that embeds in $M(h')$.  We can
assume $N$ close to some spacetime $N_0$ contained in $M_0$ (this
%%precisley means that $N$ is obtained by deforming slightly the metric
precisely means that $N$ is obtained by deforming slightly the metric
on $N_0$).

By the uniqueness of the HS-surface with holonomy $h'$, we get that
%%small neighborhoods of the collision point of $U$ embeds in
small neighborhoods of the collision point of $U$ embed in
$M(h')$. Take neighborhoods $U_1,U_2$ of the singular locus as
before.  Notice that $U_1\cap N$ embeds in $M(h')$. On the other hand
$N\setminus U_2$ embeds in $M(h')$.  Moreover the embeddings $U_1\cap
N\hookrightarrow M(h')$ and $N\setminus U_2\hookrightarrow M(h')$
coincides on the intersection. So they be glued to an embedding
$N\hookrightarrow M(h')$.

This concludes the proof of Proposition \ref{struct:prop} when only
one interaction occurs.  The following lemma allows to conclude in the
general case by an inductive argument.

\begin{lemma}
%%Let $\Sigma$ be a good surface of $M$ and let $M_-$, $M_+$ be the past
Let $\Sigma$ be a surface of $M$ with a transverse vector field, and let $M_-$, $M_+$ be the past
and the future of $\Sigma$ in $M$.  Suppose that for a small
deformation $h'$ of the holonomy $h$ of $M$ there are two spacetimes with
collisions $M_-'\cong M_-$ and $M_+'\cong M_+$ such that the holonomy
%%of $M_\pm'$ is equal to $h'|_{\pi_1(M'_\pm)}$. Then there is a
of $M_\pm'$ is equal to $h'|_{\pi_1(M_\pm)}$. Then there is a
spacetime $M'$ close to $M$ containing both $M_-'$ and $M_+'$.
\end{lemma}

\begin{proof}
Let  $N(h)$ denote the maximal GH structure with
particles on $\Sigma\times\mathbb R$ whose  holonomy is $h|_{\pi_1(\Sigma)}$.
%%There is a neighborhood of $\Sigma$ in $M$ that embeds in $N(h)$.
There is a neighborhood of $\Sigma$ in $M$ which embeds in $N(h)$.
We can suppose that $\Sigma\subset M$ is sent to $\Sigma\times\{0\}$ 
through this embedding.

Now let $U_\pm$ be a collar of $\Sigma$ in $M_\pm$ such that the image of
$U_-$ in $N(h)$ is $\Sigma\times [-\epsilon, 0]$ and the image of
$U_+$ is $\Sigma\times[0,\epsilon]$ for some $\epsilon>0$.

 If $h'$ is sufficiently close to $h$, then there is an isometric
 embedding of $U_\pm$ (considered as subset of $M_\pm'$) into $N(h')$
 \[
    i_\pm: U_\pm\hookrightarrow N(h')
 \]   
 such that the image of $U_-$ is contained in $\Sigma\times
 [-2\epsilon, \epsilon/3]$ and contains $\Sigma\times\{-\epsilon/2\}$
 and the image of $U_+$ is contained in
 $\Sigma\times[-\epsilon/3,2\epsilon]$ and contains
%% $\Sigma\times\{\epsilon/2\}$. Thus we can glue $M_\pm $ and
 $\Sigma\times\{\epsilon/2\}$. Thus we can glue $M'_\pm $ and
 $\Sigma\times [-\epsilon/2,\epsilon/2]$ by identifying $p\in
 U_\pm\cap i_\pm^{-1}(\Sigma\times[-\epsilon/2,\epsilon/2])$ with its
%% image.  The spacetime we obtain, say $M'$ clearly contains $M'_-$ and
 image.  The spacetime we obtain, say $M'$, clearly contains $M'_-$ and
 $M'_+$.
 
 %About the uniqueness of $M'$, notice that given another structure say
 %$M''$,  any  regular neighborhood
 %of $\Sigma$ in $M'$ embeds into $M''$.  Moreover, we can choose $U$ such
 %that $M'\setminus U$ is a contained in $M_-'\cup M_+'$. Since
 %$M_\pm'$ embeds in $M''$, the embedding of $U$ into $M''$ extends on
 %the whole $M'$.
\end{proof}

\begin{remark}
To prove that there is a unique $m$-spacetime in a neigbourhood of $M$ with holonomy $h'$, we again use an inductive argument. Suppose we can find in any small neighborhood of $M$ two $m$-spacetimes $M'$ and $M''$ with holonomy $h'$. We fix a  surface with a transverse vector field
 $\Sigma$ in $M$ such that 
both the future and the past of $\Sigma$, say $M_\pm$, contain some collision points.
Let $U\subset V$ be regular neighbourhoods of $\Sigma$ in $M$ with spacelike boundaries.
We can consider neighbourhoods $U'\subset V'$ in $M'$ and $U''\subset V''$ in $M''$ such that
\begin{itemize}
%%\item They are close to $U$ and $V$ respectively;
\item they are close to $U$ and $V$ respectively,
\item they do not contain any collision,
\item they have spacelike boundary.
\end{itemize}
Applying the inductive hypothesis on the connected regions of the complement of $U'$ in $M'$ and $U''$ in $M'$ we have that for $h'$ sufficiently close to $h$ there is an isometric embedding
\[
    \psi:M'\setminus U'\rightarrow M''
\]    
such that $\psi(\partial U')$ is contained in $V''$.

Now consider the isometric embeddings
\[
   u':V'\rightarrow N(h')\qquad u'':V''\rightarrow N(h')
 \]
 notice that the maps $u'$ and $u''\circ\psi$ provide two isometric embeddings
 \[
     V'\setminus U'\rightarrow N(h')
 \]
 so they must coincide (we are using the fact that the inclusion of a GH spacetime with particles in its maximal extension is uniquely determined).
 
 Eventually we can extend $\psi$ on the whole $M'$ by setting on $V'$
 \[
    \psi= (u'')^{-1}\circ u'\,.
\]        
\end{remark}

\bibliography{bibliocollision}
\bibliographystyle{alpha}

\end{document}